\documentclass[reqno,11pt]{amsart}
\usepackage{amsmath}
\usepackage{amsxtra}
\usepackage{amscd}
\usepackage{amsthm}
\usepackage{amsfonts}
\usepackage{amssymb}
\usepackage{eucal}
\usepackage{scalerel}
\usepackage{verbatim}
\usepackage[matrix,arrow,curve]{xy}

\usepackage{a4wide}

\usepackage[T1]{fontenc}

\usepackage{xr-hyper}
\usepackage[
pdftex,
bookmarks=false,
colorlinks=true,
debug=true,
pdfnewwindow=true]{hyperref}

\theoremstyle{plain}
\newtheorem{Thm}[equation]{Theorem}
\newtheorem{Cor}[equation]{Corollary}
\newtheorem{Lem}[equation]{Lemma}
\newtheorem{Prop}[equation]{Proposition}
\newtheorem{Conj}[equation]{Conjecture}
\newtheorem{Ex}[equation]{Example}

\theoremstyle{definition}
\newtheorem{Def}[equation]{Definition}

\theoremstyle{remark}

\newtheorem{Rem}[equation]{Remark}

\newtheorem{conj}{Conjecture}

\numberwithin{equation}{section}
\renewcommand{\rm}{\normalshape}

\newif\ifShowLabels
\ShowLabelstrue
\newdimen\theight
\def\TeXref#1{%
    \leavevmode\vadjust{\setbox0=\hbox{{\tt
        \quad\quad  {\small \rm #1}}}%
    \theight=\ht0
    \advance\theight by \lineskip
    \kern -\theight \vbox to
    \theight{\rightline{\rlap{\box0}}%
    \vss}%
    }}%

\ShowLabelsfalse

\newenvironment{thm}[1]%
    { \begin{Thm} \label{T:#1}  \ifShowLabels \TeXref{T:#1} \fi }%
    { \end{Thm} }

\renewcommand{\th}[1]{\begin{thm}{#1} \sl }
\renewcommand{\eth}{\end{thm} }

\newenvironment{lemma}[1]%
    { \begin{Lem} \label{L:#1}  \ifShowLabels \TeXref{L:#1} \fi }%
    { \end{Lem} }
\newcommand{\lem}[1]{\begin{lemma}{#1} \sl}
\newcommand{\elem}{\end{lemma}}

\newenvironment{propos}[1]%
    { \begin{Prop} \label{P:#1}  \ifShowLabels \TeXref{P:#1} \fi }%
    { \end{Prop} }
\newcommand{\prop}[1]{\begin{propos}{#1}\sl }
\newcommand{\eprop}{\end{propos}}

\newenvironment{corol}[1]%
    { \begin{Cor} \label{C:#1}  \ifShowLabels \TeXref{C:#1} \fi }%
    { \end{Cor} }
\newcommand{\cor}[1]{\begin{corol}{#1} \sl }
\newcommand{\ecor}{\end{corol}}

\newenvironment{defeni}[1]%
    { \begin{Def} \label{D:#1}  \ifShowLabels \TeXref{D:#1} \fi }%
    { \end{Def} }
\newcommand{\defe}[1]{\begin{defeni}{#1} \sl }
\newcommand{\edefe}{\end{defeni}}

\newenvironment{remark}[1]%
    { \begin{Rem} \label{R:#1}  \ifShowLabels \TeXref{R:#1} \fi }%
    { \end{Rem} }
\newcommand{\rem}[1]{\begin{remark}{#1}}
\newcommand{\erem}{\end{remark}}

\newenvironment{conjec}[1]%
    { \begin{Conj} \label{Co:#1}  \ifShowLabels \TeXref{Co:#1} \fi }%
    { \end{Conj} }
\renewcommand{\conj}[1]{\begin{conjec}{#1} \sl }
\newcommand{\econj}{\end{conjec}}

\newcommand{\eq}[1]%
    { \ifShowLabels \TeXref{E:#1} \fi
       \begin{equation} \label{E:#1} }
\newcommand{\eeq}{ \end{equation} }

\newcommand{\prf}{ \begin{proof} }
\newcommand{\epr}{ \end{proof} }

\usepackage[dvipsnames]{xcolor}

\newcommand\nc{\newcommand}
\nc{\unl}{\underline}
\nc{\ol}{\overline}
\nc{\on}{\operatorname}

\nc{\BA}{{\mathbb{A}}}
\nc{\BC}{{\mathbb{C}}}
\nc{\BD}{{\mathbb{D}}}
\nc{\BF}{{\mathbb{F}}}
\nc{\BG}{{\mathbb{G}}}
\nc{\BM}{{\mathbb{M}}}
\nc{\BN}{{\mathbb{N}}}
\nc{\BO}{{\mathbb{O}}}
\nc{\BQ}{{\mathbb{Q}}}
\nc{\BP}{{\mathbb{P}}}
\nc{\BR}{{\mathbb{R}}}
\nc{\BZ}{{\mathbb{Z}}}
\nc{\BS}{{\mathbb{S}}}
\nc{\BK}{{\mathbb{K}}}

\nc{\CA}{{\mathcal{A}}} \nc{\CB}{{\mathcal{B}}} \nc{\CalC}{{\mathcal
C}} \nc{\CalD}{{\mathcal D}} \nc{\CE}{{\mathcal{E}}}
\nc{\CF}{{\mathcal{F}}} \nc{\CG}{{\mathcal{G}}}
\nc{\CH}{{\mathcal{H}}} \nc{\CI}{{\mathcal{I}}}
\nc{\CK}{{\mathcal{K}}} \nc{\CL}{{\mathcal{L}}}
\nc{\CM}{{\mathcal{M}}} \nc{\CN}{{\mathcal{N}}}
\nc{\CO}{{\mathcal{O}}} \nc{\CP}{{\mathcal{P}}}
\nc{\CQ}{{\mathcal{Q}}} \nc{\CR}{{\mathcal{R}}}
\nc{\CS}{{\mathcal{S}}} \nc{\CT}{{\mathcal{T}}}
\nc{\CU}{{\mathcal{U}}} \nc{\CV}{{\mathcal{V}}}
\nc{\CW}{{\mathcal{W}}} \nc{\CX}{{\mathcal{X}}}
\nc{\CY}{{\mathcal{Y}}} \nc{\CZ}{{\mathcal{Z}}}

\nc{\fa}{{\mathfrak{a}}}
\nc{\fb}{{\mathfrak{b}}}
\nc{\fg}{{\mathfrak{g}}}
\nc{\fgl}{{\mathfrak{gl}}}
\nc{\fh}{{\mathfrak{h}}}
\nc{\fj}{{\mathfrak{j}}}
\nc{\fl}{{\mathfrak{l}}}
\nc{\fm}{{\mathfrak{m}}}
\nc{\fn}{{\mathfrak{n}}}
\nc{\fu}{{\mathfrak{u}}}
\nc{\fp}{{\mathfrak{p}}}
\nc{\frr}{{\mathfrak{r}}}
\nc{\fs}{{\mathfrak{s}}}
\nc{\ft}{{\mathfrak{t}}}
\nc{\fw}{{\mathfrak{w}}}
\nc{\fz}{{\mathfrak{z}}}

\nc{\fA}{{\mathfrak{A}}}
\nc{\fB}{{\mathfrak{B}}}
\nc{\fD}{{\mathfrak{D}}}
\nc{\fE}{{\mathfrak{E}}}
\nc{\fF}{{\mathfrak{F}}}
\nc{\fG}{{\mathfrak{G}}}
\nc{\fI}{{\mathfrak{I}}}
\nc{\fJ}{{\mathfrak{J}}}
\nc{\fK}{{\mathfrak{K}}}
\nc{\fL}{{\mathfrak{L}}}
\nc{\fM}{{\mathfrak{M}}}
\nc{\fN}{{\mathfrak{N}}}
\nc{\frP}{{\mathfrak{P}}}
\nc{\fQ}{{\mathfrak Q}}
\nc{\fR}{{\mathfrak R}}
\nc{\fS}{{\mathfrak S}}
\nc{\fT}{{\mathfrak{T}}}
\nc{\fU}{{\mathfrak{U}}}
\nc{\fW}{{\mathfrak{W}}}
\nc{\fY}{{\mathfrak{Y}}}
\nc{\fZ}{{\mathfrak{Z}}}

\nc{\ba}{{\mathbf{a}}}
\nc{\bb}{{\mathbf{b}}}
\nc{\bc}{{\mathbf{c}}}
\nc{\bd}{{\mathbf{d}}}
\nc{\be}{{\mathbf{e}}}
\nc{\bi}{{\mathbf{i}}}
\nc{\bj}{{\mathbf{j}}}
\nc{\bn}{{\mathbf{n}}}
\nc{\bp}{{\mathbf{p}}}
\nc{\bq}{{\mathbf{q}}}
\nc{\bu}{{\mathbf{u}}}
\nc{\bv}{{\mathbf{v}}}
\nc{\bw}{{\mathbf{w}}}
\nc{\bx}{{\mathbf{x}}}
\nc{\by}{{\mathbf{y}}}
\nc{\bz}{{\mathbf{z}}}

\nc{\bA}{{\mathbf{A}}}
\nc{\bB}{{\mathbf{B}}}
\nc{\bC}{{\mathbf{C}}}
\nc{\bD}{{\mathbf{D}}}
\nc{\bE}{{\mathbf{E}}}
\nc{\bI}{{\mathbf{I}}}
\nc{\bK}{{\mathbf{K}}}
\nc{\bH}{{\mathbf{H}}}
\nc{\bM}{{\mathbf{M}}}
\nc{\bN}{{\mathbf{N}}}
\nc{\bO}{{\mathbf{O}}}
\nc{\bQ}{{\mathbf Q}}
\nc{\bS}{{\mathbf{S}}}
\nc{\bT}{{\mathbf{T}}}
\nc{\bV}{{\mathbf{V}}}
\nc{\bW}{{\mathbf{W}}}
\nc{\bX}{{\mathbf{X}}}
\nc{\bP}{{\mathbf{P}}}
\nc{\bY}{{\mathbf{Y}}}
\nc{\bZ}{{\mathbf{Z}}}

\nc{\sA}{{\mathsf{A}}}
\nc{\sB}{{\mathsf{B}}}
\nc{\sC}{{\mathsf{C}}}
\nc{\sD}{{\mathsf{D}}}
\nc{\sF}{{\mathsf{F}}}
\nc{\sK}{{\mathsf{K}}}
\nc{\sM}{{\mathsf{M}}}
\nc{\sO}{{\mathsf{O}}}
\nc{\sQ}{{\mathsf{Q}}}
\nc{\sP}{{\mathsf{P}}}
\nc{\sT}{{\mathsf{T}}}
\nc{\sV}{{\mathsf{V}}}
\nc{\sW}{{\mathsf{W}}}
\nc{\sX}{{\mathsf{X}}}
\nc{\sZ}{{\mathsf{Z}}}
\nc{\sU}{{\mathsf{U}}}
\nc{\sS}{{\mathsf{S}}}

\nc{\sfb}{{\mathsf{b}}}
\nc{\sfc}{{\mathsf{c}}}
\nc{\sd}{{\mathsf{d}}}
\nc{\sg}{{\mathsf{g}}}
\nc{\sk}{{\mathsf{k}}}
\nc{\sfl}{{\mathsf{l}}}
\nc{\sfp}{{\mathsf{p}}}
\nc{\sr}{{\mathsf{r}}}
\nc{\st}{{\mathsf{t}}}
\nc{\sfu}{{\mathsf{u}}}
\nc{\sw}{{\mathsf{w}}}
\nc{\sz}{{\mathsf{z}}}
\nc{\sx}{{\mathsf{x}}}
\nc{\se}{{\mathsf{e}}}
\nc{\sfv}{{\mathsf{v}}}

\nc{\bLambda}{{\boldsymbol{\Lambda}}}
\nc{\vv}{{\boldsymbol{v}}}
\nc{\Fl}{{{\mathcal F}\ell}}
\nc{\Gr}{{\on{Gr}}}
\nc{\CHH}{{\CH\!\!\CH}}
\nc{\lambdavee}{{\lambda^{\!\scriptscriptstyle\vee}}}
\nc{\alphavee}{\alpha^{\!\scriptscriptstyle\vee}}
\nc{\rhovee}{{\rho^{\!\scriptscriptstyle\vee}}}
\newcommand\iso{\,\vphantom{j^{X^2}}\smash{\overset{\sim}{\vphantom{\rule{0pt}{0.20em}}\smash{\longrightarrow}}}\,}
\nc{\oQM}{\vphantom{j^{X^2}}\smash{\overset{\circ}{\vphantom{\vstretch{0.7}{A}}\smash{\QM}}}}
\nc{\oZ}{{}^\dagger\!\vphantom{j^{X^2}}\smash{\overset{\circ}{\vphantom{\vstretch{0.7}{A}}\smash{Z}}}}
\nc{\odZ}{{}^\dagger\!\vphantom{j^{X^2}}\smash{\overset{\circ}{\vphantom{\vstretch{0.7}{A}}\smash{\mathfrak Z}}}^{c',c}}
\nc{\bdZ}{{}^\dagger\!\vphantom{j^{X^2}}\smash{\overset{\bullet}{\vphantom{\vstretch{0.7}{A}}\smash{\mathfrak Z}}}^{c',c}}
\nc{\oS}{\vphantom{j^{X^2}}\smash{\overset{\circ}{\vphantom{\vstretch{0.7}{A}}\smash{S}}}}
\nc{\buM}{\vphantom{j^{X^2}}\smash{\overset{\bullet}{\vphantom{\vstretch{0.7}{A}}\smash{M}}}}
\nc{\dW}{{}^\dagger\ol\CW{}}
\nc{\hW}{{}^\dagger\hat\CW{}}
\nc{\wW}{{}^\dagger\widetilde\CW{}}
\nc{\dZ}{{}^\dagger\!\fZ^{c',c}}
\nc{\dZc}{{}^\dagger\!\fZ^{c,c}}
\nc{\tZ}{{}^\dagger\!\tilde{Z}{}}
\nc{\hZ}{{}^\dagger\!\hat{Z}{}}

\nc{\ssl}{\mathfrak{sl}} \nc{\gl}{\mathfrak{gl}}
\nc{\wt}{\widetilde} \nc{\Sym}{\mathrm{Sym}} \nc{\Res}{\mathrm{Res}}
\nc{\sE}{{\mathsf{E}}} \nc{\bs}{{\mathbf{s}}}
\nc{\trig}{\mathrm{trig}} \nc{\rat}{\mathrm{rat}}
\nc{\sign}{\mathrm{sign}} \nc{\sL}{{\mathsf{L}}}
\nc{\fv}{{\mathfrak{v}}} \nc{\ad}{\mathrm{ad}}
\nc{\spsi}{{\mathsf{\psi}}} \nc{\sh}{{\mathsf{h}}}
\nc{\rtt}{\mathrm{rtt}} \nc{\qdet}{\mathrm{qdet}} \nc{\pt}{{\operatorname{pt}}}
\nc{\M}{\mathrm{M}} \nc{\Ker}{\mathrm{Ker}} \nc{\ssc}{\mathrm{sc}}
\nc{\loc}{\mathrm{loc}} \nc{\fra}{\mathrm{frac}}
\nc{\ddj}{\mathrm{DJ}} \nc{\End}{\mathrm{End}} \nc{\ev}{\mathrm{ev}}
\nc{\GL}{\mathrm{GL}}
\nc{\SSym}{\mathrm{SSym}}

\begin{document}
\title[PBWD bases and shuffle realizations]
{PBWD bases and shuffle algebra realizations for
$U_\vv(L\ssl_n), U_{\vv_1,\vv_2}(L\ssl_n), U_\vv(L\ssl(m|n))$
and their integral forms}

\author[Alexander Tsymbaliuk]{Alexander Tsymbaliuk}
\address{A.T.: Purdue University, Department of Mathematics, West Lafayette, IN 47907, USA}
\email{sashikts@gmail.com}

\begin{abstract}
We construct a family of PBWD (Poincar\'e-Birkhoff-Witt-Drinfeld) bases for
the quantum loop algebras
  $U_\vv(L\ssl_n),U_{\vv_1,\vv_2}(L\ssl_n),U_\vv(L\ssl(m|n))$
in the new Drinfeld realizations. In the 2-parameter case, this
proves~\cite[Theorem 3.11]{hrz} (stated in~\emph{loc.~cit.} without a proof),
while in the super case it proves a conjecture of~\cite{z1}.
The main ingredient in our proofs is the interplay between those
quantum loop algebras and the corresponding shuffle algebras,
which are trigonometric counterparts of the elliptic shuffle algebras
of~\cite{fo1}--\cite{fo3}. Our approach is similar to that of~\cite{e}
in the formal setting, but the key novelty is an explicit shuffle algebra
realization of the corresponding algebras, which is of independent interest.
This also allows us to strengthen the above results by constructing a family
of PBWD bases for the RTT forms of those quantum loop algebras
as well as for the Lusztig form of $U_\vv(L\ssl_n)$.
The rational counterparts provide shuffle algebra realizations of
type $A$ (super) Yangians and their Drinfeld-Gavarini dual subalgebras.
\end{abstract}
\maketitle


\section{Introduction}


\subsection{Summary}
\

The quantum loop algebras (aka quantum affine algebras with the trivial central charge)
associated to a simple finite dimensional Lie algebra $\fg$ admit two well-known presentations:
the original Drinfeld-Jimbo realization $U^{\ddj}_\vv(L\fg)$ and the new Drinfeld (aka loop)
realization $U_\vv(L\fg)$. The latter presentation (which is essential to study the representation
theory of quantum loop algebras) was introduced by V.~Drinfeld in~\cite{d}, while the explicit isomorphism
\begin{equation}
\label{drinfeld iso}
  U^{\ddj}_\vv(L\fg)\iso U_\vv(L\fg)
\end{equation}
was stated without a proof in~\cite[Theorem 3]{d}. Actually,~(\ref{drinfeld iso}) was upgraded in~\emph{loc.~cit}.\
to the isomorphism of the corresponding quantum affine algebras (with nontrivial central charges)
\begin{equation}
\label{drinfeld iso aff}
  U^{\ddj}_\vv(\widehat{\fg})\iso U_\vv(\widehat{\fg})
\end{equation}
The proof of the isomorphism~(\ref{drinfeld iso aff}) (hence, also of~(\ref{drinfeld iso})) was properly
established in the works of J.~Beck~\cite{b2}, I.~Damiani~\cite{da}, and N.~Jing~\cite{j}. Let us note that~\cite{b2,da}
actually constructed the isomorphism opposite way $U_\vv(\widehat{\fg}) \iso U^{\ddj}_\vv(\widehat{\fg})$
by utilizing Lusztig's affine braid group action on $U^{\ddj}_\vv(\widehat{\fg})$, which is precisely
the inverse of~(\ref{drinfeld iso}), as shown in~\cite[Remark of~\S4]{b2}.

\medskip
\noindent
Since quantum loop algebras are natural quantizations of the universal enveloping of the loop Lie
algebras $L\fg=\fg[t,t^{-1}]$, one of the first natural tasks is to seek an analogue of
PBW bases for the former algebras. This was accomplished more than 25 years ago by J.~Beck~\cite{b1} in the
context of $U^{\ddj}_\vv(\widehat{\fg})$. More precisely, he constructed the bases of each of the
subalgebras featuring in the triangular decomposition (viewed as a vector space isomorphism)
\begin{equation}
\label{triang DJ}
  U^{\ddj}_\vv(\widehat{\fg})\simeq U^{\ddj,>}_\vv(\widehat{\fg})\otimes U^{\ddj,0}_\vv(\widehat{\fg})\otimes U^{\ddj,<}_\vv(\widehat{\fg}).
\end{equation}
We note that the construction of~\cite{b1} actually depends on a choice of an element $x\in P^\vee$
of the coweight lattice, which pairs positively with all simple roots of $\fg$, together with a choice
of a reduced decomposition of $(1,x)\in W\ltimes P^\vee\simeq \widehat{W}^{ext}$ in the extended affine Weyl group.

\medskip
\noindent
The algebra $U_\vv(\widehat{\fg})$ also admits a triangular decomposition, i.e.\ a vector space isomorphism
\begin{equation}
\label{triang Dr-loop}
  U_\vv(\widehat{\fg})\simeq U^>_\vv(\widehat{\fg})\otimes U^0_\vv(\widehat{\fg})\otimes U^<_\vv(\widehat{\fg}).
\end{equation}
However, the isomorphism~(\ref{drinfeld iso aff}) \underline{does not} intertwine the triangular
decompositions~(\ref{triang DJ},~\ref{triang Dr-loop}).
Therefore, it is desirable to construct PBWD bases (the letter ``D'' after ``PBW'' is
to indicate the new Drinfeld realization) of $U_\vv(\widehat{\fg})$, compatible with the
triangular decomposition~(\ref{triang Dr-loop}). As $U^>_\vv(\widehat{\fg})\simeq U^>_\vv(L\fg)$
is actually isomorphic to $U^<_\vv(\widehat{\fg})\simeq U^<_\vv(L\fg)$, the above boils down to:

\medskip
\noindent
\textbf{Problem:} Construct PBWD bases of the \emph{positive} subalgebras $U^>_\vv(L\fg)$.

\medskip
\noindent
To our surprise, this question seems to remain open.
The only case we found addressed in the literature is the type $A$ quantum loop
algebras and their two-parameter generalizations $U_{\vv_1,\vv_2}(L\ssl_n)$, see~\cite[Theorem~3.11]{hrz}.
However, the proof of that theorem is missing in~\emph{loc.~cit}. This gap has been
also noticed in~\cite{z1,z2}, where a weak version of the PBW
property has been established for the quantum loop superalgebra $U_\vv(L\ssl(m|n))$
of~\cite{y} by straightforward lengthy arguments.

\medskip
\noindent
One objective of this paper is to fill in the above gap by constructing a family
of PBWD bases for the aforementioned quantum loop algebras
  $U_\vv(L\ssl_n),U_{\vv_1,\vv_2}(L\ssl_n),U_\vv(L\ssl(m|n))$,
which we further refine by constructing a family of PBWD bases for their certain integral forms
  $\fU_\vv(L\ssl_n), \fU_{\vv_1,\vv_2}(L\ssl_n), \fU_\vv(L\ssl(m|n))$, 
defined over $\BC[\vv,\vv^{-1}]$ and $\BC[\vv_1,\vv_2,\vv^{-1}_1,\vv^{-1}_2]$, 
respectively.\footnote{It should be noted right away that these forms can be defined over 
$\BZ[\vv,\vv^{-1}]$ and $\BZ[\vv_1,\vv_2,\vv^{-1}_1,\vv^{-1}_2]$, respectively, and all 
our results for the integral forms generalize verbatim to this setting as well.}
This is accomplished by providing the shuffle realization of their
positive subalgebras (following the ideas of~\cite{fo1}--\cite{fo3} and~\cite{e}),
which constitutes another main result of our paper. It should be noted that
the corresponding shuffle realization of $U^>_\vv(L\ssl_n)$ can be
implicitly deduced from~\cite{n}, but we provide an alternative simpler
proof which also works for the other two algebras
$U^>_{\vv_1,\vv_2}(L\ssl_n),U^>_\vv(L\ssl(m|n))$ as well as for their integral forms
$\fU^>_\vv(L\ssl_n),\fU^>_{\vv_1,\vv_2}(L\ssl_n),\fU^>_\vv(L\ssl(m|n))$.

\medskip
\noindent
The aforementioned integral forms are not completely new in the literature.
Indeed, the $\gl_n$-counterpart of $\fU_\vv(L\ssl_n)$, the integral form $\fU_\vv(L\gl_n)$ of
the quantum loop algebra $U_\vv(L\gl_n)$, appeared recently in~\cite{ft2} where it was used
to construct and study integral forms of type $A$ shifted quantum loop algebras.
As shown in~\emph{loc.~cit.}, $\fU_\vv(L\gl_n)$ coincides with the tautological integral form
(that is, the $\BC[\vv,\vv^{-1}]$-subalgebra generated by the same collection of generators)
of $U^{\rtt}_\vv(L\gl_n)$, the RTT realization of the quantum loop $\gl_n$. The RTT approach to
quantum groups goes back to the St.~Petersburg school of L.~Faddeev, see~\cite{frt}, while the isomorphism
\begin{equation}
\label{ding-frenkel}
  U^{\rtt}_\vv(L\gl_n)\simeq U_\vv(L\gl_n)
\end{equation}
is due to~\cite{df}, where it was actually upgraded to quantum affine algebras:
  $U^{\rtt}_\vv(\widehat{\gl}_n)\simeq U_\vv(\widehat{\gl}_n)$.
The isomorphism~(\ref{ding-frenkel}) admits natural generalizations to the two-parameter and super-cases:
\begin{equation}
\label{ding-frenkel generalization}
\begin{split}
  & U^{\rtt}_{\vv_1,\vv_2}(L\gl_n)\simeq U_{\vv_1,\vv_2}(L\gl_n) \\
  & U^{\rtt}_\vv(L\gl(m|n))\simeq U_\vv(L\gl(m|n))
\end{split}
\end{equation}
see~\cite{jl,z3} and the references therein. Thus, the tautological integral forms
of the algebras in the left-hand side of~(\ref{ding-frenkel generalization}) give
rise to integral forms of the algebras in the right-hand side, which can be perceived
as $\gl_n, \gl(m|n)$-counterparts of our integral forms $\fU_{\vv_1,\vv_2}(L\ssl_n), \fU_\vv(L\ssl(m|n))$.

\medskip
\noindent
Let us point out right away both the similarities and the differences
between the current work and a much older paper~\cite{e} of B.~Enriquez.
In~\cite{e}, the author established similar results for the quantum
loop algebras in the formal setting, that is, when working over $\BC[[\hbar]]$
rather than over $\BC(\vv)$. In particular, the PBW theorem of~\cite[Theorem 1.3]{e}
is proved using an embedding of $U^>_\hbar(L\fg)$ into the corresponding type $\fg$
shuffle algebra $S^{(\fg)}$~\cite[Corollary~1.4]{e} with the image
$\bar{S}^{(\fg)}\subset S^{(\fg)}$ being the subalgebra generated by
degree $1$ components. In type $A$, this coincides with our Proposition~\ref{simple shuffle}.
However, the heart of our shuffle algebra realization is the proof of the equality
$\bar{S}^{(\fg)}=S^{(\fg)}$, at least, for $\fg=\ssl_n$ (and similarly for $\fg=\ssl(m|n)$).
This implies the (corrected) description of $\bar{S}^{(\fg)}$ conjectured
in~\cite[Remark 3.16]{e}.

\medskip
\noindent
We expect that similar arguments shall provide PBWD bases for $U_\vv(L\fg)$
as well as establish their shuffle realizations, at least for simply-laced
simple $\fg$, which will be discussed elsewhere. In contrast, the PBWD theorem
for the Yangian $Y(\fg)$ of any semisimple Lie algebra $\fg$ has been proved
long time ago in~\cite{l}.

\medskip
\noindent
A particular PBWD basis of the integral form $\fU_\vv(L\ssl_n)$
was used in~\cite{ft2} to define an integral form of type $A$
shifted quantum loop algebras of~\cite{ft}, see Remark~\ref{ft's choice}
and Theorem~\ref{Full PBWD for integral-special choice}. Furthermore, an important
family of elements of the latter form, which were crucially used in~\cite[Proof of Theorem 4.15]{ft2},
appear manifestly via their shuffle realizations~(\ref{for surjectivity in FT2}).

\medskip
\noindent
Another particular PBWD basis of $\fU^>_\vv(L\ssl_n)$ is very similar
to the one arising from the results of~\cite{n} by viewing  $U_\vv(L\ssl_n)$
as a ``vertical'' subalgebra of the quantum toroidal algebra
$U_{\vv,\bar{\vv}}(\ddot{\gl}_n)$, see Remark~\ref{Negut's pbw}.

\medskip
\noindent
Finally, let us make a few general comments about the PBWD bases
constructed in this paper. As was pointed out to us by P.~Etingof,
the linear independence of the ordered monomials~(\ref{ordered}),
which is established in Section~\ref{ssec linear indep}, can be
immediately deduced by using the PBW property of $U(\ssl_n[t,t^{-1}])$
as well as flatness of the deformation, cf.~\cite[Theorem 1.3]{e}.
Nevertheless, we provide technical details as they are needed both
for Section~\ref{ssec spanning prop} and for the generalizations to
the two-parameter and super cases. At that point, we should note that while
the two-parameter quantum affine algebras have been extensively studied
since the original work~\cite{hrz}, see~\cite{jl,jz1,jz2} for a partial list
of references (see also~\cite{bw1,bw2,jmy,t} for the case of two-parameter
quantum finite groups), not many results have been established for them.
In particular, it is still an open question whether these are flat deformations
of the corresponding universal enveloping algebras (the results of the current paper
give an affirmative answer to this question in type $A$). In~\cite{jz2}, an isomorphism
between the Drinfeld-Jimbo and the new Drinfeld realizations of these algebras
was established (generalizing~\cite[Theorem 3.12]{hrz} for type $A$), but
it is not known (at the moment) whether the former realization admits the PBW basis
analogous the one of~\cite{b1,b2}.


\subsection{Outline of the paper}
\

\noindent
$\bullet$
In Section~\ref{ssec affine sl_n}, we recall the new Drinfeld realization
of the quantum loop algebra $U_\vv(L\ssl_n)$.
In Proposition~\ref{Triangular decomposition}, we invoke its triangular decomposition
as well as explicit descriptions of its positive, negative, and Cartan subalgebras
$U^>_\vv(L\ssl_n), U^<_\vv(L\ssl_n), U^0_\vv(L\ssl_n)$ established in~\cite{he}.

\medskip
\noindent
In Section~\ref{ssec formulation main thm 1}, we introduce the
\emph{PBWD basis elements} $e_\beta(r),f_\beta(r)$ of~(\ref{higher roots}),
which do depend on certain choices (see 1)--3) prior to~(\ref{higher roots})).
Having picked a specific order~(\ref{extended order}) on $\Delta^+\times \BZ$,
we use those elements to construct the \emph{ordered PBWD monomials}
$e_h,f_h$ of~(\ref{ordered}). This provides a family of the PBWD bases
for $U^<_\vv(L\ssl_n),U^>_\vv(L\ssl_n)$ and $U_\vv(L\ssl_n)$, see
Theorems~\ref{Main Theorem 1} and~\ref{Main Theorem 1 entire algebra}
for the formulation of these results, while their proofs are deferred to
Section~\ref{sec proofs of Theorems 1,2}.

\medskip
\noindent
In Section~\ref{ssec formulation main thm 2}, we introduce integral forms
$\fU^>_\vv(L\ssl_n),\fU^<_\vv(L\ssl_n)$ as the $\BC[\vv,\vv^{-1}]$-subalgebras
generated by $\wt{e}_\beta(r),\wt{f}_\beta(r)$ of~(\ref{integral basis PBW}).
While these elements do depend on the choices 1)--3) made prior to~(\ref{higher roots}),
we prove that the forms $\fU^>_\vv(L\ssl_n),\fU^<_\vv(L\ssl_n)$ are independent
of these choices and posses PBWD bases (over $\BC[\vv,\vv^{-1}]$), see
Theorem~\ref{Main Theorem 2} (which is proved in Section~\ref{sec proofs of Theorems 1,2}).
Following Remark~\ref{ft's RTT interpretation}, the integral form
$\fU_\vv(L\ssl_n)$ of the entire $U_\vv(L\ssl_n)$ introduced in
Definition~\ref{definition entire integral form} is identified with
the RTT integral form $\fU^\rtt_\vv(L\ssl_n)$ of~\cite{frt}, which is
used in~\cite{ft2} to establish Theorem~\ref{Full PBWD for integral-special choice}.
The latter implies the triangular decomposition for $\fU_\vv(L\ssl_n)$
as well as provides a whole family of the PBWD bases for it, see
Corollary~\ref{Triangular for integral form}, Theorem~\ref{Full PBWD for integral}.

\medskip
\noindent
$\bullet$
In Section~\ref{ssec usual shuffle algebra}, we introduce the shuffle
algebra $S^{(n)}$, which may be viewed as a trigonometric degeneration
of the elliptic shuffle algebra of Feigin-Odesskii, see~\cite{fo1}--\cite{fo3}.
An algebra embedding $\Psi\colon U^>_\vv(L\ssl_n)\hookrightarrow S^{(n)}$ of
Proposition~\ref{simple shuffle} is a ``simple version'' of the shuffle realization
of $U^>_\vv(L\ssl_n)$ (which was first used in~\cite{e} in the formal setting).
The ``hard version'' of the shuffle algebra realization, Theorem~\ref{hard shuffle},
establishes that $\Psi$ is an algebra isomorphism.
We conclude this section with a construction of the specialization maps
$\phi_{\unl{d}}$~(\ref{specialization map}) which constitute the key
tool in our study of the shuffle algebra $S^{(n)}$ and the homomorphism $\Psi$.

\medskip
\noindent
In Section~\ref{ssec proof of Theorem 1}, we prove simultaneously
Theorems~\ref{Main Theorem 1} and~\ref{hard shuffle} by combining
the key properties of the specialization maps $\phi_{\unl{d}}$ established
in Lemmas~\ref{lower degrees},~\ref{same degrees},~\ref{spanning} with
a direct treatment (crucially based on the formula~(\ref{factorial formula}))
of the simplest ``rank $1$'' case $n=2$ in Lemma~\ref{n=1 case}.

\medskip
\noindent
In Section~\ref{ssec integral shuffle algebra}, we provide an explicit
description of the image $\fS^{(n)}=\Psi(\fU^>_\vv(L\ssl_n))\subset S^{(n)}$,
see Theorem~\ref{shuffle integral form} and Definition~\ref{integral element}.
While this description of $\fS^{(n)}$ is rather cumbersome (in particular,
it is not even obvious that it is a $\BC[\vv,\vv^{-1}]$-subalgebra of $S^{(n)}$),
we can still establish important properties for it, see
Proposition~\ref{results for FT2}, which play the crucial role in~\cite{ft2}.

\medskip
\noindent
In Section~\ref{ssec proof of Theorem 2}, we prove simultaneously
Theorems~\ref{Main Theorem 2} and~\ref{shuffle integral form} by first treating
$n=2$ case in Lemmas~\ref{necessity for n=2},~\ref{integral n=1 case}
and then following arguments of Section~\ref{ssec proof of Theorem 1}
to treat the general case.

\medskip
\noindent
$\bullet$
In Section~\ref{sec 2-parametric quantuma affine}, we generalize the key
results of Sections~\ref{sec Classical quantuma affine}--\ref{sec proofs of Theorems 1,2}
to the two-parameter quantum loop algebras $U^>_{\vv_1,\vv_2}(L\ssl_n)$ of~\cite{hrz}
recalled in Section~\ref{ssec 2-parameter quantum}. We construct a family
of the PBWD bases for $U^>_{\vv_1,\vv_2}(L\ssl_n)$ in Theorem~\ref{Main Theorem 3},
thus generalizing~\cite[Theorem 3.11]{hrz} presented in~\emph{loc.~cit.}\
without a proof, see Remark~\ref{rosso's choice}. We further strengthen
this by constructing a family of the PBWD bases for the integral form $\fU^>_{\vv_1,\vv_2}(L\ssl_n)$,
see Theorem~\ref{Main Theorem 3.1}. Finally, we provide the shuffle algebra realization
of $U^>_{\vv_1,\vv_2}(L\ssl_n)$ in Theorem~\ref{hard shuffle 2-parametric}.

\medskip
\noindent
$\bullet$
In Section~\ref{sec super-Lie quantuma affine}, we generalize the key results
of Sections~\ref{sec Classical quantuma affine}--\ref{sec proofs of Theorems 1,2}
to the quantum loop superalgebra  $U^>_\vv(L\ssl(m|n))$ of~\cite{y} recalled in
Section~\ref{ssec quantum super affine}. We construct a family of the PBWD bases
for $U^>_\vv(L\ssl(m|n))$ in Theorem~\ref{Main Theorem 4}, thus proving a conjecture
of~\cite{z1}, see Remark~\ref{zhang's choice}.  We further strengthen this
by constructing a family of the PBWD bases for the
integral form $\fU^>_\vv(L\ssl(m|n))$, see Theorem~\ref{Main Theorem 4.1}.

\medskip
\noindent
In Section~\ref{ssec trigonometric super shuffle}, we introduce the
shuffle algebra $S^{(m|n)}$ featuring new wheel conditions and skew-symmetry
in one family of variables. Generalizing Theorem~\ref{hard shuffle},
we construct an algebra isomorphism $U_\vv^{>}(L\ssl(m|n))\iso S^{(m|n)}$,
see Theorem~\ref{hard shuffle superLie} and its proof in Section~\ref{ssec proof of Theorem 4}.

\medskip
\noindent
$\bullet$
In Section~\ref{ssec half-yangian sl_n}, we recall the \emph{positive subalgebra} of
the Yangian $Y^>_\hbar(\ssl_n)$ and its Drinfeld-Gavarini dual subalgebra
$\bY^>_\hbar(\ssl_n)$, as well as the PBWD bases for those,
see Theorems~\ref{pbwd for yangian},~\ref{pbwd for gavarini yangian}.

\medskip
\noindent
In Section~\ref{ssec rational shuffle algebra}, we introduce a (rational)
counterpart $\bar{W}^{(n)}$ of the (trigonometric) shuffle algebra $S^{(n)}$, equipped
with an embedding $\Psi\colon Y^>_\hbar(\ssl_n)\hookrightarrow \bar{W}^{(n)}$,
see Proposition~\ref{simple 2 shuffle yangian}.
In contrast to Theorem~\ref{hard shuffle}, $\Psi$ is not an isomorphism,
and we provide explicit descriptions of the images
  $W^{(n)}=\Psi(Y^>_\hbar(\ssl_n)),\,
   \fW^{(n)}=\Psi(\bY^>_\hbar(\ssl_n))$
in Theorems~\ref{hard shuffle yangian},~\ref{shuffle integral form yangian},
see Definitions~\ref{good element yangian},~\ref{integral element yangian}.

\medskip
\noindent
$\bullet$
In Section~\ref{ssec super half-yangian sl_n}, we recall the positive subalgebra of
the super Yangian $Y^>_\hbar(\ssl(m|n))$, its Drinfeld-Gavarini dual subalgebra
$\bY^>_\hbar(\ssl(m|n))$, and their PBWD bases, see
Theorems~\ref{pbwd for superyangian},~\ref{pbwd for gavarini superyangian}.

\medskip
\noindent
In Section~\ref{ssec rational shuffle super algebra}, we introduce a
(rational) counterpart $\bar{W}^{(m|n)}$ of the (trigonometric) shuffle algebra $S^{(m|n)}$,
equipped with an embedding $\Psi\colon Y^>_\hbar(\ssl(m|n))\hookrightarrow \bar{W}^{(m|n)}$, 
see Proposition~\ref{simple shuffle super yangian}.
We provide explicit descriptions of the images
  $W^{(m|n)}=\Psi(Y^>_\hbar(\ssl(m|n)))$ and $\fW^{(m|n)}=\Psi(\bY^>_\hbar(\ssl(m|n)))$
in Theorems~\ref{hard shuffle super yangian},~\ref{shuffle integral form super yangian}, 
see Definition~\ref{good element super yangian}.

\medskip
\noindent
$\bullet$
In Section~\ref{ssec grojnowski integral form}, we recall another integral form
$\sU^>_\vv(L\ssl_n)$ of $U^>_\vv(L\ssl_n)$, first explicitly considered in~\cite{gr}.
We construct a family of the PBWD bases for $\sU^>_\vv(L\ssl_n)$ in
Theorem~\ref{Main Theorem Grojnowski} and provide its shuffle algebra realization in
Theorem~\ref{shuffle Grojnowski}, see Definition~\ref{good element quantum}. 
The former yields a family of the PBWD bases
for the Lusztig form~\cite{lus} of $U_\vv(L\ssl_n)$, due to
Remark~\ref{Lusztig vs Grojnowski} and Theorem~\ref{triangular CP}.
In the rest of Section~\ref{sec further directions}, we recall the key results
of the companion papers~\cite{ts,t2}.


\subsection{Acknowledgments}
\

I am indebted to Pavel Etingof, Boris Feigin, Michael Finkelberg, and Andrei Negu\c{t}
for numerous stimulating discussions over the years; to Naihuan Jing for a useful
correspondence on two-parameter quantum algebras; to Luan Bezerra and Evgeny Mukhin
for a useful correspondence on the quantum affine superalgebras;
to anonymous referees for extremely useful suggestions which significantly improved the overall exposition.

\medskip
\noindent
I am also grateful to MPIM (Bonn, Germany), IPMU (Kashiwa, Japan),
and RIMS (Kyoto, Japan) for the hospitality and wonderful working conditions
in the summer 2018 when the first stages of this project were performed. I would like to thank
Tomoyuki Arakawa and Todor Milanov for their invitations to RIMS and IPMU, respectively.

\medskip
\noindent
I gratefully acknowledge NSF Grants DMS-$1821185$, DMS-$2001247$, and DMS-$2037602$.


\section{Quantum loop algebra $U_\vv(L\ssl_n)$ and its integral form $\fU_\vv(L\ssl_n)$}
\label{sec Classical quantuma affine}


\subsection{Quantum loop algebra $U_\vv(L\ssl_n)$}\label{ssec affine sl_n}
\

Let $I=\{1,\ldots,n-1\}$, $(c_{ij})_{i,j\in I}$ be the Cartan matrix of $\ssl_n$,
and $\vv$ be a formal variable. Following~\cite{d}, define the quantum loop algebra
of $\ssl_n$ (in the new Drinfeld presentation), denoted by $U_\vv(L\ssl_n)$,
to be the associative $\BC(\vv)$-algebra generated by
  $\{e_{i,r},f_{i,r},\psi^\pm_{i,\pm s}\}_{i\in I}^{r\in \BZ, s\in \BN}$
with the following defining relations:
\begin{equation}\label{Aff 1}
  [\psi_i^\epsilon(z),\psi_j^{\epsilon'}(w)]=0,\
  \psi^\pm_{i,0}\cdot \psi^\mp_{i,0}=1,
\end{equation}
\begin{equation}\label{Aff 2}
  (z-\vv^{c_{ij}}w)e_i(z)e_j(w)=(\vv^{c_{ij}}z-w)e_j(w)e_i(z),
\end{equation}
\begin{equation}\label{Aff 3}
  (\vv^{c_{ij}}z-w)f_i(z)f_j(w)=(z-\vv^{c_{ij}}w)f_j(w)f_i(z),
\end{equation}
\begin{equation}\label{Aff 4}
  (z-\vv^{c_{ij}}w)\psi_i^\epsilon(z)e_j(w)=(\vv^{c_{ij}}z-w)e_j(w)\psi_i^\epsilon(z),
\end{equation}
\begin{equation}\label{Aff 5}
  (\vv^{c_{ij}}z-w)\psi^\epsilon_i(z)f_j(w)=(z-\vv^{c_{ij}}w)f_j(w)\psi^\epsilon_i(z),
\end{equation}
\begin{equation}\label{Aff 6}
  [e_i(z),f_j(w)]=
  \frac{\delta_{ij}}{\vv-\vv^{-1}}\delta\left(\frac{z}{w}\right)\left(\psi^+_i(z)-\psi^-_i(z)\right),
\end{equation}
\begin{equation}\label{Aff 7}
\begin{split}
  & e_i(z)e_j(w)=e_j(w)e_i(z)\ \mathrm{if}\ c_{ij}=0,\\
  & [e_i(z_1),[e_i(z_2),e_j(w)]_{\vv^{-1}}]_{\vv}+
    [e_i(z_2),[e_i(z_1),e_j(w)]_{\vv^{-1}}]_{\vv}=0 \ \mathrm{if}\ c_{ij}=-1,
\end{split}
\end{equation}
\begin{equation}\label{Aff 8}
\begin{split}
  & f_i(z)f_j(w)=f_j(w)f_i(z)\ \mathrm{if}\ c_{ij}=0,\\
  & [f_i(z_1),[f_i(z_2),f_j(w)]_{\vv^{-1}}]_{\vv}+
    [f_i(z_2),[f_i(z_1),f_j(w)]_{\vv^{-1}}]_{\vv}=0 \ \mathrm{if}\ c_{ij}=-1,
\end{split}
\end{equation}
where $[a,b]_x:=ab-x\cdot ba$ and the generating series are defined as follows:
\begin{equation*}
  e_i(z):=\sum_{r\in \BZ}{e_{i,r}z^{-r}},\
  f_i(z):=\sum_{r\in \BZ}{f_{i,r}z^{-r}},\
  \psi_i^{\pm}(z):=\sum_{s\geq 0}{\psi^\pm_{i,\pm s}z^{\mp s}},\
  \delta(z):=\sum_{r\in \BZ}{z^r}.
\end{equation*}

Let $U^{<}_\vv(L\ssl_n),U^{>}_\vv(L\ssl_n),U^{0}_\vv(L\ssl_n)$ be the
$\BC(\vv)$-subalgebras of $U_\vv(L\ssl_n)$ generated respectively by
  $\{f_{i,r}\}_{i\in I}^{r\in \BZ},
   \{e_{i,r}\}_{i\in I}^{r\in \BZ},
   \{\psi^\pm_{i,\pm s}\}_{i\in I}^{s\in \BN}$.
The following is standard (see e.g.~\cite[Theorem 2]{he}):

\begin{Prop}\label{Triangular decomposition}
(a) (Triangular decomposition of $U_\vv(L\ssl_n)$)
The multiplication map
\begin{equation*}
  m\colon
  U^{<}_\vv(L\ssl_n)\otimes_{\BC(\vv)} U^{0}_\vv(L\ssl_n)\otimes_{\BC(\vv)} U^{>}_\vv(L\ssl_n)
  \longrightarrow U_\vv(L\ssl_n)
\end{equation*}
is an isomorphism of $\BC(\vv)$-vector spaces.

\noindent
(b) The algebra $U^{>}_\vv(L\ssl_n)$ (resp.\ $U^{<}_\vv(L\ssl_n)$ and $U^{0}_\vv(L\ssl_n)$)
is isomorphic to the associative $\BC(\vv)$-algebra generated by $\{e_{i,r}\}_{i\in I}^{r\in \BZ}$
(resp.\ $\{f_{i,r}\}_{i\in I}^{r\in \BZ}$ and $\{\psi^\pm_{i,\pm s}\}_{i\in I}^{s\in \BN}$)
with the defining relations~(\ref{Aff 2}, \ref{Aff 7})
(resp.~(\ref{Aff 3}, \ref{Aff 8}) and~(\ref{Aff 1})).
\end{Prop}


\subsection{PBWD bases of $U_\vv(L\ssl_n)$}\label{ssec formulation main thm 1}
\

Let $\{\alpha_i\}_{i=1}^{n-1}$ be the standard simple positive roots of $\ssl_n$,
and $\Delta^+$ be the set of positive roots:
  $\Delta^+=\{\alpha_j+\alpha_{j+1}+\ldots+\alpha_i\}_{1\leq j\leq i<n}$.
Consider the following total order ``$\leq$'' on $\Delta^+$:
\begin{equation}\label{order}
  \alpha_j+\alpha_{j+1}+\ldots+\alpha_i\leq \alpha_{j'}+\alpha_{j'+1}+\ldots+\alpha_{i'}
  \ \, \mathrm{iff}\ \, j<j'\ \mathrm{or}\ j=j',i\leq i'.
\end{equation}
We also pick a total order $\preceq_\beta$ on $\BZ$ for any $\beta\in \Delta^+$.
This gives rise to the total order ``$\leq$'' on $\Delta^+\times \BZ$:
\begin{equation}\label{extended order}
  (\beta,r)\leq (\beta',r') \ \, \mathrm{iff}\ \,
  \beta<\beta' \ \mathrm{or}\ \beta=\beta', r\preceq_\beta r'.
\end{equation}
For every pair $(\beta,r)\in \Delta^+\times \BZ$, we choose:
\begin{enumerate}
\item[1)]
a decomposition $\beta=\alpha_{i_1}+\ldots+\alpha_{i_p}$ such that
$[\cdots[e_{\alpha_{i_1}},e_{\alpha_{i_2}}],\cdots,e_{\alpha_{i_p}}]$
is a non-zero root vector $e_\beta$ of $\ssl_n$
(here, $e_{\alpha_i}$ denotes the standard Chevalley generator of $\ssl_n$);

\item[2)]
a decomposition $r=r_1+\ldots+r_p$ with $r_k\in \BZ$;

\item[3)]
a sequence $(\lambda_1,\ldots,\lambda_{p-1})\in \{\vv,\vv^{-1}\}^{p-1}$.
\end{enumerate}
Then, we define the \emph{PBWD basis elements} $e_\beta(r)\in U^{>}_\vv(L\ssl_n)$
and $f_\beta(r)\in U^{<}_\vv(L\ssl_n)$ via
\begin{equation}\label{higher roots}
\begin{split}
  & e_\beta(r):=
    [\cdots[[e_{i_1,r_1},e_{i_2,r_2}]_{\lambda_1},e_{i_3,r_3}]_{\lambda_2},\cdots,e_{i_p,r_p}]_{\lambda_{p-1}},\\
  & f_\beta(r):=
    [\cdots[[f_{i_1,r_1},f_{i_2,r_2}]_{\lambda_1},f_{i_3,r_3}]_{\lambda_2},\cdots,f_{i_p,r_p}]_{\lambda_{p-1}}.
\end{split}
\end{equation}
In particular, $e_{\alpha_i}(r)=e_{i,r}$ and $f_{\alpha_i}(r)=f_{i,r}$.
We note that $e_\beta(r)$ and $f_\beta(r)$ degenerate to the corresponding
root generators $e_\beta\otimes t^r$ and $f_\beta\otimes t^r$ of
$\ssl_n[t,t^{-1}]=\ssl_n\otimes_\BC \BC[t,t^{-1}]$ as $\vv\to 1$,
hence, the terminology.

\begin{Rem}\label{ft's choice}
The following particular choice features manifestly in~\cite{ft2}
(cf.~Remark~\ref{rosso's choice}):
\begin{equation}\label{simplest choice}
\begin{split}
  & e_{\alpha_j+\alpha_{j+1}+\ldots+\alpha_i}(r):=
    [\cdots[[e_{j,r},e_{j+1,0}]_\vv,e_{j+2,0}]_\vv,\cdots,e_{i,0}]_\vv,\\
  & f_{\alpha_j+\alpha_{j+1}+\ldots+\alpha_i}(r):=
    [\cdots[[f_{j,r},f_{j+1,0}]_\vv,f_{j+2,0}]_\vv,\cdots,f_{i,0}]_\vv.
\end{split}
\end{equation}
\end{Rem}

Let $H$ denote the set of all functions $h\colon \Delta^+\times \BZ\to \BN$
with finite support. The monomials
\begin{equation}\label{ordered}
  e_h\ :=\prod\limits_{(\beta,r)\in \Delta^+\times \BZ}^{\rightarrow} e_\beta(r)^{h(\beta,r)}\ ,
  \qquad
  f_h\ :=\prod\limits_{(\beta,r)\in \Delta^+\times \BZ}^{\leftarrow} f_\beta(r)^{h(\beta,r)},
  \qquad \forall\, h\in H
\end{equation}
will be called the \emph{ordered PBWD monomials} of
$U^{>}_\vv(L\ssl_n)$ and $U^{<}_\vv(L\ssl_n)$.
Here, the arrows $\rightarrow$ and $\leftarrow$ over the product signs refer
to the total order~(\ref{extended order}) and its opposite, respectively.

Our first main result establishes the PBWD property of $U^>_\vv(L\ssl_n)$
and $U^<_\vv(L\ssl_n)$ (cf.~\cite{l}):

\begin{Thm}\label{Main Theorem 1}
(a) The ordered PBWD monomials $\{e_h\}_{h\in H}$
form a $\BC(\vv)$-basis of $U^{>}_\vv(L\ssl_n)$.

\noindent
(b) The ordered PBWD monomials $\{f_h\}_{h\in H}$
form a $\BC(\vv)$-basis of $U^{<}_\vv(L\ssl_n)$.
\end{Thm}

The proof of Theorem~\ref{Main Theorem 1} is presented in
Section~\ref{ssec proof of Theorem 1} and is based on the shuffle approach.

Let us relabel the Cartan generators via
  $\psi_{i,r}:=
   \begin{cases}
     \psi^+_{i,r}, & \text{if } r\geq 0 \\
     \psi^-_{i,r}, & \text{if } r<0
   \end{cases},$
so that $(\psi_{i,0})^{-1}=\psi^-_{i,0}$. Let $H_0$ denote the set of all
functions $g\colon I\times \BZ\to \BZ$ with finite support and such that
$g(i,r)\geq 0$ for $r\ne 0$. The monomials
(note that the order of the products is irrelevant, due to~(\ref{Aff 1}))
\begin{equation}\label{ordered cartan}
  \psi_g\ :=\prod\limits_{(i,r)\in I\times \BZ} \psi_{i,r}^{g(i,r)}\ ,
  \qquad \forall\, g\in H_0
\end{equation}
will be called the \emph{PBWD monomials} of $U^{0}_\vv(L\ssl_n)$.

According to Proposition~\ref{Triangular decomposition}(b), the
PBWD monomials $\{\psi_g\}_{g\in H_0}$ form a $\BC(\vv)$-basis of
$U^{0}_\vv(L\ssl_n)$. Combining this with Theorem~\ref{Main Theorem 1}
and Proposition~\ref{Triangular decomposition}(a), we finally get:

\begin{Thm}\label{Main Theorem 1 entire algebra}
The elements
\begin{equation*}
  \Big\{f_{h_-}\cdot \psi_{h_0}\cdot e_{h_+}\ |\ h_-,h_+\in H,h_0\in H_0\Big\}
\end{equation*}
form a $\BC(\vv)$-basis of the quantum loop algebra $U_\vv(L\ssl_n)$.
\end{Thm}


\subsection{Integral form $\fU_\vv(L\ssl_n)$ and its PBWD bases}\label{ssec formulation main thm 2}
\

Following the above notations, define $\wt{e}_\beta(r)\in U^{>}_\vv(L\ssl_n)$
and $\wt{f}_\beta(r)\in U^{<}_\vv(L\ssl_n)$ via
\begin{equation}\label{integral basis PBW}
   \wt{e}_\beta(r):=(\vv-\vv^{-1})e_\beta(r),
   \quad
   \wt{f}_\beta(r):=(\vv-\vv^{-1})f_\beta(r),
   \quad \forall\, (\beta,r)\in \Delta^+\times \BZ.
\end{equation}
We also define $\wt{e}_h,\wt{f}_h$ via the formula~(\ref{ordered}) but using
$\wt{e}_\beta(r),\wt{f}_\beta(r)$ instead of $e_\beta(r),f_\beta(r)$.
Finally, we define integral forms $\fU^>_\vv(L\ssl_n)$ and $\fU^<_\vv(L\ssl_n)$ as the
$\BC[\vv,\vv^{-1}]$-subalgebras of $U^>_\vv(L\ssl_n)$ and $U^<_\vv(L\ssl_n)$
generated by
  $\{\wt{e}_\beta(r)\}_{\beta\in \Delta^+}^{r\in \BZ}$
and
  $\{\wt{f}_\beta(r)\}_{\beta\in \Delta^+}^{r\in \BZ}$,
respectively.

We note that the above definition of $\fU^>_\vv(L\ssl_n),\fU^<_\vv(L\ssl_n)$
depends on all the choices 1)--3) made when defining $e_\beta(r),f_\beta(r)$
in~(\ref{higher roots}). Our next result establishes that they are actually
independent of these choices and posses PBWD bases analogous to those of
Theorem~\ref{Main Theorem 1}.

\begin{Thm}\label{Main Theorem 2}
(a) The subalgebras $\fU^>_\vv(L\ssl_n)$ and $\fU^<_\vv(L\ssl_n)$ are independent
of all our~choices.

\noindent
(b) The ordered PBWD monomials $\{\wt{e}_h\}_{h\in H}$ form a basis of the free
$\BC[\vv,\vv^{-1}]$-module $\fU^{>}_\vv(L\ssl_n)$.

\noindent
(c) The ordered PBWD monomials $\{\wt{f}_h\}_{h\in H}$ form a basis of the free
$\BC[\vv,\vv^{-1}]$-module $\fU^{<}_\vv(L\ssl_n)$.
\end{Thm}

The proof of Theorem~\ref{Main Theorem 2} is presented in
Section~\ref{ssec proof of Theorem 2} and is based on the shuffle approach.

\medskip
We also define an integral form $\fU^0_\vv(L\ssl_n)$ as the
$\BC[\vv,\vv^{-1}]$-subalgebra of $U^0_\vv(L\ssl_n)$ generated by
$\{\psi^\pm_{i,\pm s}\}_{i\in I}^{s\in \BN}$. Due to
Proposition~\ref{Triangular decomposition}(b), the PBWD monomials
$\{\psi_g\}_{g\in H_0}$ of~(\ref{ordered cartan}) form a basis of
the free $\BC[\vv,\vv^{-1}]$-module $\fU^{0}_\vv(L\ssl_n)$.

\begin{Def}\label{definition entire integral form}
The integral form $\fU_\vv(L\ssl_n)$ is defined as the
$\BC[\vv,\vv^{-1}]$-subalgebra of $U_\vv(L\ssl_n)$ generated by
  $\{\wt{e}_\beta(r), \wt{f}_\beta(r)\}_{\beta\in \Delta^+}^{r\in \BZ}\cup
   \{\psi^\pm_{i,\pm s}\}_{i\in I}^{s\in \BN}$.
\end{Def}

\begin{Rem}\label{independence integral form}
Due to Theorem~\ref{Main Theorem 2}(a), the algebra $\fU_\vv(L\ssl_n)$ itself
is independent of any choices (made in the
definition~(\ref{higher roots},~\ref{integral basis PBW})
of the PBWD basis elements $\wt{e}_\beta(r), \wt{f}_\beta(r)$).
\end{Rem}

The following result is proved in~\cite[Theorem 3.24]{ft2}
(cf.~Theorem~\ref{Main Theorem 1 entire algebra}):

\begin{Thm}[\cite{ft2}]\label{Full PBWD for integral-special choice}
For the particular choice~(\ref{simplest choice}) of $e_\beta(r),f_\beta(r)$
in~(\ref{integral basis PBW}), the elements
\begin{equation*}
  \Big\{\wt{f}_{h_-}\cdot \psi_{h_0}\cdot \wt{e}_{h_+}\ |\ h_-,h_+\in H,h_0\in H_0\Big\}
\end{equation*}
form a basis of the free $\BC[\vv,\vv^{-1}]$-module $\fU_\vv(L\ssl_n)$.
\end{Thm}

In view of Theorem~\ref{Main Theorem 2}, this gives rise to the
\emph{triangular decomposition} of $\fU_\vv(L\ssl_n)$:

\begin{Cor}\label{Triangular for integral form}
The multiplication map
\begin{equation*}
  m\colon
  \fU^{<}_\vv(L\ssl_n)\otimes_{\BC[\vv,\vv^{-1}]}
  \fU^{0}_\vv(L\ssl_n)\otimes_{\BC[\vv,\vv^{-1}]}
  \fU^{>}_\vv(L\ssl_n)\longrightarrow \fU_\vv(L\ssl_n)
\end{equation*}
is an isomorphism of the free $\BC[\vv,\vv^{-1}]$-modules.
\end{Cor}

Combining this Corollary with Theorem~\ref{Main Theorem 2}, we finally obtain:

\begin{Thm}\label{Full PBWD for integral}
For any choices 1)--3) made prior to~(\ref{higher roots}), the elements
\begin{equation*}
  \Big\{\wt{f}_{h_-}\cdot \psi_{h_0}\cdot \wt{e}_{h_+}\ |\ h_-,h_+\in H,h_0\in H_0\Big\}
\end{equation*}
form a basis of the free $\BC[\vv,\vv^{-1}]$-module $\fU_\vv(L\ssl_n)$.
\end{Thm}

\begin{Rem}
All results of this Section also hold when $\BC[\vv,\vv^{-1}]$ is replaced with $\BZ[\vv,\vv^{-1}]$.
\end{Rem}

We conclude this section with a remark discussing the $\gl_n$-counterpart of $\fU_\vv(L\ssl_n)$:

\begin{Rem}\label{ft's RTT interpretation}
(a) It is often more convenient to work with the quantum loop algebra of $\gl_n$,
denoted by $U_\vv(L\gl_n)$, which roughly speaking is obtained from $U_\vv(L\ssl_n)$
by adding a Cartan current. Its integral form $\fU_\vv(L\gl_n)$ is defined
similarly to $\fU_\vv(L\ssl_n)$ of Definition~\ref{definition entire integral form},
and admits a triangular decomposition
  $\fU_\vv(L\gl_n) \simeq
   \fU^{<}_\vv(L\gl_n)\otimes_{\BC[\vv,\vv^{-1}]}\fU^{0}_\vv(L\gl_n)
   \otimes_{\BC[\vv,\vv^{-1}]} \fU^{>}_\vv(L\gl_n)$,
cf.~Corollary~\ref{Triangular for integral form}.
Here,
  $\fU^{<}_\vv(L\gl_n)\simeq \fU^{<}_\vv(L\ssl_n),
   \fU^{>}_\vv(L\gl_n)\simeq \fU^{>}_\vv(L\ssl_n),
   \fU^{0}_\vv(L\gl_n)\supset \fU^{0}_\vv(L\ssl_n)$.

\noindent
(b) The proof of Theorem~\ref{Full PBWD for integral-special choice} presented
in~\cite{ft2} crucially utilizes the identification of $\fU_\vv(L\gl_n)$ and
the RTT integral form $\fU^\mathrm{rtt}_\vv(L\gl_n)$ of~\cite{frt}
(see~\cite[\S3.2, Proposition 3.11]{ft2}) under the $\BC(\vv)$-algebra isomorphism
  $U_\vv(L\gl_n)\simeq \fU^\mathrm{rtt}_\vv(L\gl_n)\otimes_{\BC[\vv,\vv^{-1}]}\BC(\vv)$
of~\cite{df}.

\noindent
(c) We note that the integral form $\fU_\vv(L\gl_n)$ provides a quantization
of the thick slice $^\dagger\CW_0$ of~\cite[\S4.8]{ft}, see~\cite[Remark 3.26]{ft2}.
More precisely, we have
\begin{equation}\label{integral_affine_gln}
  \fU_\vv(L\gl_n)/(\vv-1)\simeq
  \BC\left[t^\pm_{ji}[\pm r]\right]_{1\leq j,i\leq n}^{r\geq 0}\Big/
  \left(\langle t^+_{ij}[0],t^-_{ji}[0],t^\pm_{kk}[0]t^\mp_{kk}[0]-1 \rangle_{1\leq j<i\leq n}^{1\leq k\leq n}\right).
\end{equation}
\end{Rem}


\section{Shuffle algebra realizations of $U^>_\vv(L\ssl_n)$ and $\fU^>_\vv(L\ssl_n)$}
\label{sec proofs of Theorems 1,2}

In this section, we establish the shuffle algebra realizations of $U^{>}_\vv(L\ssl_n)$
and $\fU^{>}_\vv(L\ssl_n)$ (hence, the independence of the latter from all choices made,
Theorem~\ref{Main Theorem 2}(a)), and use those to prove Theorems~\ref{Main Theorem 1}(a)
and~\ref{Main Theorem 2}(b). As the assignment $e_{i,r}\mapsto f_{i,r}\ (i\in I,r\in \BZ)$
gives rise to a $\BC(\vv)$-algebra anti-isomorphism $U^{>}_\vv(L\ssl_n)\to U^{<}_\vv(L\ssl_n)$
(resp.\ a $\BC[\vv,\vv^{-1}]$-algebra anti-isomorphism $\fU^{>}_\vv(L\ssl_n)\to \fU^{<}_\vv(L\ssl_n)$)
that maps the ordered PBWD monomials of the source to the ordered PBWD monomials of the target
(up to a sign and an integer power of $\vv$), both Theorems~\ref{Main Theorem 1}(b)
and~\ref{Main Theorem 2}(c) follow as well.


\subsection{Shuffle algebra $S^{(n)}$}\label{ssec usual shuffle algebra}
\

We follow the notations of~\cite[Appendix I(ii)]{ft} (cf.~\cite{n}).\footnote{These
are trigonometric counterparts of the elliptic shuffle algebras of
Feigin-Odesskii~\cite{fo1}--\cite{fo3}.}
Let $\Sigma_k$ denote the symmetric group in $k$ elements, and set
$\Sigma_{(k_1,\ldots,k_{n-1})}:=\Sigma_{k_1}\times \cdots\times \Sigma_{k_{n-1}}$
for $k_1,\ldots,k_{n-1}\in \BN$. Consider an $\BN^I$-graded $\BC(\vv)$-vector space
  $\BS^{(n)}=
   \underset{\underline{k}=(k_1,\ldots,k_{n-1})\in \BN^{I}}
   \bigoplus\BS^{(n)}_{\underline{k}}$,
where $\BS^{(n)}_{(k_1,\ldots,k_{n-1})}$ consists of $\Sigma_{\unl{k}}$-symmetric
rational functions in the variables $\{x_{i,r}\}_{i\in I}^{1\leq r\leq k_i}$.
We fix an $I\times I$ matrix of rational functions
$(\zeta_{i,j}(z))_{i,j\in I} \in \mathrm{Mat}_{I\times I}(\BC(\vv)(z))$
via $\zeta_{i,j}(z):=\frac{z-\vv^{-c_{ij}}}{z-1}$.
Let us now introduce the bilinear \emph{shuffle product} $\star$ on $\BS^{(n)}$:
given  $F\in \BS^{(n)}_{\underline{k}}$ and $G\in \BS^{(n)}_{\underline{\ell}}$,
define $F\star G\in \BS^{(n)}_{\underline{k}+\underline{\ell}}$ via
\begin{equation}\label{shuffle product}
\begin{split}
  & (F\star G)(x_{1,1},\ldots,x_{1,k_1+\ell_1};\ldots;x_{n-1,1},\ldots, x_{n-1,k_{n-1}+\ell_{n-1}}):=
     \frac{1}{\unl{k}!\cdot\unl{\ell}!}\times\\
  & \Sym_{\Sigma_{\unl{k}+\unl{\ell}}}
    \left(F\left(\{x_{i,r}\}_{i\in I}^{1\leq r\leq k_i}\right) G\left(\{x_{i',r'}\}_{i'\in I}^{k_{i'}<r'\leq k_{i'}+\ell_{i'}}\right)\cdot
    \prod_{i\in I}^{i'\in I}\prod_{r\leq k_i}^{r'>k_{i'}}\zeta_{i,i'}(x_{i,r}/x_{i',r'})\right).
\end{split}
\end{equation}
Here, $\unl{k}!=\prod_{i\in I}k_i!$, while the \emph{symmetrization}
of $f\in \BC(\{x_{i,1},\ldots,x_{i,m_i}\}_{i\in I})$
is defined via
\begin{equation*}
  \Sym_{\Sigma_{\unl{m}}}(f)(\{x_{i,1},\ldots,x_{i,m_i}\}_{i\in I})\ :=
  \sum_{(\sigma_1,\ldots,\sigma_{n-1})\in \Sigma_{\unl{m}}}
  f\left(\{x_{i,\sigma_i(1)},\ldots,x_{i,\sigma_i(m_i)}\}_{i\in I}\right).
\end{equation*}
This endows $\BS^{(n)}$ with a structure of an associative unital algebra
with the unit $\textbf{1}\in \BS^{(n)}_{(0,\ldots,0)}$.

We will be interested only in the subspace of $\BS^{(n)}$ defined by
the \emph{pole} and \emph{wheel conditions}:

\begin{itemize}

\item[$\bullet$]
We say that $F\in \BS^{(n)}_{\underline{k}}$ satisfies the \emph{pole conditions} if
\begin{equation}\label{pole conditions}
  F=
  \frac{f(x_{1,1},\ldots,x_{n-1,k_{n-1}})}
       {\prod_{i=1}^{n-2}\prod_{r\leq k_i}^{r'\leq k_{i+1}}(x_{i,r}-x_{i+1,r'})},\
  \mathrm{where}\ f\in \BC(\vv)[\{x_{i,r}^{\pm 1}\}_{i\in I}^{1\leq r\leq k_i}]^{\Sigma_{\unl{k}}}.
\end{equation}

\item[$\bullet$]
We say that $F\in \BS^{(n)}_{\underline{k}}$ satisfies the
\emph{wheel conditions}\footnote{Following~\cite{fo1}--\cite{fo3}, the role of
the wheel conditions is exactly to replace complicated Serre relations.} if
\begin{equation}\label{wheel conditions}
  F(\{x_{i,r}\})=0\ \mathrm{once}\ x_{i,r_1}=\vv x_{i+\epsilon,s}=\vv^2 x_{i,r_2}\
  \mathrm{for\ some}\ \epsilon, i, r_1, r_2, s,
\end{equation}
where
  $\epsilon\in \{\pm 1\},\, i,i+\epsilon\in I,\,
   1\leq r_1\ne r_2\leq k_i,\, 1\leq s\leq k_{i+\epsilon}$.

\end{itemize}

\medskip
\noindent
Let $S^{(n)}_{\underline{k}}\subset \BS^{(n)}_{\underline{k}}$ denote
the subspace of all elements $F$ satisfying these two conditions and set
  $$S^{(n)}:=\underset{\underline{k}\in \BN^{I}}\bigoplus S^{(n)}_{\underline{k}}$$
It is straightforward to check that the subspace $S^{(n)}\subset\BS^{(n)}$ is $\star$-closed.
The \textbf{shuffle algebra} $\left(S^{(n)},\star\right)$ is related to $U^>_\vv(L\ssl_n)$
via the following construction:\footnote{In the formal setup (when working over
$\BC[[\hbar]]$ rather than over $\BC(v)$), this goes back to~\cite[Corollary 1.4]{e}.}

\begin{Prop}\label{simple shuffle}
The assignment $e_{i,r}\mapsto x_{i,1}^r\ (i\in I,r\in \BZ)$
gives rise to an injective  $\BC(\vv)$-algebra homomorphism
$\Psi\colon U_\vv^{>}(L\ssl_n)\to S^{(n)}$.
\end{Prop}

\begin{proof}
The assignment $e_{i,r}\mapsto x_{i,1}^r$ is compatible with
relations~(\ref{Aff 2}, \ref{Aff 7}), hence, it gives rise to a
$\BC(\vv)$-algebra homomorphism $\Psi\colon U_\vv^{>}(L\ssl_n)\to S^{(n)}$,
due to Proposition~\ref{Triangular decomposition}(b).

The injectivity of $\Psi$ follows from the general arguments based on
the existence of a non-degenerate pairing on the source and a pairing
on the target compatible with the former one via $\Psi$. This is explained
in details in~\cite[Lemma 2.20, Proposition 2.30, Proposition 3.8]{n}.
\end{proof}

The next result follows from its much harder counterpart~\cite[Theorem 1.1]{n},
but we will provide its alternative simpler proof\footnote{One of the benefits
of our proof is that it will be directly generalized to establish the isomorphisms of
Theorems~\ref{hard shuffle 2-parametric} and~\ref{hard shuffle superLie} below,
for which no analogue of~\cite[Theorem 1.1]{n} is known at the moment.} below,
see Remark~\ref{easier hard shuffle}:

\begin{Thm}\label{hard shuffle}
$\Psi\colon U_\vv^{>}(L\ssl_n)\iso S^{(n)}$ of Proposition~\ref{simple shuffle} is a $\BC(\vv)$-algebra isomorphism.
\end{Thm}

Before we proceed to the proofs of Theorem~\ref{Main Theorem 1}(a) and
Theorem~\ref{hard shuffle}, let us introduce the key tool in our study
of the shuffle algebra, the so-called \emph{specialization maps}.
For a positive root $\beta=\alpha_j+\alpha_{j+1}+\ldots+\alpha_i$,
define $j(\beta):=j,i(\beta):=i$, and let $[\beta]$ denote the integer interval
$[j(\beta);i(\beta)]$. Consider a collection of the intervals
$\{[\beta]\}_{\beta\in \Delta^+}$ each taken with a multiplicity $d_{\beta}\in \BN$
and ordered with respect to the total order~(\ref{order}) on $\Delta^+$
(the order inside each group is irrelevant).
Let $\unl{d}\in \BN^{\frac{n(n-1)}{2}}$ denote the collection $\{d_\beta\}_{\beta\in \Delta^+}$.
Define $\unl{\ell}=(\ell_1,\ldots,\ell_{n-1})\in \BN^I$, the $\BN^I$-degree of $\unl{d}$, via
\begin{equation}
\label{eqn: from d to l}
  \sum_{i\in I} \ell_i\alpha_i=\sum_{\beta\in \Delta^+} d_{\beta}\beta.
\end{equation}
Let us now define the \textbf{specialization map}
\begin{equation}\label{specialization map}
  \phi_{\unl{d}}\colon S^{(n)}_{\unl{\ell}}\longrightarrow
  \BC(\vv)[\{y_{\beta,s}^{\pm 1}\}_{\beta\in \Delta^+}^{1\leq s\leq d_\beta}]^{\Sigma_{\unl{d}}}.
\end{equation}
We split the variables $\{x_{i,r}\}_{i\in I}^{1\leq r\leq \ell_i}$ into
$\sum_{\beta\in \Delta^+} d_\beta$ groups corresponding to the above intervals,
and specialize those in the $s$-th copy of $[\beta]$ to
  $\vv^{-j(\beta)}\cdot y_{\beta,s},\ldots,\vv^{-i(\beta)}\cdot y_{\beta,s}$
in the natural order
(the variable $x_{i,r}$ is specialized to $\vv^{-i}y_{\beta,s}$). For
  $F=\frac{f(x_{1,1},\ldots,x_{n-1,\ell_{n-1}})}
   {\prod_{i=1}^{n-2}\prod_{1\leq r\leq \ell_i}^{1\leq r'\leq \ell_{i+1}} (x_{i,r}-x_{i+1,r'})}
   \in S^{(n)}_{\unl{\ell}}$,
we define $\phi_{\unl{d}}(F)$ as the corresponding specialization of $f$.
We note that $\phi_{\unl{d}}(F)$ is independent of our splitting of the
variables $\{x_{i,r}\}_{i\in I}^{1\leq r\leq \ell_i}$ into groups and is
symmetric in $\{y_{\beta,s}\}_{s=1}^{d_\beta}$ for any $\beta$.

The key properties of the specialization maps $\phi_{\unl{d}}$ and their
relevance to $\Psi(e_h)$, the images of the ordered PBWD monomials in the shuffle algebra,
are discussed in Lemmas~\ref{lower degrees},~\ref{same degrees},~\ref{spanning} below.
We conclude this section with an explicit example of the specialization maps:

\begin{Ex}\label{example specialization maps}
Consider
  $F=\frac{x^a_{1,1}x^b_{2,1}x^c_{3,1}}{(x_{1,1}-x_{2,1})(x_{2,1}-x_{3,1})}
   \in S^{(n)}_{\unl{\ell}}$,
where $\unl{\ell}=(1,1,1,0,\ldots,0)\in \BN^I$ and $a,b,c\in \BZ$.
Let us compute the images of $F$ under all possible specialization maps:

\noindent
(a) If $\unl{d}$ encodes a single positive root $\beta=\alpha_1+\alpha_2+\alpha_3$,
then $\phi_{\underline{d}}(F)$ is a Laurent polynomial in a single variable $y_{\beta,1}$
and equals
  $(\vv^{-1}y_{\beta,1})^a (\vv^{-2}y_{\beta,1})^b (\vv^{-3}y_{\beta,1})^c$.

\noindent
(b) If $\unl{d}$ encodes two positive roots $\beta_1=\alpha_1,\beta_2=\alpha_2+\alpha_3$,
then $\phi_{\underline{d}}(F)$ is a Laurent polynomial in two variables
$y_{\beta_1,1},y_{\beta_2,1}$ and equals
  $(\vv^{-1}y_{\beta_1,1})^a (\vv^{-2}y_{\beta_2,1})^b (\vv^{-3}y_{\beta_2,1})^c$.

\noindent
(c) If $\unl{d}$ encodes two positive roots $\beta_1=\alpha_1+\alpha_2,\beta_2=\alpha_3$,
then $\phi_{\underline{d}}(F)$ is a Laurent polynomial in two variables
$y_{\beta_1,1},y_{\beta_2,1}$ and equals
  $(\vv^{-1}y_{\beta_1,1})^a (\vv^{-2}y_{\beta_1,1})^b (\vv^{-3}y_{\beta_2,1})^c$.

\noindent
(d) If $\unl{d}$ encodes three positive roots
$\beta_1=\alpha_1,\beta_2=\alpha_2,\beta_3=\alpha_3$, then $\phi_{\underline{d}}(F)$
is a Laurent polynomials in three variables $y_{\beta_1,1},y_{\beta_2,1}, y_{\beta_3,1}$
and equals
  $(\vv^{-1}y_{\beta_1,1})^a (\vv^{-2}y_{\beta_2,1})^b (\vv^{-3}y_{\beta_3,1})^c$.
\end{Ex}


\subsection{Proof of Theorem~\ref{Main Theorem 1}(a)}\label{ssec proof of Theorem 1}
\

Our proof of Theorem~\ref{Main Theorem 1}(a) will proceed in two steps:
first, we shall establish the linear independence\footnote{As pointed out in the
introduction, the linear independence can be deduced from the general arguments
based on the flatness of the deformation and the PBW property of $U(\ssl_n[t,t^{-1}])$.
However, the \emph{specialization maps} of~(\ref{specialization map}) and
formulas~(\ref{explicit factors},~\ref{explicit formula for same degrees})
will be used below to prove that $\{e_h\}_{h\in H}$ span $U^>_\vv(L\ssl_n)$. We will
use the same approach for two-parameter quantum loop algebra, for which
the general arguments do not apply.}
of the ordered PBWD monomials in Section~\ref{ssec linear indep},
and then we will verify that they linearly span the entire algebra
$U_\vv^>(L\ssl_n)$ in Section~\ref{ssec spanning prop} (we note that
the order of these two steps is usually opposite in the proofs of PBW-type theorems).

We start by establishing Theorem~\ref{Main Theorem 1}(a) for $n=2$
in Section~\ref{ssec proof baby case}.


\subsubsection{Proof of Theorem~\ref{Main Theorem 1}(a) for $n=2$}\label{ssec proof baby case}
\

For $k\in \BN$, set $[k]_\vv:=\frac{\vv^k-\vv^{-k}}{\vv-\vv^{-1}}$
and $[k]_\vv!:=[1]_\vv\cdots [k]_\vv$. We start from the following
simple computation in $S^{(2)}$ (as $I=\{1\}$, we shall denote
the variables $x_{1,r}$ simply by $x_r$):

\begin{Lem}\label{k-th fold product}
For any $k\geq 1$ and $r\in \BZ$, the $k$-th power of $x^r\in S^{(2)}_1$ equals
\begin{equation}\label{factorial formula}
  \underbrace{x^r\star\cdots \star x^r}_{k\ \mathrm{times}}=\vv^{-\frac{k(k-1)}{2}} [k]_\vv!\cdot (x_1\cdots x_k)^r.
\end{equation}
\end{Lem}

\begin{proof}
The proof is by induction on $k$, the base case $k=1$ being trivial. Applying the induction
assumption to the $(k-1)$-st power of $x^r$, the proof of~(\ref{factorial formula})
boils down to the verification of
\begin{equation}\label{combinatorial identity}
   \sum_{p=1}^k \prod_{1\leq s\leq k}^{s\ne p} \frac{x_s-\vv^{-2}x_p}{x_s-x_p}=\vv^{1-k}[k]_\vv.
\end{equation}
The left-hand side of~(\ref{combinatorial identity}) is a rational function in
$\{x_p\}_{p=1}^k$ of degree $0$ and is easily checked to have no poles, hence, must be a constant.
To evaluate this constant, let $x_k\to \infty$: the last summand (corresponding to $p=k$) tends to $\vv^{-2(k-1)}$,
while the sum of the first $k-1$ summands (corresponding to $1\leq p\leq k-1$) tends to $1+\vv^{-2}+\ldots+\vv^{-2(k-2)}$
by the induction assumption, thus, resulting in
$1+\vv^{-2}+\vv^{-4}+\ldots+\vv^{-2(k-1)}=\vv^{1-k}[k]_\vv$ as claimed.
\end{proof}

Theorem~\ref{Main Theorem 1}(a) for $n=2$ is equivalent to the following result:

\begin{Lem}\label{n=1 case}
For any total order $\preceq$ on $\BZ$, the ordered monomials
$\{e_{r_1}e_{r_2}\cdots e_{r_k}\}_{k\in \BN}^{r_1\preceq\cdots\preceq r_k}$
form a $\BC(\vv)$-basis of $U^{>}_\vv(L\ssl_2)$.
\end{Lem}

\begin{proof}
For
  $r_1=\cdots=r_{k_1}\prec r_{k_1+1}=\cdots=r_{k_1+k_2}\prec
   \cdots \prec r_{k_1+\ldots+k_{\ell-1}+1}=\cdots=r_{k_1+\ldots+k_{\ell}}$,
set $k:=k_1+\ldots+k_{\ell}$ and choose $\sigma\in \Sigma_k$
so that $r_{\sigma(1)}\leq \cdots\leq r_{\sigma(k)}$. Then
$x^{r_1}\star \cdots \star x^{r_k}$ is a symmetric Laurent polynomial of the form
  $$\nu_{\unl{r}} m_{(r_{\sigma(1)},\ldots,r_{\sigma(k)})}(x_1,\ldots,x_k)+
    \sum \nu_{\unl{r}'}m_{\unl{r}'}(x_1,\ldots,x_k)$$
where
\begin{enumerate}

\item[$\bullet$]
$m_{(s_1,\ldots,s_k)}(x_1,\ldots,x_k)$ (with $s_1\leq \cdots\leq s_k$)
are the monomial symmetric polynomials (that is, the sums of all monomials
$x^{t_1}_1\cdots x^{t_k}_k$ as $(t_1,\ldots,t_k)$ ranges over all
distinct permutations of $(s_1,\ldots,s_k)$),

\item[$\bullet$]
the sum is over
  $\unl{r}'=(r'_1\leq\cdots\leq r'_k)\in \BZ^k$
distinct from $(r_{\sigma(1)},\ldots,r_{\sigma(k)})$ and satisfying
\begin{equation*}
  r_{\sigma(1)}\leq r'_1\leq r'_k\leq r_{\sigma(k)}
  \ \, \mathrm{as\ well\ as}\, \
  r'_1+\ldots+r'_k=r_1+\ldots+r_k,
\end{equation*}

\item[$\bullet$]
the coefficients $\nu_{\unl{r}'}$ are Laurent polynomials in $\vv$,
that is, $\nu_{\unl{r}'} \in \BC[\vv,\vv^{-1}]$,

\item[$\bullet$]
the coefficient $\nu_{\unl{r}}$ is explicitly given by
\begin{equation}
\label{eqn:explicit mu value}
  \nu_{\unl{r}}\ =\prod_{1\leq p\leq \ell} \left(\vv^{-\frac{k_p(k_p-1)}{2}}[k_p]_\vv!\right),
\end{equation}
due to Lemma~\ref{k-th fold product}.
\end{enumerate}
Therefore, the shuffle products
  $\{x^{r_1}\star x^{r_2}\star\cdots \star x^{r_k}\}_{k\in \BN}^{r_1\preceq\cdots\preceq r_k}$
form a $\BC(\vv)$-basis of $S^{(2)}$, since
$\{m_{(s_1,\ldots,s_k)}(x_1,\ldots,x_k)\}_{s_1\leq\cdots\leq s_k}$
form a $\BC(\vv)$-basis of $\BC[\{x_p^{\pm 1}\}_{p=1}^k]^{\Sigma_k}$ and
$S^{(2)}_k\simeq \BC[\{x_p^{\pm 1}\}_{p=1}^k]^{\Sigma_k}$ as vector spaces.
This completes our proof of Lemma~\ref{n=1 case}, due to the injectivity of $\Psi$.
\end{proof}


\subsubsection{Linear independence of $e_h$ and two properties of the specialization maps}\label{ssec linear indep}
\

For an ordered PBWD monomial $e_h\ (h\in H)$, define its \emph{degree}
$\deg(e_h)=\deg(h)\in \BN^{\frac{n(n-1)}{2}}$ as a collection of
$d_{\beta}:=\sum_{r\in \BZ} h(\beta,r)\in \BN\ (\beta\in \Delta^+)$ ordered
with respect to the total order~(\ref{order}) on $\Delta^+$. We consider the
\emph{anti-lexicographical order} on $\BN^{\frac{n(n-1)}{2}}$:
\begin{equation*}
  \{d_\beta\}_{\beta\in \Delta^+}<\{d'_\beta\}_{\beta\in \Delta^+}
  \ \ \mathrm{iff}\ \ \exists\,\gamma\in \Delta^+\
  \mathrm{s.t.}\ d_\gamma>d'_\gamma\ \mathrm{and}\
  d_\beta=d'_\beta\ \forall \beta<\gamma.
\end{equation*}
In what follows, we shall need an explicit formula for $\Psi(e_\beta(r))$:

\begin{Lem}\label{shuffle root elt}
For $1\leq j<i<n$ and $r\in \BZ$, we have
\begin{equation*}
  \Psi(e_{\alpha_j+\alpha_{j+1}+\ldots+\alpha_i}(r))=
  (1-\vv^2)^{i-j}\frac{p(x_{j,1},\ldots,x_{i,1})}{(x_{j,1}-x_{j+1,1})\cdots (x_{i-1,1}-x_{i,1})},
\end{equation*}
where $p(x_{j,1},\ldots,x_{i,1})$ is a degree $r+i-j$ monomial,
up to a sign and an integer power of $\vv$.
\end{Lem}

\begin{proof}
Straightforward computation.
\end{proof}

\begin{Ex}
$p(x_{j,1},\ldots,x_{i,1})=x_{j,1}^{r+1}x_{j+1,1}\cdots x_{i-1,1}$
for the particular choice of~(\ref{simplest choice}).
\end{Ex}

Our proof of the linear independence of $\{e_h\}_{h\in H}$ is crucially based on the
following two key properties of the specialization maps introduced in~(\ref{specialization map}):

\begin{Lem}\label{lower degrees}
If $\deg(h)<\unl{d}$ and $\BN^I$-degrees of $\unl{d}$ and $\deg(h)$ coincide,
then $\phi_{\unl{d}}(\Psi(e_h))=0$.
\end{Lem}

\begin{proof}
The condition $\deg(h)<\unl{d}$ guarantees that $\phi_{\unl{d}}$-specialization
of any summand of the symmetrization appearing in $\Psi(e_h)$ contains among
all the $\zeta$-factors at least one factor of the form $\zeta_{i,i+1}(\vv)=0$,
hence, it is zero. The result follows.
\end{proof}

\begin{Lem}\label{same degrees}
The specializations $\{\phi_{\unl{d}}(\Psi(e_h))\}_{h\in H}^{\deg(h)=\unl{d}}$
are linearly independent over $\BC(\vv)$.
\end{Lem}

\begin{proof}
Consider the image $\Psi(e_h)\in S^{(n)}_{\unl{k}}$ (here, $\unl{k}$ is
the $\BN^I$-degree of $\unl{d}=\deg(h)$). It is a sum of $\unl{k}!$ terms,
but most of them specialize to zero under $\phi_{\unl{d}}$, as in the proof
of Lemma~\ref{lower degrees}. The summands which do not specialize to zero
are parametrized by
  $\Sigma_{\unl{d}}:=\prod_{\beta\in \Delta^+} \Sigma_{d_{\beta}}$.
More precisely, given
  $(\sigma_\beta)_{\beta\in \Delta^+}\in \Sigma_{\unl{d}}$,
the associated summand corresponds to the case when for all $\beta\in \Delta^+$
and $1\leq s\leq d_\beta$, the $(\sum_{\beta'<\beta}d_{\beta'}+s)$-th factor of
the corresponding term of $\Psi(e_h)$ is evaluated at
  $\vv^{-j(\beta)}y_{\beta,\sigma_\beta(s)},\ldots,\vv^{-i(\beta)}y_{\beta,\sigma_\beta(s)}$.
The image of this summand under $\phi_{\unl{d}}$ equals
  $$\prod_{\beta<\beta'} G_{\beta,\beta'}\cdot
    \prod_{\beta}G_\beta\cdot \prod_{\beta} G_\beta^{(\sigma_\beta)}$$
(up to a common sign and an integer power of $\vv$), where
\begin{equation}\label{explicit factors}
\begin{split}
  & G_{\beta,\beta'}\ =
    \prod_{1\leq s\leq d_\beta}^{1\leq s'\leq d_{\beta'}}
    \left(\prod_{j\in[\beta],j'\in [\beta']}^{j=j'}(y_{\beta,s}-\vv^{-2}y_{\beta',s'})\ \cdot
          \prod_{j\in[\beta],j'\in [\beta']}^{j=j'+1}(y_{\beta,s}-\vv^{2}y_{\beta',s'})\right)\times\\
  & \ \ \ \ \ \ \ \ \ \ \prod_{1\leq s\leq d_\beta}^{1\leq s'\leq d_{\beta'}}
    (y_{\beta,s}-y_{\beta',s'})^{\delta_{j(\beta')>j(\beta)}\delta_{i(\beta)+1\in [\beta']}},\\
  & G_\beta=(1-\vv^2)^{d_\beta(i(\beta)-j(\beta))}\ \cdot
    \prod_{1\leq s\ne s'\leq d_{\beta}} (y_{\beta,s}-\vv^2 y_{\beta,s'})^{i(\beta)-j(\beta)}\ \cdot
    \prod_{1\leq s\leq d_\beta} y_{\beta,s}^{i(\beta)-j(\beta)},\\
  & G_\beta^{(\sigma_\beta)}=
    \prod_{s=1}^{d_\beta} y_{\beta,\sigma_\beta(s)}^{r_\beta(h,s)}\cdot
    \prod_{s<s'} \frac{y_{\beta,\sigma_\beta(s)}-\vv^{-2}y_{\beta,\sigma_\beta(s')}}
                      {y_{\beta,\sigma_\beta(s)}-y_{\beta,\sigma_\beta(s')}}.
\end{split}
\end{equation}
Here, the collection $\{r_\beta(h,1),\ldots,r_\beta(h,d_\beta)\}$ is obtained by
listing every $r\in \BZ$ with multiplicity $h(\beta,r)>0$ with respect to the
total order $\preceq_\beta$ on $\BZ$, see Section~\ref{ssec formulation main thm 1}.
We also use the standard \emph{delta function} notation:
  $\delta_{\mathrm{condition}}=
   \begin{cases}
     1, & \text{if } \mathrm{condition}\ \text{holds} \\
     0, & \text{if } \mathrm{condition}\ \text{fails}
   \end{cases}.$

Note that the factors $\{G_{\beta,\beta'}\}_{\beta<\beta'}\cup\{G_\beta\}_{\beta}$
in~(\ref{explicit factors}) are independent of
$(\sigma_\beta)_{\beta\in \Delta^+}\in \Sigma_{\unl{d}}$.
Therefore, the specialization $\phi_{\unl{d}}(\Psi(e_h))$ has the following form:
\begin{equation}\label{explicit formula for same degrees}
  \phi_{\unl{d}}(\Psi(e_h))=c\ \cdot
  \prod_{\beta,\beta'\in \Delta^+}^{\beta<\beta'} G_{\beta,\beta'}\ \cdot
  \prod_{\beta\in \Delta^+}G_\beta\cdot
  \prod_{\beta\in \Delta^+}\left(\sum_{\sigma_\beta\in \Sigma_{d_{\beta}}} G_\beta^{(\sigma_\beta)}\right)
  \ \, \mathrm{with}\, \ c\in \BC^\times\cdot \vv^\BZ.
\end{equation}

For any $\beta\in \Delta^+$, we note that the sum
  $\sum_{\sigma_\beta\in \Sigma_{d_{\beta}}} G_\beta^{(\sigma_\beta)}$
coincides (up to a factor of $\BC^\times$) with the value of the shuffle element
  $x^{r_\beta(h,1)}\star\cdots \star x^{r_\beta(h,d_\beta)}\in S^{(2)}_{d_\beta}$
(in the shuffle algebra $S^{(2)}$!) evaluated at $\{y_{\beta,s}\}_{s=1}^{d_\beta}$.
The latter elements are linearly independent, due to Lemma~\ref{n=1 case}.

Thus,~(\ref{explicit formula for same degrees}) together with the above observation
complete our proof of Lemma~\ref{same degrees}.
\end{proof}

Now we can complete our proof of the linear independence of the ordered
PBWD monomials $\{e_h\}_{h\in H}$. Assume the contrary, that a nontrivial
linear combination $\sum_{h\in H} c_h e_h$ vanishes in $U_\vv^>(L\ssl_n)$
(here, all but finitely many of $c_h$ are zero, but at least one of them is non-zero).
Define $\unl{d}:=\max\{\deg(h)|c_h\ne 0\}$. Applying the specialization map
$\phi_{\unl{d}}$ to $\sum_{h\in H} c_h \Psi(e_h)=0$, we get
$\sum_{h\in H}^{\deg(h)=\unl{d}} c_h \phi_{\unl{d}}(\Psi(e_h))=0$ by
Lemma~\ref{lower degrees}. Furthermore, we get $c_h=0$ for all $h\in H$
of degree $\deg(h)=\unl{d}$, due to Lemma~\ref{same degrees}.
This contradicts our choice of $\unl{d}$.

\begin{Rem}
The machinery of the \emph{specialization maps} $\phi_{\unl{d}}$ that was used in the above
proof is of its own interest (cf.~\cite[(1.4)]{fhhsy} and~\cite[(4.24)]{n}).
\end{Rem}


\subsubsection{Spanning property of $e_h$ and dominance property of the specialization maps}\label{ssec spanning prop}
\

Let $M\subset S^{(n)}$ be the $\BC(\vv)$-span of $\{\Psi(e_h)\}_{h\in H}$.
For $\unl{k}\in \BN^I$, let $T_{\unl{k}}$ be a finite set consisting of all degree vectors
$\unl{d}=\{d_\beta\}_{\beta\in \Delta^+}\in \BN^{\frac{n(n-1)}{2}}$ such that
  $\sum_{\beta\in \Delta^+} d_\beta \beta = \sum_{i\in I} k_i\alpha_i$.
We order $T_{\unl{k}}$ with respect
to the anti-lexicographical order on $\BN^{\frac{n(n-1)}{2}}$. In particular,
the minimal element $\unl{d}_{\min}=\{d_\beta\}_{\beta\in \Delta^+}\in T_{\unl{k}}$
is characterized by $d_\beta=0$ for all non-simple roots $\beta\in \Delta^+$.

\begin{Lem}\label{spanning}
Let $F\in S^{(n)}_{\unl{k}}$ and $\unl{d}\in T_{\unl{k}}$.
If $\phi_{\unl{d}'}(F)=0$ for all $\unl{d}'\in T_{\unl{k}}$ such that $\unl{d}'>\unl{d}$,
then there exists an element $F_{\unl{d}}\in M$ such that
$\phi_{\unl{d}}(F)=\phi_{\unl{d}}(F_{\unl{d}})$ and
$\phi_{\unl{d}'}(F_{\unl{d}})=0$ for all $\unl{d}'>\unl{d}$.
\end{Lem}

\begin{proof}
Consider the following total order on the set
$\{(\beta,s)| \beta\in \Delta^+, 1\leq s\leq d_\beta\}$:
\begin{equation}\label{ordering of pairs}
 (\beta,s)\leq(\beta',s')\ \mathrm{iff}\
 \beta<\beta'\ \mathrm{or}\ \beta=\beta',s\leq s'.
\end{equation}

First, we note that the wheel conditions~(\ref{wheel conditions}) for $F$ guarantee
that $\phi_{\unl{d}}(F)$ (which is a Laurent polynomial in $\{y_{\beta,s}\}$)
vanishes up to appropriate orders under the following specializations:
\begin{enumerate}
\item[(i)] $y_{\beta,s}=\vv^{-2}y_{\beta',s'}$ for $(\beta,s)<(\beta',s')$,

\item[(ii)]$y_{\beta,s}=\vv^2y_{\beta',s'}$ for $(\beta,s)<(\beta',s')$.
\end{enumerate}
\noindent
The orders of vanishing are computed similarly to~\cite{fhhsy,n}. Explicitly,
let us view the specialization appearing in the definition of $\phi_{\unl{d}}$
as a step-by-step specialization in each interval $[\beta]$, ordered first in the
non-increasing length order, while the intervals of the same length are ordered
in the non-decreasing order of $j(\beta)$. As we specialize the variables in
the $s$-th interval ($1\leq s\leq \sum_{\beta\in \Delta^+} d_\beta$), we count
only those wheel conditions that arise from the non-specialized yet variables.
A straightforward case-by-case verification\footnote{This can be checked by
treating each of the following cases separately:
  $j=j'=i=i'$, $j=j'=i<i'$, $j=j'<i=i'$, $j=j'<i<i'$, $j<j'\leq i'<i$, $j<j'<i'=i$,
  $j<j'<i<i'$, $j=i<j'=i'$, $j=i<j'<i'$, $j<j'=i=i'$, $j<j'=i<i'$, $j<i<j'\leq i'$,
  where we set $j:=j(\beta),j':=j(\beta'),i:=i(\beta),i':=i(\beta')$.}
shows that the corresponding orders of vanishing under the specializations
(i) and (ii) equal
$\#\{(j,j')\in [\beta]\times [\beta']|j=j'\}-\delta_{\beta,\beta'}$
and $\#\{(j,j')\in [\beta]\times [\beta']|j=j'+1\}$, respectively.

Second, we claim that $\phi_{\unl{d}}(F)$ vanishes under the following specializations:
\begin{enumerate}
\item[(iii)] $y_{\beta,s}=y_{\beta',s'}$ for $(\beta,s)<(\beta',s')$
             such that $j(\beta)<j(\beta')$ and $i(\beta)+1\in [\beta']$.
\end{enumerate}
Indeed, if $j(\beta)<j(\beta')$ and $i(\beta)+1\in [\beta']$, there are positive
roots $\gamma,\gamma'\in \Delta^+$ such that
  $j(\gamma)=j(\beta), i(\gamma)=i(\beta'),
   j(\gamma')=j(\beta'), i(\gamma')=i(\beta)$.
Consider the degree vector $\unl{d}'\in T_{\unl{k}}$ given by
  $d'_{\alpha}=
   d_{\alpha}+\delta_{\alpha,\gamma}+\delta_{\alpha,\gamma'}-\delta_{\alpha,\beta}-\delta_{\alpha,\beta'}$.
Then, $\unl{d}'>\unl{d}$ and thus $\phi_{\unl{d}'}(F)=0$. The result follows.

Combining the above vanishing conditions for $\phi_{\unl{d}}(F)$, we see that
it is divisible exactly by the product
$\prod_{\beta<\beta'} G_{\beta,\beta'}\cdot \prod_{\beta}G_\beta$
of~(\ref{explicit factors}). Therefore, we have
\begin{equation}\label{woops 1}
  \phi_{\unl{d}}(F)\ =
  \prod_{\beta,\beta'\in \Delta^+}^{\beta<\beta'} G_{\beta,\beta'}\ \cdot
  \prod_{\beta\in\Delta^+}G_\beta\cdot G
\end{equation}
for some symmetric Laurent polynomial
\begin{equation}\label{woops 2}
  G\in \BC(\vv)[\{y^{\pm 1}_{\beta,s}\}_{\beta\in \Delta^+}^{1\leq s\leq d_\beta}]^{\Sigma_{\unl{d}}}
  \simeq \bigotimes_{\beta\in \Delta^+} \BC(\vv)[\{y^{\pm 1}_{\beta,s}\}_{s=1}^{d_\beta}]^{\Sigma_{d_\beta}}.
\end{equation}

Combining~(\ref{woops 1},~\ref{woops 2}) with formula~(\ref{explicit formula for same degrees})
and the discussion after it, we see that there is a (unique) linear combination
$F_{\unl{d}}=\sum_{h\in H}^{\deg(h)=\unl{d}} c_h \Psi(e_h)$ with $c_h\in \BC(\vv)$ such that
$\phi_{\unl{d}}(F)=\phi_{\unl{d}}(F_{\unl{d}})$, due to Lemma~\ref{n=1 case}.
The equality $\phi_{\unl{d}'}(F_{\unl{d}})=0$ for $\unl{d}'>\unl{d}$ is due
to Lemma~\ref{lower degrees}.

This completes our proof of Lemma~\ref{spanning}.
\end{proof}

Using Lemma above, we can now show that any shuffle element $F\in S^{(n)}_{\unl{k}}$
belongs to $M\cap S^{(n)}_{\unl{k}}$. Let $\unl{d}_{\max}$ and $\unl{d}_{\min}$
denote the maximal and the minimal elements of $T_{\unl{k}}$, respectively.
The condition of Lemma~\ref{spanning} is vacuous for $\unl{d}=\unl{d}_{\max}$.
Therefore, Lemma~\ref{spanning} applies. Applying it iteratively, we eventually
find an element $\wt{F}\in M$ such that $\phi_{\unl{d}}(F)=\phi_{\unl{d}}(\wt{F})$
for all $\unl{d}\in T_{\unl{k}}$. In the particular case of $\unl{d}=\unl{d}_{\min}$,
this yields $F=\wt{F}$, cf.~Example~\ref{example specialization maps}(d). Hence, $F\in M$.

\medskip
Invoking the injectivity of $\Psi$ (Proposition~\ref{simple shuffle}), we thus see that
$\{e_h\}_{h\in H}$ span $U_\vv^>(L\ssl_n)$. Combining this with the linear independence
of $\{e_h\}_{h\in H}$ established in Section~\ref{ssec linear indep} completes our proof of
Theorem~\ref{Main Theorem 1}(a). We note that the result of Theorem~\ref{Main Theorem 1}(b)
follows as well, as explained in the beginning of Section~\ref{sec proofs of Theorems 1,2}.

\begin{Rem}\label{easier hard shuffle}
The above argument also implies the surjectivity of $\Psi$. Combining this
with the injectivity of $\Psi$, established in Proposition~\ref{simple shuffle},
we obtain a new proof of Theorem~\ref{hard shuffle}.
\end{Rem}

\begin{Rem}
We note that the above argument actually provides a much bigger class of the PBWD bases
for $U^>_\vv(L\ssl_n)$, with the PBWD basis elements given rather in the shuffle form.
\end{Rem}

\begin{Rem}\label{Negut's pbw}
In~\cite{n}, the shuffle realization of the quantum toroidal algebra
$U_{\vv,\bar{\vv}}(\ddot{\gl}_n)$ (which is an associative $\BC(\vv,\bar{\vv})$-algebra
with $\vv,\bar{\vv}$ being two independent formal variables) was established
by studying the crucial \emph{slope $\leq \mu$ subalgebras}. In particular,
combining the proofs of Proposition 3.9 and Lemma 3.14 of~\emph{loc.cit.},
one obtains the PBWD basis of $U_{\vv,\bar{\vv}}(\ddot{\gl}_n)$ with the PBWD basis
elements given explicitly in the shuffle realization, see elements $E^\mu_{[j;i)}$
of~\cite[(3.46)]{n}. This gives rise to the PBWD basis of $U_\vv(L\ssl_n)$ by
 viewing the latter as a ``vertical'' subalgebra of $U_{\vv,\bar{\vv}}(\ddot{\gl}_n)$.
The corresponding PBWD basis elements are given by
  $\Psi^{-1}((1-\vv^2)^{i-j}\frac{p(x_{j,1},\ldots,x_{i,1})}{(x_{j,1}-x_{j+1,1})\cdots(x_{i-1,1}-x_{i,1})})$,
where
  $p(x_{j,1},\ldots,x_{i,1})=\prod_{a=j}^i x_{a,1}^{\lfloor\mu(a-j+1)\rfloor-\lfloor\mu(a-j)\rfloor}$
with $1\leq j\leq i<n$ and $\mu\in \frac{1}{i-j+1}\BZ$.
As we fix $1\leq j\leq i<n$ and let $\mu$ run over $\frac{1}{i-j+1}\BZ$, the degree of $p$
varies over $\BZ$ multiplicity-free. Comparing this to Lemma~\ref{shuffle root elt},
we see that the corresponding PBWD basis of $U^>_\vv(L\ssl_n)$ is reminiscent to
a particular one of Theorem~\ref{Main Theorem 1}(a), but it should be stressed right away that
the former utilizes a total order on $\Delta^+\times \BZ$ that is different
from~(\ref{extended order}) used in~(\ref{ordered}).
\end{Rem}


\subsection{Integral form $\fS^{(n)}$}\label{ssec integral shuffle algebra}
\

For $\unl{k}\in \BN^I$, set $|\unl{k}|:=\sum_{i=1}^{n-1} k_i$.
Consider a $\BC[\vv,\vv^{-1}]$-submodule $\wt{\fS}^{(n)}_{\unl{k}}\subset S^{(n)}_{\unl{k}}$ 
consisting of those $F\in S^{(n)}_{\unl{k}}$ for which the numerator $f$ in~(\ref{pole conditions})
is of the form:
\begin{equation}
\label{eqn:integral f}
  f\in (\vv-\vv^{-1})^{|\unl{k}|} \cdot 
  \BC[\vv,\vv^{-1}][\{x_{i,r}^{\pm 1}\}_{i\in I}^{1\leq r\leq k_i}]^{\Sigma_{\unl{k}}}.
\end{equation}
Then
  $\wt{\fS}^{(n)}:=\underset{\unl{k}\in \BN^I}\bigoplus \wt{\fS}^{(n)}_{\underline{k}}$
is clearly a $\BC[\vv,\vv^{-1}]$-subalgebra of $S^{(n)}$.

\begin{Rem}
If one wishes to work with forms defined over $\BZ[\vv,\vv^{-1}]$, formula~(\ref{eqn:integral f}) 
should be replaced accordingly. To this end, we note that the shuffle product $F\star G$ 
of~(\ref{shuffle product}) still makes sense as $\frac{1}{\unl{k}!\cdot \unl{\ell}!}\cdot \Sym(\dots)$
may be written simply as a sum over so-called $(\unl{k},\unl{\ell})$-shuffles, since $F$ and $G$ are symmetric. 
We leave details to the interested reader.
\end{Rem}

Moreover, $\Psi(\fU^>_\vv(L\ssl_n))\subset \wt{\fS}^{(n)}$, due to Lemma~\ref{shuffle root elt}.
The key goal of this section is to explicitly describe the image $\Psi(\fU^>_\vv(L\ssl_n))$.

\begin{Ex}
Consider $F(x_1,x_2)=(\vv-\vv^{-1})^2(x_1x_2)^r\in \wt{\fS}^{(2)}_2$ with $r\in \BZ$.
Then, we have $[2]_\vv\cdot F\in \Psi(\fU^>_\vv(L\ssl_2))$, but
$F\notin \Psi(\fU^>_\vv(L\ssl_2))$, as follows from
Lemmas~\ref{necessity for n=2},~\ref{integral n=1 case} below.
\end{Ex}

Pick $F\in \wt{\fS}^{(n)}_{\unl{k}}$. For any
  $\unl{d}=\{d_\beta\}_{\beta\in \Delta^+}\in \BN^{\frac{n(n-1)}{2}}$
such that
  $\sum_{\beta\in \Delta^+} d_\beta \beta=\sum_{i\in I} k_i\alpha_i$,
consider
  $\phi_{\unl{d}}(F)\in
   \BC[\vv,\vv^{-1}][\{y_{\beta,s}^{\pm 1}\}_{\beta\in \Delta^+}^{1\leq s\leq d_\beta}]^{\Sigma_{\unl{d}}}$
of~(\ref{specialization map}).
First, we note that $\phi_{\unl{d}}(F)$ is divisible by
\begin{equation}\label{factor A}
  A_{\unl{d}}:=(\vv-\vv^{-1})^{|\unl{k}|}.
\end{equation}
Second, following the first part of the proof of Lemma~\ref{spanning},
$\phi_{\unl{d}}(F)$ is also divisible by
\begin{equation}\label{factor B}
\begin{split}
  & B_{\unl{d}}:=
  \prod_{(\beta,s)<(\beta',s')}
  (y_{\beta,s}-\vv^{-2}y_{\beta',s'})^{\#\{(j,j')\in [\beta]\times [\beta']|j=j'\}-\delta_{\beta,\beta'}}\ \times \\
  & \ \ \qquad \prod_{(\beta,s)<(\beta',s')}(y_{\beta,s}-\vv^{2}y_{\beta',s'})^{\#\{(j,j')\in [\beta]\times [\beta']|j=j'+1\}}
\end{split}  
\end{equation}
due to the wheel conditions~(\ref{wheel conditions}), where we use the total order~(\ref{ordering of pairs}) 
on the set $\{(\beta,s)\}_{\beta\in \Delta^+}^{1\leq s\leq d_\beta}$. 
Combining these observations, we define the \emph{reduced specialization map}
\begin{equation}\label{reduced specialization map}
  \varphi_{\unl{d}}\colon \wt{\fS}^{(n)}_{\unl{k}}\longrightarrow
  \BC[\vv,\vv^{-1}][\{y_{\beta,s}^{\pm 1}\}_{\beta\in \Delta^+}^{1\leq s\leq d_\beta}]^{\Sigma_{\unl{d}}}
  \quad \mathrm{via}\quad
  \varphi_{\unl{d}}(F):=\frac{\phi_{\unl{d}}(F)}{A_{\unl{d}}B_{\unl{d}}}.
\end{equation}

Let us introduce another type of specialization maps.
Given a collection of positive integers
  $\unl{t}=\{t_{\beta,r}\}_{\beta\in \Delta^+}^{1\leq r\leq \ell_\beta}$
($\ell_\beta\in \BN$), define a degree vector
$\unl{d}=\{d_\beta\}_{\beta\in \Delta^+}\in \BN^{\frac{n(n-1)}{2}}$ via
$d_\beta:=\sum_{r=1}^{\ell_\beta} t_{\beta,r}$.
Let us now define the \emph{vertical specialization map}
\begin{equation}\label{vertical specialization map}
  \varpi_{\unl{t}}\colon
  \BC[\vv,\vv^{-1}][\{y_{\beta,s}^{\pm 1}\}_{\beta\in \Delta^+}^{1\leq s\leq d_\beta}]^{\Sigma_{\unl{d}}}
  \longrightarrow \BC[\vv,\vv^{-1}][\{z_{\beta,r}^{\pm 1}\}_{\beta\in \Delta^+}^{1\leq r\leq \ell_\beta}].
\end{equation}
For each $\beta\in \Delta^+$, we split the variables $\{y_{\beta,s}\}_{s=1}^{d_\beta}$
into $\ell_\beta$ groups of size $t_{\beta,r}$ each ($1\leq r\leq \ell_\beta$)
and specialize the variables in the $r$-th group to
  $\vv^{-2}z_{\beta,r},\vv^{-4}z_{\beta,r},\vv^{-6}z_{\beta,r},\ldots,\vv^{-2t_{\beta,r}}z_{\beta,r}$.
For any
  $K\in
   \BC[\vv,\vv^{-1}][\{y_{\beta,s}^{\pm 1}\}_{\beta\in \Delta^+}^{1\leq s\leq d_\beta}]^{\Sigma_{\unl{d}}}$,
we define $\varpi_{\unl{t}}(K)$ as the corresponding specialization of $K$.
We note that $\varpi_{\unl{t}}(K)$ is independent of our splitting of the
variables $\{y_{\beta,s}\}_{s=1}^{d_\beta}$ into groups.

Finally, we shall combine~(\ref{reduced specialization map}) and~(\ref{vertical specialization map})
to obtain the key tool in the study of the integral form $\Psi(\fU^>_\vv(L\ssl_n))$ of $S^{(n)}$.
Given $\unl{k}\in \BN^I$, $\unl{d}=\{d_\beta\}_{\beta\in \Delta^+}\in \BN^{\frac{n(n-1)}{2}}$
and a collection of positive integers
  $\unl{t}=\{t_{\beta,r}\}_{\beta\in \Delta^+}^{1\leq r\leq \ell_\beta}$
such that
\begin{equation}
\label{eqn:kdl-compatibility}
  \sum_{\beta\in \Delta^+} d_\beta \beta = \sum_{i\in I} k_i\alpha_i 
  \qquad \mathrm{and}\qquad
  \sum_{r=1}^{\ell_\beta} t_{\beta,r}=d_\beta,
\end{equation}  
we define the \textbf{cross specialization map}
\begin{equation}\label{cross specialization map}
  \Upsilon_{\unl{d},\unl{t}}\colon \wt{\fS}^{(n)}_{\unl{k}}\longrightarrow
  \BC[\vv,\vv^{-1}][\{z_{\beta,r}^{\pm 1}\}_{\beta\in \Delta^+}^{1\leq r\leq \ell_\beta}]
  \quad \mathrm{via}\quad \Upsilon_{\unl{d},\unl{t}}(F):=\varpi_{\unl{t}}(\varphi_{\unl{d}}(F)).
\end{equation}

\begin{Def}\label{integral element}
$F\in S^{(n)}_{\unl{k}}$ is \textbf{integral} if $F\in \wt{\fS}^{(n)}_{\unl{k}}$
and $\Upsilon_{\unl{d},\unl{t}}(F)$ is divisible by
$\prod_{\beta\in \Delta^+}^{1\leq r\leq \ell_\beta}[t_{\beta,r}]_\vv!$
(the product of $\vv$-factorials) for all possible $\unl{d}$ and $\unl{t}$
satisfying~(\ref{eqn:kdl-compatibility}).
\end{Def}

\begin{Ex}\label{explaining integrality for n=2}
In the simplest case $n=2$, a symmetric Laurent polynomial $F\in S^{(2)}_k$
is integral if and only if it has the form $F=(\vv-\vv^{-1})^k\cdot \bar{F}$
with $\bar{F}\in \BC[\vv,\vv^{-1}][\{x_p^{\pm 1}\}_{p=1}^k]^{\Sigma_k}$ satisfying
the following divisibility condition:
\begin{equation}\label{integrality condition for n=2}
   \bar{F}(\vv^{-2}z_1,\vv^{-4}z_1,\ldots,\vv^{-2k_1}z_1,\ldots,\vv^{-2}z_{\ell},\ldots,\vv^{-2k_{\ell}}z_{\ell})
   \ \ \mathrm{is\ divisible\ by}\ \ [k_1]_\vv!\cdots [k_{\ell}]_\vv!
\end{equation}
for any decomposition $k=k_1+\ldots+k_{\ell}$ into a sum of positive integers.
\end{Ex}

Let $\fS^{(n)}_{\unl{k}}\subset \wt{\fS}^{(n)}_{\unl{k}}$ denote the 
$\BC[\vv,\vv^{-1}]$-submodule of all integral elements and set
  $$\fS^{(n)}:=\underset{\unl{k}\in \BN^I}\bigoplus \fS^{(n)}_{\underline{k}}.$$
The following is the key result of this section:

\begin{Thm}\label{shuffle integral form}
The $\BC(\vv)$-algebra isomorphism $\Psi\colon U_\vv^{>}(L\ssl_n)\iso S^{(n)}$
of Theorem~\ref{hard shuffle} gives rise to a $\BC[\vv,\vv^{-1}]$-algebra
isomorphism $\Psi\colon \fU_\vv^{>}(L\ssl_n)\iso \fS^{(n)}$.
\end{Thm}

The proof of Theorem~\ref{shuffle integral form} is presented in
Section~\ref{ssec proof of Theorem 2}.

\begin{Cor}\label{useful corollary}
(a) $\fS^{(n)}$ is a $\BC[\vv,\vv^{-1}]$-subalgebra of $S^{(n)}$.

\noindent
(b) Theorem~\ref{Main Theorem 2}(a) holds, that is, the subalgebra
$\fU^>_\vv(L\ssl_n)$ is independent of all the choices.
\end{Cor}

The following two properties of the integral form $\fS^{(n)}$
are crucially used in~\cite{ft2}:

\begin{Prop}\label{results for FT2}
(a) Fix $1\leq \ell <n$ and consider the linear map
$\iota'_{\ell}\colon S^{(n)}\to S^{(n)}$ given by
\begin{equation}\label{shuffle shift homomorphism}
  \iota'_{\ell}(F)\left(\{x_{i,r}\}_{i\in I}^{1\leq r\leq k_i}\right):=
  \prod_{r=1}^{k_{\ell}}(1-x_{\ell,r}^{-1})\cdot F\left(\{x_{i,r}\}_{i\in I}^{1\leq r\leq k_i}\right)
  \ \ \mathrm{for}\ \ F\in S^{(n)}_{\unl{k}}, \unl{k}\in \BN^I.
\end{equation}
Then
\begin{equation}\label{shift homomorphisms}
  F\in \fS^{(n)}\Longleftrightarrow \iota'_{\ell}(F)\in \fS^{(n)}.
\end{equation}

\noindent
(b) For any $\unl{k}\in \BN^I$ and a collection
  $g_i(\{x_{i,r}\}_{r=1}^{k_i})\in
   \BC[\vv,\vv^{-1}][\{x_{i,r}^{\pm 1}\}_{r=1}^{k_i}]^{\Sigma_{k_i}}\ (\forall\, i\in I)$,
consider
\begin{equation}\label{for surjectivity in FT2}
  F=(\vv-\vv^{-1})^{|\unl{k}|}\cdot
  \frac{\prod_{i=1}^{n-1}\prod_{1\leq r\ne r'\leq k_i} (x_{i,r}-\vv^{-2}x_{i,r'})\cdot
        \prod_{i=1}^{n-1}g_i(\{x_{i,r}\}_{r=1}^{k_i})}
       {\prod_{i=1}^{n-2}\prod_{1\leq r\leq k_i}^{1\leq r'\leq k_{i+1}}(x_{i,r}-x_{i+1,r'})}.
\end{equation}
Then, $F$ is integral, i.e.\ $F\in \fS^{(n)}_{\unl{k}}$.
\end{Prop}

\begin{proof}
(a) Obvious from the above definition of the integral form $\fS^{(n)}$.

(b) The presence of the factor
  $\prod_{i=1}^{n-1}\prod_{1\leq r\ne r'\leq k_i} (x_{i,r}-\vv^{-2}x_{i,r'})$
in $F$ of~(\ref{for surjectivity in FT2}) guarantees that for any degree vector $\unl{d}=\{d_\beta\}_{\beta\in \Delta^+}$
satisfying $\sum_{\beta\in \Delta^+} d_\beta \beta = \sum_{i\in I} k_i\alpha_i$,
the reduced specialization $\varphi_{\unl{d}}(F)$ of~(\ref{reduced specialization map})
is divisible by
  $\prod_{\beta\in \Delta^+}\prod_{1\leq s\ne s'\leq d_\beta} (y_{\beta,s}-\vv^{-2}y_{\beta,s'})$.
Thus, the further specialization
  $\varpi_{\unl{t}}(\varphi_{\unl{d}}(F))=\Upsilon_{\unl{d},\unl{t}}(F)$
vanishes if at least one of $t_{\beta,r}$'s is greater than $1$.
Meanwhile, the divisibility condition of Definition~\ref{integral element}
is vacuous in the remaining case when all $t_{\beta,r}=1$.
Therefore, $F\in \fS^{(n)}_{\unl{k}}$ as claimed.
\end{proof}


\subsection{Proofs of Theorem~\ref{Main Theorem 2}(b) and Theorem~\ref{shuffle integral form}}
\label{ssec proof of Theorem 2}
\

We note that both Theorem~\ref{Main Theorem 2}(b) and Theorem~\ref{shuffle integral form}
follow from the following two results:
\begin{enumerate}
 \item[(I)]
 For any $k\geq 1, \{\beta_p\}_{p=1}^k\subset \Delta^+, \{r_p\}_{p=1}^k\subset \BZ$,
 we have $\Psi(\wt{e}_{\beta_1}(r_1)\cdots \wt{e}_{\beta_k}(r_k))\in \fS^{(n)}$.

 \item[(II)]
 Any element $F\in \fS^{(n)}$ may be written as a $\BC[\vv,\vv^{-1}]$-linear
 combination of $\{\Psi(\wt{e}_h)\}_{h\in H}$.
\end{enumerate}
The proof of~(I) is straightforward and will be easily deduced from
our definition of $\fS^{(n)}$, while the proof of (II) will follow from
Lemma~\ref{spanning} and the validity of~(II) for $n=2$, see~Lemma~\ref{integral n=1 case}.

We start by establishing both (I) and (II) for $n=2$ in Section~\ref{ssec proof baby case integral}.


\subsubsection{$n=2$ case}\label{ssec proof baby case integral}
\

For $n=2$, the description of the integral form $\fS^{(n)}\subset S^{(n)}$
is the simplest, see Example~\ref{explaining integrality for n=2}.
Set $\wt{e}_r:=(\vv-\vv^{-1})e_r\in U^>_\vv(L\ssl_2)$ for $r\in \BZ$.
The following result establishes~(I) for $n=2$:

\begin{Lem}\label{necessity for n=2}
For any $k\geq 1$ and $r_1,\ldots,r_k\in \BZ$, we have
$\Psi(\wt{e}_{r_1}\cdots \wt{e}_{r_k})\in \fS^{(2)}_k$.
\end{Lem}

\begin{proof}
Pick any decomposition $k=k_1+\ldots+k_{\ell}$ with all $k_p\geq 1$.
We claim that as we specialize the variables $x_1,\ldots,x_k$ to
$\{\vv^{-2r}z_p\}_{1\leq p\leq \ell}^{1\leq r\leq k_p}$, the image
of any summand of the symmetrization appearing in
$\Psi(e_{r_1}\cdots e_{r_k})\in S^{(2)}_k$ is divisible by
the product $\prod_{p=1}^{\ell} [k_p]_\vv!$ of $\vv$-factorials.

To this end, let us fix $1\leq p\leq \ell$ and consider the relative
position of the variables $\vv^{-2}z_p,\vv^{-4}z_p,\ldots,\vv^{-2k_p}z_p$.
If there is an index $1\leq r<k_p$ such that $\vv^{-2(r+1)}z_p$ is placed
to the left of $\vv^{-2r}z_p$, then the above specialization of the
corresponding $\zeta$-factor equals
  $\frac{\vv^{-2(r+1)}z_p-\vv^{-2}\cdot \vv^{-2r}z_p}{\vv^{-2(r+1)}z_p-\vv^{-2r}z_p}=0$.
However, if
  $\vv^{-2r}z_p$ stays to the left of $\vv^{-2(r+1)}z_p$
for all $1\leq r<k_p$, then the total contribution of the specializations
of the corresponding $\zeta$-factors equals
\begin{equation}
\label{eqn:factorial reappearance}
  \prod_{1\leq r<s\leq k_p}
  \frac{\vv^{-2r}z_p-\vv^{-2}\cdot \vv^{-2s}z_p}{\vv^{-2r}z_p-\vv^{-2s}z_p}=
  \vv^{-\frac{k_p(k_p-1)}{2}}[k_p]_\vv!
\end{equation}
Combining this over all $1\leq p\leq \ell$, we see that $\prod_{p=1}^{\ell} [k_p]_\vv!$
indeed divides the above specialization of $\Psi(e_{r_1}\cdots e_{r_k})$.
This completes our proof of Lemma~\ref{necessity for n=2}.
\end{proof}

For simplicity of the exposition, we will assume now that
the total order $\preceq$ on $\BZ$ is the usual one $\leq$.
The following result establishes~(II) for $n=2$:

\begin{Lem}\label{integral n=1 case}
Any symmetric Laurent polynomial
  $\bar{F}\in \BC[\vv,\vv^{-1}][\{x_p^{\pm 1}\}_{p=1}^k]^{\Sigma_k}$
satisfying the divisibility condition~(\ref{integrality condition for n=2})
may be written as a $\BC[\vv,\vv^{-1}]$-linear combination of $\{\Psi(e_h)\}_{h\in H}$.
\end{Lem}

\begin{proof}
We may assume that $\bar{F}$ is homogeneous of the total degree $N$.
Let $V_N$ denote the set of all ordered $k$-tuples of integers
$\unl{r}=(r_1,r_2,\ldots,r_k), r_1\leq\cdots\leq r_k$, such that $r_1+\ldots+r_k=N$.
This set is totally ordered with respect to the lexicographical order:
\begin{equation*}
  \unl{r}\prec \unl{r}'
  \ \ \mathrm{iff} \ \ \exists\, 1\leq p\leq k \
  \mathrm{s.t.}\ r_p < r'_p \ \mathrm{and}\
  r_s=r'_s\ \forall\, s>p.
\end{equation*}

Let us present $\bar{F}$ as a linear combination of the monomial symmetric polynomials:
\begin{equation*}
  \bar{F}(x_1,\ldots,x_k)=\sum_{\unl{r}\in V_N} \mu_{\unl{r}}m_{\unl{r}}(x_1,\ldots,x_k)
  \ \ \mathrm{with}\ \  \mu_{\unl{r}}\in \BC[\vv,\vv^{-1}].
\end{equation*}
Pick the maximal element $\unl{r}_{\max}=(r_1,\ldots,r_k)$ of
the finite set $V_N(\bar{F}):=\{\unl{r}\in V_N|\mu_{\unl{r}}\ne 0\}$
and consider a decomposition $k=k_1+\ldots+k_{\ell}$ such that
\begin{equation*}
  r_1=\cdots=r_{k_1}<r_{k_1+1}=\cdots=r_{k_1+k_2}<\cdots<r_{k_1+\ldots+k_{\ell-1}+1}=\cdots=r_k.
\end{equation*}
Evaluating $\bar{F}$ at the corresponding specialization
$\{\vv^{-2s}z_p\}_{1\leq p\leq \ell}^{1\leq s\leq k_p}$,
we see that the coefficient of the lexicographically largest monomial
in the variables $\{z_p\}_{p=1}^{\ell}$ equals $\mu_{\unl{r}_{\max}}$,
up to an integer power of $\vv$.
Therefore, the divisibility condition~(\ref{integrality condition for n=2}) implies that
\begin{equation}
\label{eqn:key integrality}
  \frac{\mu_{\unl{r}_{\max}}}{\prod_{p=1}^{\ell} [k_p]_\vv!}\in \BC[\vv,\vv^{-1}].
\end{equation}
Set $\bar{F}^{(0)}:=\bar{F}$ and define
$\bar{F}^{(1)}\in \BC[\vv,\vv^{-1}][\{x_p^{\pm 1}\}_{p=1}^k]^{\Sigma_k}$ via
\begin{equation}\label{iterated sequence}
  \bar{F}^{(1)}:=\bar{F}^{(0)}-
  \vv^{\sum_{p=1}^{\ell} k_p(k_p-1)/2}
  \frac{\mu_{\unl{r}_{\max}}}{\prod_{p=1}^{\ell} [k_p]_\vv!}\Psi(e_{r_1}\cdots e_{r_k}).
\end{equation}
We note that $\bar{F}^{(1)}$ satisfies~(\ref{integrality condition for n=2}),
due to~(\ref{eqn:key integrality}) and  Lemma~\ref{necessity for n=2}. Applying the same argument 
to $\bar{F}^{(1)}$ in place of $F=\bar{F}^{(0)}$, we obtain $\bar{F}^{(2)}$ also 
satisfying~(\ref{integrality condition for n=2}).
Proceeding further, we thus construct a sequence of symmetric Laurent polynomials
$\{\bar{F}^{(s)}\}_{s\in \BN}$ satisfying~(\ref{integrality condition for n=2}).

According to our proof of Lemma~\ref{n=1 case}, especially formula~(\ref{eqn:explicit mu value}),
the sequence $\unl{r}^{(i)}_{\max}\in V_N$ of the maximal elements of
$V_N(\bar{F}^{(i)})$ strictly decreases. Meanwhile, the sequence of the minimal
powers of any variable in $\bar{F}^{(s)}$ is a non-decreasing sequence.
Hence, $\bar{F}^{(s)}=0$ for some $s\in \BN$.

This completes our proof of Lemma~\ref{integral n=1 case}.
\end{proof}


\subsubsection{General case}\label{ssec proof general case integral}
\

Let us now generalize the arguments of Section~\ref{ssec proof baby case integral}
to prove (I) and (II) for any $n>2$. The proof of the former is quite similar
(though is more elaborate) to that of Lemma~\ref{necessity for n=2}:

\begin{Lem}\label{necessity for n>2}
$\Psi(\wt{e}_{\beta_1}(r_1)\cdots \wt{e}_{\beta_{m}}(r_{m}))\in \fS^{(n)}$ for any
$m\geq 1, \{\beta_p\}_{p=1}^{m}\subset \Delta^+, \{r_p\}_{p=1}^{m}\subset \BZ$.
\end{Lem}

\begin{proof}
Define $\unl{k}\in \BN^I$ via
$\sum_{i\in I} k_i\alpha_i = \sum_{p=1}^{m} \beta_p$, so that
  $F:=\Psi(\wt{e}_{\beta_1}(r_1)\cdots \wt{e}_{\beta_{m}}(r_{m}))
   \in S^{(n)}_{\unl{k}}$.
First, we note that $F$ is divisible by $(\vv-\vv^{-1})^{|\unl{k}|}$,
due to Lemma~\ref{shuffle root elt}. Therefore, $F\in \wt{\fS}^{(n)}_{\unl{k}}$.

It remains to show that $\Upsilon_{\unl{d},\unl{t}}(F)$ satisfies
the divisibility condition of Definition~\ref{integral element} for any
$\unl{d}=\{d_\beta\}_{\beta\in \Delta^+}\in \BN^{\frac{n(n-1)}{2}}$
and a collection of positive integers
$\unl{t}=\{t_{\beta,s}\}_{\beta\in \Delta^+}^{1\leq s\leq \ell_\beta}$
satisfying~(\ref{eqn:kdl-compatibility}). To this end, recall that 
the cross specialization $\Upsilon_{\unl{d},\unl{t}}(F)$ is computed in three steps:

\begin{itemize}

 \item[--] 
 first, we specialize $x_{\ast,\ast}$-variables in $f$ of~(\ref{pole conditions})
 to $\vv^{?}$-multiples of $y_{\ast,\ast}$-variables as in~(\ref{specialization map});

 \item[--] 
 second, we divide that specialization by the product of appropriate powers
 of $(\vv-\vv^{-1})$ and linear terms in $y_{\ast,\ast}$-variables arising
 via the wheel conditions, see~(\ref{factor A},~\ref{factor B})
 and~(\ref{reduced specialization map});

 \item[--] 
 finally, we specialize $y_{\ast,\ast}$-variables to $\vv^{?}$-multiples
 of $z_{\ast,\ast}$-variables as in~(\ref{vertical specialization map}).

\end{itemize}

\noindent
Fix $\beta\in \Delta^+$ and $1\leq s\leq \ell_\beta$, and consider those
$x_{\ast,\ast}$ that eventually got specialized to $\vv^{?}z_{\beta,s}$.
Without loss of generality, we may assume those are
$\{x_{i,r}\}_{j(\beta)\leq i\leq i(\beta)}^{1\leq r\leq t_{\beta,s}}$.
We may also assume that $x_{i,r}$ was specialized to $\vv^{-i}y_{\beta,r}$
under the first specialization~(\ref{specialization map}), while
$y_{\beta,r}$ was specialized to $\vv^{-2r}z_{\beta,s}$ under
the second specialization~(\ref{vertical specialization map}), for any
$j(\beta)\leq i\leq i(\beta),\ 1\leq r\leq t_{\beta,s}$.

For $j(\beta)\leq i<i(\beta)$ and $1\leq r\ne r'\leq t_{\beta,s}$, consider
the relative position of the variables $x_{i,r}, x_{i,r'}, x_{i+1,r'}$.
As $x_{i,r}, x_{i,r'}$ cannot enter the same function $\Psi(\wt{e}_\ast(\ast))$,
$x_{i,r}$ is placed either to the left of $x_{i,r'}$ or to the right.
In the former case, we gain the factor $\zeta_{i,i}(x_{i,r}/x_{i,r'})$, which upon
the specialization $\phi_{\unl{d}}$ contributes the factor $(y_{\beta,r}-\vv^{-2}y_{\beta,r'})$.
Likewise, if $x_{i+1,r'}$ is placed to the left of $x_{i,r}$, we gain the factor
$\zeta_{i+1,i}(x_{i+1,r'}/x_{i,r})$, which upon the specialization $\phi_{\unl{d}}$
contributes the factor $(y_{\beta,r}-\vv^{-2}y_{\beta,r'})$ as well. In the remaining
case, when $x_{i,r'}$ is to the left of $x_{i,r}$ while $x_{i+1,r'}$ is not,
we gain the factor $\zeta_{i,i+1}(x_{i,r'}/x_{i+1,r'})$,
which upon the specialization $\phi_{\unl{d}}$ specializes to $0$.
As $i$ ranges from $j(\beta)$ up to $i(\beta)-1$, we thus gain the
$(i(\beta)-j(\beta))$-th power of $(y_{\beta,r}-\vv^{-2}y_{\beta,r'})$.
Note that this power exactly coincides with the power of
$(y_{\beta,r}-\vv^{-2}y_{\beta,r'})$ in $B_{\unl{d}}$ of~(\ref{factor B}),
by which we divide $\phi_{\unl{d}}(F)$ to define the reduced specialization
$\varphi_{\unl{d}}(F)$ of~(\ref{reduced specialization map}).

However, we have not used yet the $\zeta$-factors
$\zeta_{i(\beta),i(\beta)}(x_{i(\beta),r}/x_{i(\beta),r'})$ for
$x_{i(\beta),r}$ placed to the left of $x_{i(\beta),r'}$. If there
is $1\leq r<t_{\beta,s}$ such that $x_{i(\beta),r+1}$ is placed
to the left of $x_{i(\beta),r}$, then
$\zeta_{i(\beta),i(\beta)}(x_{i(\beta),r+1}/x_{i(\beta),r})$ specializes
to zero upon~(\ref{vertical specialization map}). In the remaining case,
when each $x_{i(\beta),r}$ stays to the left of $x_{i(\beta),r+1}$,
the total contribution of the specializations of the corresponding $\zeta$-factors equals
$\vv^{-t_{\beta,s}(t_{\beta,s}-1)/2}[t_{\beta,s}]_\vv!$ as
in formula~(\ref{eqn:factorial reappearance}).

This completes our proof of Lemma~\ref{necessity for n>2}.
\end{proof}

Let $\wt{M}\subset S^{(n)}$ be the $\BC[\vv,\vv^{-1}]$-span of $\{\Psi(\wt{e}_h)\}_{h\in H}$.
Generalizing Lemma~\ref{spanning}, we have:

\begin{Lem}\label{spanning integral}
Let $F\in \fS^{(n)}_{\unl{k}}$ and $\unl{d}\in T_{\unl{k}}$.
If $\phi_{\unl{d}'}(F)=0$ for all $\unl{d}'\in T_{\unl{k}}$ such that $\unl{d}'>\unl{d}$,
then there exists an element $F_{\unl{d}}\in \wt{M}$ such that
$\phi_{\unl{d}}(F)=\phi_{\unl{d}}(F_{\unl{d}})$ and
$\phi_{\unl{d}'}(F_{\unl{d}})=0$ for all $\unl{d}'>\unl{d}$.
\end{Lem}

\begin{proof}
The proof of this lemma is completely analogous to the one of Lemma~\ref{spanning}.
More precisely, combining formulas~(\ref{woops 1},~\ref{woops 2}) and the condition
$F\in \fS^{(n)}_{\unl{k}}$ together with formula~(\ref{explicit formula for same degrees})
and the discussion following it, the result follows from its $n=2$ counterpart.
The latter has been already taken care of in Lemma~\ref{integral n=1 case}.
\end{proof}

Combining Lemmas~\ref{necessity for n>2},~\ref{spanning integral} with the argument 
following our proof of Lemma~\ref{spanning} completes our proof of both 
Theorems~\ref{Main Theorem 2}(b) and~\ref{shuffle integral form}.
We note that the result of Theorem~\ref{Main Theorem 2}(c)
follows as well, as explained in the beginning of Section~\ref{sec proofs of Theorems 1,2}.


\section{Generalizations to $U_{\vv_1,\vv_2}(L\ssl_n)$}
\label{sec 2-parametric quantuma affine}

The two-parameter quantum loop algebra $U_{\vv_1,\vv_2}(L\ssl_n)$ was introduced
in~\cite{hrz}\footnote{To be more precise, this recovers the algebra
of~\emph{loc.cit}.\ with the trivial central charges.}
as a generalization of $U_\vv(L\ssl_n)$ (one recovers the latter from the former
by setting $\vv_1=\vv_2^{-1}=\vv$ and identifying some Cartan elements,
see~\cite[Remark~3.3(4)]{hrz}). The key results of~\cite{hrz} are:
\begin{enumerate}
 \item[1)] the Drinfeld-Jimbo type realization of $U_{\vv_1,\vv_2}(L\ssl_n)$,
 see~\cite[Theorem 3.12]{hrz};

 \item[2)] the PBW basis of its subalgebras
 $U^>_{\vv_1,\vv_2}(L\ssl_n), U^<_{\vv_1,\vv_2}(L\ssl_n)$,
 see~\cite[Theorem 3.11]{hrz}.
\end{enumerate}
\noindent
However, the latter result (\cite[Theorem 3.11]{hrz}) is stated  without
any glimpse of a proof.

The primary goal of this section is to generalize Theorem~\ref{Main Theorem 1}
to the case of $U^>_{\vv_1,\vv_2}(L\ssl_n)$, thus proving~\cite[Theorem 3.11]{hrz}.
Along the way, we also generalize Theorem~\ref{hard shuffle} by providing the
shuffle realization of $U^>_{\vv_1,\vv_2}(L\ssl_n)$, which is of independent interest.
The latter is used to construct the PBWD bases for the integral form
of $U^>_{\vv_1,\vv_2}(L\ssl_n)$, generalizing Theorem~\ref{Main Theorem 2}.


\subsection{Two-parameter quantum loop algebra $U^>_{\vv_1,\vv_2}(L\ssl_n)$}
\label{ssec 2-parameter quantum}
\

For the purpose of this section, it suffices to work only with the
subalgebra $U^>_{\vv_1,\vv_2}(L\ssl_n)$ of $U_{\vv_1,\vv_2}(L\ssl_n)$.
Let $\vv_1,\vv_2$ be two independent formal variables and set
$\BK:=\BC(\vv_1^{1/2},\vv_2^{1/2})$. Following~\cite[Definition 3.1]{hrz},
define $U^>_{\vv_1,\vv_2}(L\ssl_n)$ to be the associative $\BK$-algebra
generated by $\{e_{i,r}\}_{i\in I}^{r\in \BZ}$
with the following defining relations:
\begin{equation}\label{Twoparam 1}
  (z - (\langle j,i\rangle \langle i,j\rangle)^{1/2} w)e_i(z)e_j(w)=
  (\langle j,i\rangle z - (\langle j,i\rangle \langle i,j\rangle^{-1})^{1/2} w)e_j(w)e_i(z)
\end{equation}
as well as Serre relations: 
\begin{equation}\label{Twoparam 2}
\begin{split}
  & e_i(z)e_j(w)=e_j(w)e_i(z)\ \mathrm{if}\ c_{ij}=0,\\
  & [e_i(z_1),[e_i(z_2),e_{i+1}(w)]_{\vv_2}]_{\vv_1}+
    [e_i(z_2),[e_i(z_1),e_{i+1}(w)]_{\vv_2}]_{\vv_1}=0,\\
  & [e_i(z_1),[e_i(z_2),e_{i-1}(w)]_{\vv_2^{-1}}]_{\vv_1^{-1}}+
    [e_i(z_2),[e_i(z_1),e_{i-1}(w)]_{\vv_2^{-1}}]_{\vv_1^{-1}}=0,\\
\end{split}
\end{equation}
where $e_i(z)=\sum_{r\in \BZ}{e_{i,r}z^{-r}}$ and $\langle i,j\rangle\in \BK$ is defined via
  $\langle i,j\rangle:=\vv_1^{\delta_{ij}-\delta_{i+1,j}}\vv_2^{\delta_{i,j+1}-\delta_{ij}}$.


\subsection{PBWD bases of $U^>_{\vv_1,\vv_2}(L\ssl_n)$}\label{ssec formulation main thm 3}
\

We shall follow the notations of Section~\ref{ssec formulation main thm 1},
except that now $(\lambda_1,\ldots,\lambda_{p-1})\in \{\vv_1,\vv_2\}^{p-1}$.
Similarly to~(\ref{higher roots}), we define the \emph{PBWD basis elements}
$e_\beta(r)\in U^{>}_{\vv_1,\vv_2}(L\ssl_n)$ via
\begin{equation*}
  e_\beta(r):=
  [\cdots[[e_{i_1,r_1},e_{i_2,r_2}]_{\lambda_1},e_{i_3,r_3}]_{\lambda_2},\cdots,e_{i_p,r_p}]_{\lambda_{p-1}}.
\end{equation*}
The monomials
  $e_h:=\prod\limits_{(\beta,r)\in \Delta^+\times \BZ}^{\rightarrow} e_\beta(r)^{h(\beta,r)}\ (h\in H)$
will be called the \emph{ordered PBWD monomials} of $U^{>}_{\vv_1,\vv_2}(L\ssl_n)$.
Here, the arrow over the product sign refers to the total order~(\ref{extended order}).

Our first main result establishes the PBWD property of $U^>_{\vv_1,\vv_2}(L\ssl_n)$:

\begin{Thm}\label{Main Theorem 3}
The ordered PBWD monomials $\{e_h\}_{h\in H}$ form a $\BK$-basis
of $U^{>}_{\vv_1,\vv_2}(L\ssl_n)$.
\end{Thm}

The proof of Theorem~\ref{Main Theorem 3} is outlined in
Section~\ref{ssec proof of Theorem 3} and is based on
the shuffle approach.

\begin{Rem}\label{rosso's choice}
We note that the PBWD basis elements introduced in~\cite[(3.14)]{hrz} are
\begin{equation}\label{simplest choice 2-param}
  e_{\alpha_j+\alpha_{j+1}+\ldots+\alpha_i}(r):=
  [\cdots[[e_{j,r},e_{j+1,0}]_{\vv_1},e_{j+2,0}]_{\vv_1},\cdots,e_{i,0}]_{\vv_1}.
\end{equation}
In this particular case, Theorem~\ref{Main Theorem 3}
recovers~\cite[Theorem 3.11]{hrz} provided without a proof.
\end{Rem}

\begin{Rem}\label{full 2-parameter vs halves}
The entire two-parameter quantum loop algebra $U_{\vv_1,\vv_2}(L\ssl_n)$
admits a triangular decomposition, cf.~Proposition~\ref{Triangular decomposition}(a).
Hence, an analogue of Theorem~\ref{Main Theorem 1 entire algebra} holds
for $U_{\vv_1,\vv_2}(L\ssl_n)$ as well, thus providing a family of
PBWD $\BK$-bases for $U_{\vv_1,\vv_2}(L\ssl_n)$.
\end{Rem}


\subsection{Integral form $\fU^>_{\vv_1,\vv_2}(L\ssl_n)$ and its PBWD bases}
\label{ssec formulation main thm 3.1}
\

For $(\beta,r)\in \Delta^+\times \BZ$, we define $\wt{e}_\beta(r)\in U^>_{\vv_1,\vv_2}(L\ssl_n)$ via
  $$\wt{e}_\beta(r):=(\vv_1^{1/2}\vv_2^{-1/2}-\vv_1^{-1/2}\vv_2^{1/2})e_\beta(r).$$
We also define $\wt{e}_h\in U^>_{\vv_1,\vv_2}(L\ssl_n)$
via~(\ref{ordered}) but using $\wt{e}_\beta(r)$ instead of $e_\beta(r)$.
Finally, we define integral form $\fU^>_{\vv_1,\vv_2}(L\ssl_n)$ as the
$\BC[\vv_1^{1/2},\vv_2^{1/2},\vv_1^{-1/2},\vv_2^{-1/2}]$-subalgebra of
$U^>_{\vv_1,\vv_2}(L\ssl_n)$ generated by $\{\wt{e}_\beta(r)\}_{\beta\in \Delta^+}^{r\in \BZ}$.

The following counterpart of Theorem~\ref{Main Theorem 2} provides
a much stronger version of Theorem~\ref{Main Theorem 3}:

\begin{Thm}\label{Main Theorem 3.1}
(a) The subalgebra $\fU^{>}_{\vv_1,\vv_2}(L\ssl_n)$ is
independent of all our choices.

\noindent
(b) Elements $\{\wt{e}_h\}_{h\in H}$ form a basis of the free
$\BC[\vv_1^{1/2},\vv_2^{1/2},\vv_1^{-1/2},\vv_2^{-1/2}]$-module
$\fU^{>}_{\vv_1,\vv_2}(L\ssl_n)$.
\end{Thm}

The proof of Theorem~\ref{Main Theorem 3.1} follows easily from the one
of Theorem~\ref{Main Theorem 3} presented below in the same way as we
deduced the proof of Theorem~\ref{Main Theorem 2} in Section~\ref{ssec proof of Theorem 2}
from that of Theorem~\ref{Main Theorem 1}.

\begin{Rem}\label{rosso's RTT interpretation}
We note that it is often more convenient to work with the two-parameter
quantum loop algebra $U_{\vv_1,\vv_2}(L\gl_n)$, cf.~Remark~\ref{ft's RTT interpretation}(a).
Its integral form $\fU_{\vv_1,\vv_2}(L\gl_n)$
is defined analogously to $\fU_{\vv}(L\gl_n)$. Following the arguments
of~\cite[Proposition~3.11]{ft2}, $\fU_{\vv_1,\vv_2}(L\gl_n)$ is identified with
the RTT integral form $\fU^\mathrm{rtt}_{\vv_1,\vv_2}(L\gl_n)$
under the $\BK$-algebra isomorphism
  $U_{\vv_1,\vv_2}(L\gl_n)\simeq
   \fU^\mathrm{rtt}_{\vv_1,\vv_2}(L\gl_n)\otimes_{\BC[\vv_1^{1/2},\vv_2^{1/2},\vv_1^{-1/2},\vv_2^{-1/2}]} \BK$
of~\cite{jl}, cf.~Remark~\ref{ft's RTT interpretation}(b).
Thus, the analogue of~\cite[Theorem 3.24]{ft2} provides a family of PBWD bases for 
$\fU_{\vv_1,\vv_2}(L\gl_n)$, cf.~Theorems~\ref{Full PBWD for integral-special choice},~\ref{Full PBWD for integral}.
\end{Rem}


\subsection{Shuffle algebra $\wt{S}^{(n)}$}
\label{ssec 2-param shuffle algebra}
\

Define the shuffle algebra $(\wt{S}^{(n)},\star)$ analogously to $(S^{(n)},\star)$
with the following modifications:

\begin{itemize}

\item[(1)] 
all vector spaces are now defined over $\BK$ (rather than over $\BC(\vv)$);

\item[(2)] 
$(\zeta_{i,j}(z))_{i,j\in I} \in \mathrm{Mat}_{I\times I}(\BK(z))$,
used in the shuffle product~(\ref{shuffle product}), are now chosen as:
\begin{equation*}
  \zeta_{i,j}(z)=
  \left(\frac{z-\vv_1^{1/2}\vv_2^{-1/2}}{z-1}\right)^{\delta_{j,i-1}}
  \left(\frac{z-\vv_1^{-1}\vv_2}{z-1}\right)^{\delta_{ji}}
  \left(\vv_1^{1/2}\vv_2^{1/2}\cdot \frac{z-\vv_1^{1/2}\vv_2^{-1/2}}{z-1}\right)^{\delta_{j,i+1}};
\end{equation*}

\item[(3)] 
the wheel conditions~(\ref{wheel conditions}) for $F$ are replaced with
\begin{equation*}
  F(\{x_{i,r}\})=0\quad \mathrm{once}\quad
  x_{i,r_1}=\vv_1^{1/2}\vv_2^{-1/2} x_{i+\epsilon,s}=\vv_1\vv_2^{-1} x_{i,r_2}
\end{equation*}
for some $\epsilon\in \{\pm 1\},\, i,\, r_1\ne r_2,\, s$.

\end{itemize}

\noindent
The following result is analogous to Proposition~\ref{simple shuffle} 
(recovering the latter for $\vv_1=\vv_2^{-1}=\vv$):

\begin{Prop}\label{simple shuffle for 2-parametric}
The assignment $e_{i,r}\mapsto x_{i,1}^r\ (i\in I, r\in \BZ)$
gives rise to an injective  $\BK$-algebra homomorphism
$\Psi\colon U_{\vv_1,\vv_2}^{>}(L\ssl_n)\to \widetilde{S}^{(n)}$.
\end{Prop}

Our proof of Theorem~\ref{Main Theorem 3} below implies the counterpart of
Theorem~\ref{hard shuffle}, see Remark~\ref{proof of shuffle 2-parameter}:

\begin{Thm}\label{hard shuffle 2-parametric}
$\Psi\colon U_{\vv_1,\vv_2}^{>}(L\ssl_n)\iso \widetilde{S}^{(n)}$
of Proposition~\ref{simple shuffle for 2-parametric} is a $\BK$-algebra isomorphism.
\end{Thm}


\subsection{Proof of Theorem~\ref{Main Theorem 3}}\label{ssec proof of Theorem 3}
\

The proof of Theorem~\ref{Main Theorem 3} is completely analogous to our proof
of Theorem~\ref{Main Theorem 1}(a) and is based on the embedding
$\Psi\colon U_{\vv_1,\vv_2}^{>}(L\ssl_n)\hookrightarrow \widetilde{S}^{(n)}$
of Proposition~\ref{simple shuffle for 2-parametric}.
Indeed, the linear independence of $\{e_h\}_{h\in H}$ is deduced exactly
as in Section~\ref{ssec linear indep} with the only modification of the
\emph{specialization maps}
  $\phi_{\unl{d}}\colon \widetilde{S}^{(n)}_{\unl{\ell}}\to
   \BK[\{y_{\beta,s}^{\pm 1}\}_{\beta\in \Delta^+}^{1\leq s\leq d_\beta}]^{\Sigma_{\unl{d}}}$
of~(\ref{specialization map}) by replacing $\vv^{-i}$ with $(\vv_1^{1/2}\vv_2^{-1/2})^{-i}$.
Then, the results of Lemmas~\ref{lower degrees} and~\ref{same degrees} still hold,
thus implying the linear independence of $\{e_h\}_{h\in H}$. Meanwhile,
the fact that $\{e_h\}_{h\in H}$ span $U^{>}_{\vv_1,\vv_2}(L\ssl_n)$ is deduced
using the arguments of Section~\ref{ssec spanning prop}. To be more precise, Lemma~\ref{spanning}
still holds and its iterative application immediately implies that any shuffle
element $F\in \widetilde{S}^{(n)}$ belongs to the $\BK$-span of $\{\Psi(e_h)\}_{h\in H}$.

\begin{Rem}\label{proof of shuffle 2-parameter}
Combining the last statement in the above proof of Theorem~\ref{Main Theorem 3}
with the injectivity of $\Psi$ (Proposition~\ref{simple shuffle for 2-parametric}),
we obtain a proof of Theorem~\ref{hard shuffle 2-parametric}.
\end{Rem}


\section{Generalizations to $U_{\vv}(L\ssl(m|n))$}
\label{sec super-Lie quantuma affine}

The quantum loop superalgebra $U_\vv(L\ssl(m|n))$\footnote{To be more precise,
one actually needs to use the classical Lie superalgebra $A(m-1,n-1)$ in place
of $\ssl(m|n)$, which do coincide only when $m\ne n$. However, we shall ignore
this difference, since we will be working only with the \emph{positive subalgebras}
and those are isomorphic: $U^>_\vv(L\ssl(m|n))\simeq U^>_\vv(LA(m-1,n-1))$.}
was introduced in~\cite{y}, both
in the Drinfeld-Jimbo and the new Drinfeld realizations, see~\cite[Theorem 8.5.1]{y}
for an identification of those. The representation theory of these algebras was
partially studied in~\cite{z1} by crucially utilizing a weak version of the
PBW theorem for $U^>_{\vv}(L\ssl(m|n))$,~\cite[Theorem 3.12]{z1}.
Inspired by~\cite{hrz}, the author also conjectured the PBW theorem
for $U^>_{\vv}(L\ssl(m|n))$, see~\cite[Remark 3.13(2)]{z1}.

The primary goal of this section is to generalize Theorem~\ref{Main Theorem 1}
to the case of $U^>_\vv(L\ssl(m|n))$, thus proving the conjecture of~\cite{z1}.
Along the way, we also generalize Theorem~\ref{hard shuffle} by providing
the shuffle realization of $U^>_{\vv}(L\ssl(m|n))$, which is of independent interest.
The latter is used to construct the PBWD bases for the integral form of
$U^>_{\vv}(L\ssl(m|n))$, generalizing Theorem~\ref{Main Theorem 2}.

\begin{Rem}
(a) We should stress right away that in the exposition below we do choose
a distinguished Dynkin diagram with a single simple positive root of odd degree.
The generalization of all our results to an arbitrary Dynkin diagram is
carried out in~\cite{ts}, cf.~Section~\ref{ssec all Dynkin diagrams}.

\noindent
(b) The shuffle algebras associated to quantum loop superalgebras
seem to be new in the literature as they involve both symmetric and
skew-symmetric functions (``bosons'' and ``fermions'').
\end{Rem}


\subsection{Quantum loop superalgebra $U^>_\vv(L\ssl(m|n))$}
\label{ssec quantum super affine}
\

For the purpose of this section, it suffices to work only with
the subalgebra $U^>_\vv(L\ssl(m|n))$ of $U_\vv(L\ssl(m|n))$. Let
$I=\{1,\ldots,m+n-1\}$ from now on. Consider a free $\BZ$-module
$\oplus_{i=1}^{m+n} \BZ\epsilon_i$ with the bilinear form $(\cdot,\cdot)$
determined by $(\epsilon_i,\epsilon_j)=(-1)^{\delta_{i>m}}\delta_{ij}$.
Let $\vv$ be a formal variable and define $\{\vv_i\}_{i\in I}\subset \{\vv,\vv^{-1}\}$
via $\vv_i:=\vv^{(\epsilon_i,\epsilon_i)}$. For $i,j\in I$, set
$\bar{c}_{ij}:=(\alpha_i,\alpha_j)$ with $\alpha_i:=\epsilon_i-\epsilon_{i+1}$.

Following~\cite{y} (cf.~\cite[Theorem 3.3]{z1}), define $U^>_\vv(L\ssl(m|n))$
to be the associative $\BC(\vv)$-superalgebra generated by
$\{e_{i,r}\}_{i\in I}^{r\in \BZ}$ with the $\BZ_2$-grading
$|e_{m,r}|=\bar{1}, |e_{i,r}|=\bar{0}\ (i\ne m, r\in \BZ)$,
and subject to the following defining relations:
\begin{equation}\label{SuperLie 1}
    (z - \vv^{\bar{c}_{ij}} w) e_i(z)e_j(w)=(\vv^{\bar{c}_{ij}} z - w)e_j(w)e_i(z)
  \ \mathrm{if}\ \bar{c}_{ij}\ne 0,
\end{equation}
\begin{equation}\label{SuperLie 2}
\begin{split}
  & [e_i(z),e_j(w)]=0\ \mathrm{if}\ \bar{c}_{ij}=0,\\
  & [e_i(z_1),[e_i(z_2),e_j(w)]_{\vv^{-1}}]_\vv+[e_i(z_2),[e_i(z_1),e_j(w)]_{\vv^{-1}}]_\vv=0
    \ \mathrm{if}\ \bar{c}_{ij}=\pm 1 \ \mathrm{and}\ i\ne m,
\end{split}
\end{equation}
as well as quartic Serre relations:
\begin{multline}\label{SuperLie 3}
  [[[e_{m-1}(w),e_m(z_1)]_{\vv^{-1}},e_{m+1}(u)]_\vv,e_m(z_2)]\ +\\
  [[[e_{m-1}(w),e_m(z_2)]_{\vv^{-1}},e_{m+1}(u)]_\vv,e_m(z_1)]=0,
\end{multline}
where $e_i(z)=\sum_{r\in \BZ} e_{i,r}z^{-r}$ and we use the super-bracket notations:
\begin{equation*}
  [a,b]_x:=ab-(-1)^{|a||b|}x\cdot ba,\qquad [a,b]:=[a,b]_1
\end{equation*}
for $\BZ_2$-homogeneous elements $a,b$ (we set $(-1)^{\bar{0}}:=1$ and $(-1)^{\bar{1}}:=-1$).


\subsection{PBWD bases of $U^>_\vv(L\ssl(m|n))$}\label{ssec formulation main thm 4}
\

Let $\Delta^+=\{\alpha_j+\alpha_{j+1}+\ldots+\alpha_i\}_{1\leq j\leq i< m+n}$.
For $\beta\in \Delta^+$, define its parity $|\beta|\in \BZ_2$ via
\begin{equation}\label{root parity}
  |\beta|=
  \begin{cases}
     \bar{1}, & \text{if } m\in [\beta] \\
     \bar{0}, & \text{if } m\notin [\beta]
  \end{cases}.
\end{equation}
We shall follow the notations of Section~\ref{ssec formulation main thm 1}.
In particular, we define the \emph{PBWD basis elements}
$e_\beta(r)\in U^{>}_\vv(L\ssl(m|n))$ via~(\ref{higher roots}), but with 
$[\cdot,\cdot]$ denoting the super-bracket. 

Let $\bar{H}$ denote the set of all functions
$h\colon \Delta^+\times \BZ\to \BN$ with finite support and such that
$h(\beta,r)\leq 1$ if $|\beta|=\bar{1}$. The monomials
\begin{equation}\label{ordered supercase}
  e_h\ :=\prod\limits_{(\beta,r)\in \Delta^+\times \BZ}^{\rightarrow} e_\beta(r)^{h(\beta,r)}\ ,
  \qquad \forall\, h\in \bar{H}
\end{equation}
will be called the \emph{ordered PBWD monomials} of $U^{>}_\vv(L\ssl(m|n))$.
Here, the arrow over the product sign refers to
the total order~(\ref{extended order}), as before.

Our first main result establishes the PBWD property of $U^{>}_\vv(L\ssl(m|n))$:

\begin{Thm}\label{Main Theorem 4}
The ordered PBWD monomials $\{e_h\}_{h\in \bar{H}}$ form a $\BC(\vv)$-basis
of $U^{>}_\vv(L\ssl(m|n))$.
\end{Thm}

The proof of Theorem~\ref{Main Theorem 4} is presented in
Section~\ref{ssec proof of Theorem 4} and is based on
the shuffle approach.

\begin{Rem}\label{zhang's choice}
We note that the PBWD basis elements introduced in~\cite[(3.12)]{z1} are
\begin{equation}\label{simplest choice super-Lie}
  e_{\alpha_j+\alpha_{j+1}+\ldots+\alpha_i}(r):=
  [\cdots[[e_{j,r},e_{j+1,0}]_{\vv_{j+1}},e_{j+2,0}]_{\vv_{j+2}},\cdots,e_{i,0}]_{\vv_i}.
\end{equation}
In this particular case, Theorem~\ref{Main Theorem 4} recovers
the conjecture of~\cite[Remark 3.13(2)]{z1}.
\end{Rem}

\begin{Rem}
The entire quantum loop superalgebra $U_\vv(L\ssl(m|n))$ admits a triangular
decomposition as in Proposition~\ref{Triangular decomposition},
see~\cite[Theorem 3.3]{z1}. Hence, an analogue of
Theorem~\ref{Main Theorem 1 entire algebra} holds for $U_\vv(L\ssl(m|n))$ as well,
thus providing a family of PBWD $\BC(\vv)$-bases for $U_\vv(L\ssl(m|n))$.
\end{Rem}


\subsection{Integral form $\fU^>_\vv(L\ssl(m|n))$ and its PBWD bases}
\label{ssec formulation main thm 4.1}
\

Following~(\ref{integral basis PBW}), for any $(\beta,r)\in \Delta^+\times \BZ$, 
we define $\wt{e}_\beta(r)\in U^>_\vv(L\ssl(m|n))$ via 
  $$\wt{e}_\beta(r):=(\vv-\vv^{-1})e_\beta(r).$$ 
We also define $\wt{e}_h$ via~(\ref{ordered supercase}) but using $\wt{e}_\beta(r)$ instead of $e_\beta(r)$.
Finally, let $\fU^>_\vv(L\ssl(m|n))$ denote the $\BC[\vv,\vv^{-1}]$-subalgebra
of $U^>_\vv(L\ssl(m|n))$ generated by $\{\wt{e}_\beta(r)\}_{\beta\in \Delta^+}^{r\in \BZ}$.

The following counterpart of Theorem~\ref{Main Theorem 2} provides
a much stronger version of Theorem~\ref{Main Theorem 4}:

\begin{Thm}\label{Main Theorem 4.1}
(a) The subalgebra $\fU^>_\vv(L\ssl(m|n))$ is independent of all our choices.

\noindent
(b) The elements $\{\wt{e}_h\}_{h\in \bar{H}}$ form a basis of the free
$\BC[\vv,\vv^{-1}]$-module $\fU^{>}_\vv(L\ssl(m|n))$.
\end{Thm}

The proof of Theorem~\ref{Main Theorem 4.1} follows easily from the one
of Theorem~\ref{Main Theorem 4} presented below in the same way as we
deduced the proof of Theorem~\ref{Main Theorem 2} in
Section~\ref{ssec proof of Theorem 2} from that of Theorem~\ref{Main Theorem 1}.

\begin{Rem}\label{zhang's RTT interpretation}
We note that it is often more convenient to work with the quantum loop superalgebra 
$U_\vv(L\gl(m|n))$, cf.~Remark~\ref{ft's RTT interpretation}(a). Its integral form
$\fU_\vv(L\gl(m|n))$ is defined analogously to $\fU_{\vv}(L\gl_n)$.
Following the arguments of~\cite[Proposition 3.11]{ft2}, $\fU_\vv(L\gl(m|n))$
is identified with the RTT integral form
$\fU^\mathrm{rtt}_\vv(L\gl(m|n))$,~\cite[Definition 3.1]{z3},
under the $\BC(\vv)$-algebra isomorphism
  $U_\vv(L\gl(m|n))\simeq
   \fU^\mathrm{rtt}_\vv(L\gl(m|n))\otimes_{\BC[\vv,\vv^{-1}]} \BC(\vv)$,
cf.~Remark~\ref{ft's RTT interpretation}(b).
Hence, the analogue of~\cite[Theorem 3.24]{ft2} provides a family of PBWD bases for
$\fU_\vv(L\gl(m|n))$, cf.~Theorems~\ref{Full PBWD for integral-special choice},~\ref{Full PBWD for integral}.
\end{Rem}


\subsection{Shuffle algebra $S^{(m|n)}$}\label{ssec trigonometric super shuffle}
\

Consider an $\BN^I$-graded $\BC(\vv)$-vector space
  $\BS^{(m|n)}=\underset{\underline{k}\in \BN^{I}}\bigoplus\BS^{(m|n)}_{\underline{k}}$,
where $\BS^{(m|n)}_{(k_1,\ldots,k_{m+n-1})}$ consists of rational
functions in the variables $\{x_{i,r}\}_{i\in I}^{1\leq r\leq k_i}$
which are \textbf{supersymmetric}, that is
\begin{enumerate}
 \item[1)]
 symmetric in $\{x_{i,r}\}_{r=1}^{k_i}$ for every $i\ne m$;

 \item[2)]
 skew-symmetric in $\{x_{m,r}\}_{r=1}^{k_m}$.
\end{enumerate}
We also fix an $I\times I$ matrix of rational functions
$(\zeta_{i,j}(z))_{i,j\in I} \in \mathrm{Mat}_{I\times I}(\BC(\vv)(z))$ via
\begin{equation}
  \zeta_{i,j}(z)=
  \frac{z-\vv^{-\bar{c}_{ij}}}{z-1}.
\end{equation}
This allows us to endow $\BS^{(m|n)}$ with a structure of an associative unital
algebra with the shuffle product defined via~(\ref{shuffle product}),
but a supersymmetrization $\SSym$ in place of the symmetrization $\Sym$.
Here, the \emph{supersymmetrization} of $f\in \BC(\{x_{i,1},\ldots,x_{i,s_i}\}_{i\in I})$
is defined via
\begin{equation*}
  \SSym_{\Sigma_{\unl{s}}}(f)(\{x_{i,1},\ldots,x_{i,s_i}\}_{i\in I})\ :=
  \sum_{(\sigma_1,\ldots,\sigma_{m+n-1})\in \Sigma_{\unl{s}}}
  \mathrm{sgn}(\sigma_m) f\left(\{x_{i,\sigma_i(1)},\ldots,x_{i,\sigma_i(s_i)}\}_{i\in I}\right).
\end{equation*}

As before, we will be interested only in the subspace of
$\BS^{(m|n)}$ defined by the \emph{pole} and \emph{wheel conditions}
(but now there are two kinds of the latter one):

\begin{itemize}

\item[$\bullet$] 
We say that $F\in \BS^{(m|n)}_{\underline{k}}$ satisfies the
\emph{pole conditions} if
\begin{equation}\label{pole conditions supercase}
  F=
  \frac{f(x_{1,1},\ldots,x_{m+n-1,k_{m+n-1}})}
  {\prod_{i=1}^{m+n-2}\prod_{r\leq k_i}^{r'\leq k_{i+1}}(x_{i,r}-x_{i+1,r'})},
\end{equation}
where $f\in \BC(\vv)[\{x_{i,r}^{\pm 1}\}_{i\in I}^{1\leq r\leq k_i}]$ is a
supersymmetric Laurent polynomial, that is, symmetric in $\{x_{i,r}\}_{r=1}^{k_i}$ for
every $i\ne m$ and skew-symmetric in $\{x_{m,r}\}_{r=1}^{k_m}$.

\item[$\bullet$] 
We say that $F\in \BS^{(m|n)}_{\underline{k}}$ satisfies the
\emph{first kind wheel conditions} if
\begin{equation}\label{wheel conditions 1 supercase}
  F(\{x_{i,r}\})=0\quad \mathrm{once}\quad
  x_{i,r_1}=\vv_i x_{i+\epsilon,s}=\vv_i^2x_{i,r_2}
\end{equation}
for some
  $\epsilon\in \{\pm 1\},\, i\in I\backslash\{m\},\, 
   1\leq r_1\ne r_2\leq k_i,\, 1\leq s\leq k_{i+\epsilon}$.

\item[$\bullet$] 
We say that $F\in \BS^{(m|n)}_{\underline{k}}$ satisfies the
\emph{second kind wheel conditions} if
\begin{equation}\label{wheel conditions 2 supercase}
  F(\{x_{i,r}\})=0\quad \mathrm{once}\quad x_{m-1,s}=\vv x_{m,r_1}=x_{m+1,s'}=\vv^{-1} x_{m,r_2}
\end{equation}
for some 
  $1\leq r_1\ne r_2\leq k_m,\, 1\leq s\leq k_{m-1},\, 1\leq s'\leq k_{m+1}$.

\end{itemize}

\noindent
Let $S^{(m|n)}_{\underline{k}}\subset \BS^{(m|n)}_{\underline{k}}$ denote
the subspace of all elements $F$ satisfying these three conditions. We define
  $$S^{(m|n)}:=\underset{\underline{k}\in \BN^{I}}\bigoplus S^{(m|n)}_{\underline{k}}$$
It is straightforward to check that $S^{(m|n)}\subset\BS^{(m|n)}$ is $\star$-closed.
Similar to Proposition~\ref{simple shuffle}, the \textbf{shuffle algebra}
$\left(S^{(m|n)},\star\right)$ is related to $U^>_\vv(L\ssl(m|n))$ via:

\begin{Prop}\label{simple shuffle superLie}
The assignment $e_{i,r}\mapsto x_{i,1}^r\ (i\in I, r\in \BZ)$
gives rise to an injective $\BC(\vv)$-algebra homomorphism
$\Psi\colon U_\vv^{>}(L\ssl(m|n))\to S^{(m|n)}$.
\end{Prop}

Our proof of Theorem~\ref{Main Theorem 4} below implies the counterpart
of Theorem~\ref{hard shuffle}, see Remark~\ref{proof of shuffle super-Lie}:

\begin{Thm}\label{hard shuffle superLie}
$\Psi\colon U_\vv^{>}(L\ssl(m|n))\iso S^{(m|n)}$ of Proposition~\ref{simple shuffle superLie} 
is an algebra isomorphism.
\end{Thm}


\subsection{Proof of Theorem~\ref{Main Theorem 4}}\label{ssec proof of Theorem 4}
\

The proof of Theorem~\ref{Main Theorem 4} is similar to our proof of
Theorem~\ref{Main Theorem 1}(a) and is based on the embedding
$\Psi\colon U_\vv^{>}(L\ssl(m|n))\hookrightarrow S^{(m|n)}$ of
Proposition~\ref{simple shuffle superLie}. Thus, we will only
outline the proof, highlighting the key changes.

\medskip
We start by establishing Theorem~\ref{Main Theorem 4} in
the simplest case $m=n=1$:

\begin{Lem}\label{m=n=1 case}
For any total order $\preceq$ on $\BZ$, the ordered monomials
$\{e_{r_1}e_{r_2}\cdots e_{r_k}\}_{k\in \BN}^{r_1\prec\cdots\prec r_k}$
form a $\BC(\vv)$-basis of $U^{>}_\vv(L\ssl(1|1))$.
\end{Lem}

\begin{proof}
This follows from the $\BC(\vv)$-algebra isomorphism
$S^{(1|1)}\simeq \bigoplus_k\Lambda_k$, where $\Lambda_k$ denotes
the vector space of skew-symmetric Laurent polynomials in $k$ variables,
while the algebra structure on the direct sum arises via the standard
skew-symmetrization maps $\Lambda_k\otimes \Lambda_{\ell}\to \Lambda_{k+\ell}$.
\end{proof}

Let us now treat the general case of $m, n$.
Given a degree vector $\unl{d}=\{d_\beta\}_{\beta\in \Delta^+}\in \BN^{\Delta^+}$,
define $\unl{\ell}\in \BN^I$ via
$\sum_{\beta\in \Delta^+} d_\beta \beta=\sum_{i\in I} \ell_i\alpha_i$.
The \textbf{specialization map}
\begin{equation*}
  \phi_{\unl{d}}\colon S^{(m|n)}_{\unl{\ell}}\longrightarrow
  \BC(\vv)[\{y_{\beta,s}^{\pm 1}\}_{\beta\in \Delta^+}^{1\leq s\leq d_\beta}]
\end{equation*}
is defined similar to~(\ref{specialization map}), but with the only change
that the variable $x_{i,r}$ in the $s$-th copy of the interval $[\beta]$
is specialized to $\vv^{-i}y_{\beta,s}$ if $i\leq m$ and to $\vv^{i-2m}y_{\beta,s}$ if $i>m$.
We note that $\phi_{\unl{d}}(F)$ is a supersymmetric Laurent polynomial, that is,
symmetric in $\{y_{\beta,s}\}_{s=1}^{d_\beta}$ if $|\beta|=\bar{0}$
and skew-symmetric in $\{y_{\beta,s}\}_{s=1}^{d_\beta}$ if $|\beta|=\bar{1}$.

Thus defined specialization maps $\phi_{\unl{d}}$ still satisfy
Lemmas~\ref{lower degrees} and~\ref{same degrees} (with $\bar{H}$ used
in place of $H$ in the formulation of the latter), hence,
the linear independence of $\{e_h\}_{h\in \bar{H}}$ as follows from
the argument presented  right after our proof of Lemma~\ref{same degrees}.

Furthermore, for any $h\in \bar{H}$ with $\deg(h)=\unl{d}$, we have the following
generalization of the formulas~(\ref{explicit factors},~\ref{explicit formula for same degrees})
from our proof of Lemma~\ref{same degrees}:
\begin{equation}\label{specialization 2-param}
  \phi_{\unl{d}}(\Psi(e_h))\ =
  c\ \cdot \prod_{\beta,\beta'\in \Delta^+}^{\beta<\beta'} \bar{G}_{\beta,\beta'}\ \cdot
  \prod_{\beta\in \Delta^+}\bar{G}_\beta\cdot
  \prod_{\beta\in \Delta^+} \left(\sum_{\sigma_\beta\in \Sigma_{d_{\beta}}} \bar{G}_\beta^{(\sigma_\beta)}\right)
  \ \mathrm{with}\ c\in \BC^\times\cdot \vv^\BZ
\end{equation}
where
\begin{equation}\label{explicit factors superLie}
\begin{split}
   & \bar{G}_{\beta,\beta'}\ =
     \prod_{1\leq s\leq d_\beta}^{1\leq s'\leq d_{\beta'}}
     (y_{\beta,s}-\vv^{-2}y_{\beta',s'})^{\nu^-(\beta,\beta')}\cdot
     (y_{\beta,s}-\vv^2y_{\beta',s'})^{\nu^+(\beta,\beta')}\times\\
   & \ \ \ \ \ \ \ \ \ \ \prod_{1\leq s\leq d_\beta}^{1\leq s'\leq d_{\beta'}}
    (y_{\beta,s}-y_{\beta',s'})^{\delta_{j(\beta')>j(\beta)}\delta_{i(\beta)+1\in [\beta']}+\delta_{m\in[\beta]}\delta_{m\in [\beta']}},\\
   & \bar{G}_\beta=(1-\vv^2)^{d_\beta(i(\beta)-j(\beta))}\ \cdot
     \prod_{1\leq s\ne s'\leq d_{\beta}} (y_{\beta,s}-\vv^2 y_{\beta,s'})^{i(\beta)-j(\beta)}\cdot
     \prod_{1\leq s\leq d_\beta} y_{\beta,s}^{i(\beta)-j(\beta)},\\
   & \bar{G}_\beta^{(\sigma_\beta)}=
     \prod_{s=1}^{d_\beta} y_{\beta,\sigma_\beta(s)}^{r_\beta(h,s)}\cdot
     \begin{cases}
       \prod_{s<s'} \frac{y_{\beta,\sigma_\beta(s)}-\vv^{-2}y_{\beta,\sigma_\beta(s')}}
                    {y_{\beta,\sigma_\beta(s)}-y_{\beta,\sigma_\beta(s')}}, & \text{if } m>i(\beta)\\
       \prod_{s<s'} \frac{y_{\beta,\sigma_\beta(s)}-\vv^2y_{\beta,\sigma_\beta(s')}}
                    {y_{\beta,\sigma_\beta(s)}-y_{\beta,\sigma_\beta(s')}}, & \text{if } m<j(\beta)\\
       \mathrm{sgn}(\sigma_\beta), & \text{if } m\in [\beta]
     \end{cases}.
\end{split}
\end{equation}
Here, the collection $\{r_\beta(h,1),\ldots,r_\beta(h,d_\beta)\}$ is defined
exactly as after~(\ref{explicit factors})
(that is, listing every $r\in \BZ$ with multiplicity $h(\beta,r)>0$ with respect
to the total order $\preceq_\beta$ on $\BZ$), while the powers
$\nu^\pm(\beta,\beta')$ are given by the following explicit formulas:
\begin{equation}\label{power 1}
\begin{split}
  & \nu^-(\beta,\beta')=\#\{(j,j')\in [\beta]\times[\beta']|j=j'<m\} \ + \\
  & \qquad \qquad \quad \, \#\{(j,j')\in [\beta]\times [\beta']|j=j'+1>m\}
\end{split}
\end{equation}
and
\begin{equation}\label{power 2}
\begin{split}
  & \nu^+(\beta,\beta')=\#\{(j,j')\in [\beta]\times [\beta']|j=j'>m\} \ + \\
  & \qquad \qquad \quad \, \#\{(j,j')\in [\beta]\times [\beta']|j=j'+1\leq m\}.
\end{split}
\end{equation}

For any $\beta\in \Delta^+$, we note that the sum
  $\sum_{\sigma_\beta\in \Sigma_{d_{\beta}}} \bar{G}_\beta^{(\sigma_\beta)}$
coincides (up to a factor of $\BC^\times$) with the value of the shuffle element
  $x^{r_\beta(h,1)}\star\cdots \star x^{r_\beta(h,d_\beta)}$,
viewed as an element of

\begin{enumerate}
 \item[1)] the shuffle algebra $S^{(2|0)}$ if $m>i(\beta)$,

 \item[2)] the shuffle algebra $S^{(0|2)}$ if $m<j(\beta)$,

 \item[3)] the shuffle algebra $S^{(1|1)}$ if $m\in [\beta]$,
\end{enumerate}
evaluated at $\{y_{\beta,s}\}_{s=1}^{d_\beta}$.
The latter elements are linearly independent, due to
Lemmas~\ref{n=1 case},~\ref{m=n=1 case}.

\begin{Rem}
The above reduction to the rank $1$ cases (that is, $m+n=2$)
together with Lemma~\ref{m=n=1 case} explains why $H$ had to be
replaced by $\bar{H}$ in the current setting.
\end{Rem}

The fact that $\{e_h\}_{h\in \bar{H}}$ span $U^>_\vv(L\ssl(m|n))$ follows from
the validity of Lemma~\ref{spanning} in the current setting. Let us now prove
the latter using the same ideas and notations as before.

First, we note that the wheel
conditions~(\ref{wheel conditions 1 supercase},~\ref{wheel conditions 2 supercase})
for $F$ guarantee that $\phi_{\unl{d}}(F)$
(which is a Laurent polynomial in $\{y_{\beta,s}\}$)
vanishes up to appropriate orders under the following specializations:
\begin{enumerate}
\item[(i)] $y_{\beta,s}=\vv^{-2}y_{\beta',s'}$ for $(\beta,s)<(\beta',s')$,

\item[(ii)] $y_{\beta,s}=\vv^{2}y_{\beta',s'}$ for $(\beta,s)<(\beta',s')$.
\end{enumerate}
A straightforward case-by-case verification shows that these
orders of vanishing exactly equal the corresponding powers of
$y_{\beta,s}-\vv^{-2}y_{\beta',s'}$ and $y_{\beta,s}-\vv^{2}y_{\beta',s'}$
appearing in $\bar{G}_{\beta,\beta'}$ (if $\beta<\beta'$) or $\bar{G}_\beta$
(if $\beta=\beta'$) of~(\ref{explicit factors superLie}).
In the former case, those are explicitly given by~(\ref{power 1},~\ref{power 2}).

\begin{Rem}
We should
point out right away that the computation of the corresponding orders requires
an extra argument in the case when $\beta=\beta'$ and $m\in [\beta]$.
 Recall that the way we counted these orders in the proof of Lemma~\ref{spanning}
was by realizing the specialization $\phi_{\unl{d}}$ as a step-by-step
specialization in each interval in the specified order. A priori, we can choose
another order of the intervals or even another way to perform this specialization.
Let us now illustrate how our argument should be modified in the particular case
$\beta=\beta', m\in [\beta]$.
 Note that if we first specialize the variables in the interval $[\beta]$ to the corresponding
$\vv$-multiples of $y_{\beta,s}$, then the wheel conditions contribute $i(\beta)-j(\beta)$
to the order of vanishing at $y_{\beta,s}=\vv^2y_{\beta,s'}$ and $i(\beta)-j(\beta)-1$
to the order of vanishing at $y_{\beta,s}=\vv^{-2}y_{\beta,s'}$. If instead we first
specialize the variables in the interval $[\beta]$ to the corresponding $\vv$-multiples
of $y_{\beta,s'}$, then the wheel conditions contribute $i(\beta)-j(\beta)-1$ to the order
of vanishing at $y_{\beta,s}=\vv^2y_{\beta,s'}$ and $i(\beta)-j(\beta)$ to the order of
vanishing at $y_{\beta,s}=\vv^{-2}y_{\beta,s'}$. Thus, none of these two specializations
provides the desired orders of vanishing simultaneously for $y_{\beta,s}=\vv^2 y_{\beta,s'}$
and $y_{\beta,s}=\vv^{-2}y_{\beta,s'}$. However, picking the maximal of the orders separately
for $y_{\beta,s}=\vv^2 y_{\beta,s'}$ and $y_{\beta,s}=\vv^{-2}y_{\beta,s'}$, we recover
$i(\beta)-j(\beta)$ for both of them, so that they equal the corresponding powers of
$y_{\beta,s}-\vv^{2}y_{\beta,s'}$ and $ y_{\beta,s}-\vv^{-2}y_{\beta,s'}$
appearing in $\bar{G}_\beta$ of~(\ref{explicit factors superLie}).
\end{Rem}

Second, we claim that $\phi_{\unl{d}}(F)$ vanishes under the following specializations:
\begin{enumerate}
\item[(iii)] $y_{\beta,s}=y_{\beta',s'}$ for $(\beta,s)<(\beta',s')$
such that $j(\beta)<j(\beta')$ and $i(\beta)+1\in [\beta']$.
\end{enumerate}
Indeed, if $j(\beta)<j(\beta')$ and $i(\beta)+1\in [\beta']$, there are positive
roots $\gamma,\gamma'\in \Delta^+$ such that
  $j(\gamma)=j(\beta), i(\gamma)=i(\beta'),
   j(\gamma')=j(\beta'), i(\gamma')=i(\beta)$.
Consider the degree vector $\unl{d}'\in T_{\unl{k}}$ given by
  $d'_{\alpha}=
   d_{\alpha}+\delta_{\alpha,\gamma}+\delta_{\alpha,\gamma'}-\delta_{\alpha,\beta}-\delta_{\alpha,\beta'}$.
Then, $\unl{d}'>\unl{d}$ and thus $\phi_{\unl{d}'}(F)=0$. The result follows.

Finally, we also note that the skew-symmetry of the elements of
$S^{(m|n)}$ with respect to the variables $\{x_{m,\ast}\}$ implies that
$\phi_{\unl{d}}(F)$ vanishes under the following specializations:
\begin{enumerate}
\item[(iv)] $y_{\beta,s}=y_{\beta',s'}$ for all $\beta<\beta'$ (and any $s,s'$) such
that $[\beta]\ni m\in [\beta']$.
\end{enumerate}

\medskip

Combining the above vanishing conditions for $\phi_{\unl{d}}(F)$,
we see that it is divisible exactly by the product
$\prod_{\beta<\beta'} \bar{G}_{\beta,\beta'}\cdot \prod_{\beta} \bar{G}_\beta$
of~(\ref{explicit factors superLie}). Therefore, we have
\begin{equation*}
  \phi_{\unl{d}}(F)\ =
  \prod_{\beta,\beta'\in \Delta^+}^{\beta<\beta'} \bar{G}_{\beta,\beta'}\ \cdot
  \prod_{\beta\in \Delta^+}\bar{G}_\beta\cdot \bar{G},
\end{equation*}
where $\bar{G}$ is a supersymmetric Laurent polynomial in $\{y_{\beta,s}\}_{\beta\in \Delta^+}^{1\leq s\leq d_\beta}$, that is:

\begin{itemize}
  
\item[(1)] 
$\bar{G}$ is symmetric in $\{y_{\beta,s}\}_{s=1}^{d_\beta}$ if $|\beta|=\bar{0}$,

\item[(2)]
$\bar{G}$ is skew-symmetric in $\{y_{\beta,s}\}_{s=1}^{d_\beta}$ if $|\beta|=\bar{1}$.

\end{itemize}

\noindent
Combining this observation with
formulas~(\ref{specialization 2-param},~\ref{explicit factors superLie})
and the discussion following them, we obtain a proof of Lemma~\ref{spanning}
in the current setting, due to Lemmas~\ref{n=1 case},~\ref{m=n=1 case}.

Following the arguments from the end of Section~\ref{ssec spanning prop},
we see that $\{\Psi(e_h)\}_{h\in \bar{H}}$ linearly span $S^{(m|n)}$.
Invoking the injectivity of $\Psi$ (Proposition~\ref{simple shuffle superLie}),
this implies that $\{e_h\}_{h\in \bar{H}}$ span $U^>_\vv(L\ssl(m|n))$.
This completes our proof of Theorem~\ref{Main Theorem 4}.

\begin{Rem}\label{proof of shuffle super-Lie}
The above argument also provides a proof of Theorem~\ref{hard shuffle superLie}.
\end{Rem}


\section{Generalizations to the Yangian $Y_\hbar(\ssl_n)$}
\label{sec yangian counterparts}

The PBWD bases for the Yangian $Y_\hbar(\fg)$ of any semisimple
Lie algebra $\fg$ have been constructed $25$ years ago
in~\cite{l}\footnote{See~\cite[Appendix B]{ft2} for a correction of
a gap in the proof of~\cite{l}.}. Note that while the Yangian deforms
the universal enveloping of the loop algebra, that is
$Y_\hbar(\fg)/(\hbar)\simeq U(\fg[t])$, there is a canonical construction
of the \emph{Drinfeld-Gavarini dual} (Hopf) subalgebra
$Y'_\hbar(\fg)\subset Y_\hbar(\fg)$ such that $Y'_\hbar(\fg)/(\hbar)$ is
a commutative $\BC$-algebra, see~\cite[Appendix A]{ft2} and the original
references~\cite{d2,g}. The PBWD bases for the Drinfeld-Gavarini dual
$Y'_\hbar(\fg)$ were constructed in~\cite[Theorem A.7]{ft2}, following~\cite{g}.

As just mentioned, the PBWD results
(cf.~Theorems~\ref{Main Theorem 1},~\ref{Main Theorem 1 entire algebra},~\ref{Main Theorem 2},~\ref{Full PBWD for integral})
are known both for $Y_\hbar(\fg)$ and $Y'_\hbar(\fg)$
for an arbitrary semisimple $\fg$. Thus, the key objective of this section
is to provide the shuffle realizations of $Y_\hbar(\ssl_n)$
and $Y'_\hbar(\ssl_n)$ similar to those of
Theorems~\ref{hard shuffle},~\ref{shuffle integral form}.
For the latter purpose, it suffices to consider only the subalgebras
$Y^>_\hbar(\ssl_n), {Y'_\hbar}^{>}(\ssl_n)\simeq \bY^>_\hbar(\ssl_n)$.


\subsection{Algebras $Y^>_\hbar(\ssl_n)$ and $\bY^>_\hbar(\ssl_n)$}\label{ssec half-yangian sl_n}
\

Let $I=\{1,\ldots,n-1\}$, $(c_{ij})_{i,j\in I}$ be the Cartan matrix
of $\ssl_n$, and $\hbar$ be a formal variable. Following~\cite{d}, define
the \emph{positive subalgebra} of the Yangian of $\ssl_n$, denoted by $Y^>_\hbar(\ssl_n)$, to be
the associative $\BC[\hbar]$-algebra generated by $\{e_{i,r}\}_{i\in I}^{r\in \BN}$
with the following defining relations:
\begin{equation}\label{Yan 1}
  [e_{i,r+1}, e_{j,s}]-[e_{i,r},e_{j,s+1}]=
  \frac{c_{ij}\hbar}{2}\left(e_{i,r}e_{j,s}+e_{j,s}e_{i,r}\right)
\end{equation}
as well as Serre relations:
\begin{equation}\label{Yan 7}
\begin{split}
  & [e_{i,r},e_{j,s}]=0\  \mathrm{if}\ c_{ij}=0,\\
  & [e_{i,r_1},[e_{i,r_2},e_{j,s}]]+
    [e_{i,r_2},[e_{i,r_1},e_{j,s}]]=0\ \mathrm{if}\ c_{ij}=-1.
\end{split}
\end{equation}

Let $\{\alpha_i\}_{i=1}^{n-1}$ and $\Delta^+$ be as in
Section~\ref{ssec formulation main thm 1}. For any
$(\beta,r)\in \Delta^+\times \BN$, we choose:
\begin{enumerate}
\item[1)]
a decomposition $\beta=\alpha_{i_1}+\ldots+\alpha_{i_p}$ such that
$[\cdots[e_{\alpha_{i_1}},e_{\alpha_{i_2}}],\cdots,e_{\alpha_{i_p}}]$
is a non-zero root vector $e_\beta$ of $\ssl_n$
(here, $e_{\alpha_i}$ denotes the standard Chevalley generator of $\ssl_n$);

\item[2)]
a decomposition $r=r_1+\ldots+r_p$ with $r_k\in \BN$.
\end{enumerate}
Then, we define the \emph{PBWD basis elements} $e_\beta(r)\in Y^>_\hbar(\ssl_n)$ via
\begin{equation}\label{higher roots yangian}
  e_\beta(r):=
  [\cdots[[e_{i_1,r_1},e_{i_2,r_2}],e_{i_3,r_3}],\cdots,e_{i_p,r_p}].
\end{equation}
Let $H^+$ denote the set of all functions $h\colon \Delta^+\times \BN\to \BN$
with finite support. The monomials
\begin{equation}\label{ordered yangian}
  e_h\ :=\prod\limits_{(\beta,r)\in \Delta^+\times \BN}^{\rightarrow} e_\beta(r)^{h(\beta,r)}\ ,
  \quad \forall\, h\in H^+
\end{equation}
will be called the \emph{ordered PBWD monomials} of $Y^>_\hbar(\ssl_n)$.
Here, the arrow over the product sign refers to the total order
on $\Delta^+\times \BN$ obtained as the restriction of the order~(\ref{extended order}).

The following is due to~\cite{l} (cf.~\cite[Theorem B.3]{ft2}):

\begin{Thm}[\cite{l}]\label{pbwd for yangian}
The elements $\{e_h\}_{h\in H^+}$ form a basis of the free
$\BC[\hbar]$-module $Y^>_\hbar(\ssl_n)$.
\end{Thm}

\begin{Rem}
This result actually holds for any total order on $\Delta^+\times \BN$
used in~(\ref{ordered yangian}), see~\cite{l}.
\end{Rem}

For any $\beta\in \Delta^+$ and $r\in \BN$, define $\wt{e}_\beta(r)\in Y^>_\hbar(\ssl_n)$ via
\begin{equation}\label{integral basis PBW yangian}
   \wt{e}_\beta(r):=\hbar\cdot e_\beta(r).
\end{equation}
For $h\in H^+$, we also define $\wt{e}_h$ via the formula~(\ref{ordered yangian})
but using $\wt{e}_\beta(r)$ instead of $e_\beta(r)$. Finally, define an integral form
$\bY^>_\hbar(\ssl_n)$ as the $\BC[\hbar]$-subalgebra of $Y^>_\hbar(\ssl_n)$
generated by $\{\wt{e}_\beta(r)\}_{\beta\in \Delta^+}^{r\in \BN}$.

The following result is proved in~\cite[Theorem A.7]{ft2}:

\begin{Thm}[\cite{ft2}]\label{pbwd for gavarini yangian}
(a) The subalgebra $\bY^>_\hbar(\ssl_n)$ is independent of all our choices
1)--2) made when defining $e_\beta(r)$ in~(\ref{higher roots yangian}) and hence
$\wt{e}_\beta(r)$ in~(\ref{integral basis PBW yangian}).

\noindent
(b) The ordered PBWD monomials $\{\wt{e}_h\}_{h\in H^+}$ form
a basis of the free $\BC[\hbar]$-module $\bY^>_\hbar(\ssl_n)$.
\end{Thm}


\subsection{Rational shuffle algebra $W^{(n)}$ and its integral form $\fW^{(n)}$}
\label{ssec rational shuffle algebra}
\

Define the shuffle algebra $(\bar{W}^{(n)},\star)$ analogously to
the shuffle algebra $(S^{(n)},\star)$ of Section~\ref{ssec usual shuffle algebra}
with the following modifications:

\begin{itemize}

\item[(1)] 
all rational functions are defined over $\BC[\hbar]$;

\item[(2)] 
the rational functions
  $(\zeta_{i,j}(z))_{i,j\in I} \in \mathrm{Mat}_{I\times I}(\BC[\hbar](z))$
are chosen via
\begin{equation*}
  \zeta_{i,j}(z)=1+\frac{c_{ij}\hbar}{2z};
\end{equation*}

\item[(3)] 
shuffle product $\star$ is defined
via~(\ref{shuffle product}), but $\zeta_{i,i'}(x_{i,r}/x_{i',r'})$ 
replaced with $\zeta_{i,i'}(x_{i,r}-x_{i',r'})$;

\item[(4)] 
the \emph{pole conditions}~(\ref{pole conditions})
for $F\in \bar{W}^{(n)}_{\unl{k}}$ are replaced with
\begin{equation}\label{pole conditions yangian}
  F=
  \frac{f(x_{1,1},\ldots,x_{n-1,k_{n-1}})}
       {\prod_{i=1}^{n-2}\prod_{r\leq k_i}^{r'\leq k_{i+1}}(x_{i,r}-x_{i+1,r'})},\
  \mathrm{where}\ f\in \BC[\hbar][\{x_{i,r}\}_{i\in I}^{1\leq r\leq k_i}]^{\Sigma_{\unl{k}}};
\end{equation}

\item[(5)] 
the \emph{wheel conditions}~(\ref{wheel conditions})
for $F\in \bar{W}^{(n)}_{\unl{k}}$ are replaced with
\begin{equation}\label{wheel conditions yangian}
  F(\{x_{i,r}\})=0\quad \mathrm{once}\quad
  x_{i,r_1}=x_{i+\epsilon,s}+\frac{\hbar}{2}=x_{i,r_2}+\hbar\
\end{equation}
for some $\epsilon\in \{\pm 1\},\, i,\, r_1\ne r_2,\, s$.

\end{itemize}

\noindent
The \textbf{rational shuffle algebra} $\left(\bar{W}^{(n)},\star\right)$
is related to $Y^>_\hbar(\ssl_n)$ via the following construction:

\begin{Prop}\label{simple 1 shuffle yangian}
The assignment $e_{i,r}\mapsto x_{i,1}^r\ (i\in I,r\in \BN)$
gives rise to a $\BC[\hbar]$-algebra homomorphism
$\Psi\colon Y^>_\hbar(\ssl_n)\to \bar{W}^{(n)}$.
\end{Prop}

The next result is straightforward, cf.~Lemma~\ref{shuffle root elt}:

\begin{Lem}\label{shuffle root elt yangian}
For $1\leq j<i<n$ and $r\in \BN$, we have
\begin{equation*}
  \Psi(e_{\alpha_j+\alpha_{j+1}+\ldots+\alpha_i}(r))=
  \hbar^{i-j}\frac{p(x_{j,1},\ldots,x_{i,1})}{(x_{j,1}-x_{j+1,1})\cdots (x_{i-1,1}-x_{i,1})},
\end{equation*}
where $p(x_{j,1},\ldots,x_{i,1})\in \BC[\hbar][x_{j,1},\ldots,x_{i,1}]$
is a degree $r$ monomial, up to a sign.
\end{Lem}

\begin{Ex}
For the particular choice
  $$e_{\alpha_j+\alpha_{j+1}+\ldots+\alpha_i}(r)=[\cdots[e_{j,r},e_{j+1,0}],\cdots,e_{i,0}]$$
of the PBWD basis elements (as used in~\cite[\S2]{ft2}),
we have $p(x_{j,1},\ldots,x_{i,1})=(-1)^{i-j}x_{j,1}^r$.
\end{Ex}

Let us adapt our key tool, the specialization maps, to this setting.
Given a degree vector $\unl{d}=\{d_\beta\}_{\beta\in \Delta^+}\in \BN^{\Delta^+}$,
define $\unl{\ell}\in \BN^I$ via $\sum_{\beta\in \Delta^+} d_\beta \beta = \sum_{i\in I} \ell_i\alpha_i$.
The \textbf{specialization map}
\begin{equation}\label{specialization map yangian}
  \phi_{\unl{d}}\colon \bar{W}^{(n)}_{\unl{\ell}}\longrightarrow
  \BC[\hbar][\{y_{\beta,s}\}_{\beta\in \Delta^+}^{1\leq s\leq d_\beta}]^{\Sigma_{\unl{d}}}
\end{equation}
is defined similar to~(\ref{specialization map}), but with the only change that
the variable $x_{i,r}$ in the $s$-th copy of the interval $[\beta]$ is specialized
to $y_{\beta,s}-\frac{i\hbar}{2}$.

Then, arguing exactly as in Section~\ref{ssec linear indep}, we get:

\begin{Prop}
The elements $\{\Psi(e_h)\}_{h\in H^+}$ are linearly independent.
\end{Prop}

Combining this with Theorem~\ref{pbwd for yangian}, we obtain:

\begin{Prop}\label{simple 2 shuffle yangian}
$\Psi\colon Y^>_\hbar(\ssl_n)\to \bar{W}^{(n)}$ is
an injective $\BC[\hbar]$-algebra homomorphism.
\end{Prop}

However, in contrast to Theorem~\ref{hard shuffle}, the embedding
$\Psi\colon Y^>_\hbar(\ssl_n)\hookrightarrow \bar{W}^{(n)}$
is not an isomorphism. The description of its image is similar
to Theorem~\ref{shuffle integral form}, but is significantly simpler.

\begin{Def}\label{good element yangian}
$F\in \bar{W}^{(n)}_{\unl{k}}$ is \textbf{good} if $\phi_{\unl{d}}(F)$
is divisible by $\hbar^{\sum_{\beta\in \Delta^+} d_\beta(i(\beta)-j(\beta))}$
for any degree vector $\unl{d}=\{d_\beta\}_{\beta\in \Delta^+}\in \BN^{\Delta^+}$
such that $\sum_{\beta\in \Delta^+} d_\beta \beta = \sum_{i\in I}k_i\alpha_i$.
\end{Def}

\begin{Ex}\label{explaining integrality for n=2 yangian}
In the simplest case $n=2$, any element
$F\in \bar{W}^{(n)}_{\unl{k}}\ (\unl{k}\in \BN^I)$ is good.
\end{Ex}

Let $W^{(n)}_{\unl{k}}\subset \bar{W}^{(n)}_{\unl{k}}$ denote the
$\BC[\hbar]$-submodule of all good elements and set 
  $$W^{(n)}:=\underset{\unl{k}\in \BN^I}\bigoplus W^{(n)}_{\underline{k}}$$

\begin{Lem}\label{necessity for Yangian}
$\Psi(Y^>_\hbar(\ssl_n))\subseteq W^{(n)}$.
\end{Lem}

\begin{proof}
Let $F=\Psi(e_{i_1,r_1}\cdots e_{i_N,r_N})\in \bar{W}^{(n)}_{\unl{k}}$ and choose 
$\unl{d}=\{d_\beta\}_{\beta\in \Delta^+}\in \BN^{\Delta^+}$ such that 
$\sum_{\beta\in \Delta^+} d_\beta \beta = \sum_{i\in I} k_i\alpha_i$.
For any $\beta\in \Delta^+$ and $1\leq s\leq d_\beta$, consider $\zeta$-factors
between those pairs of $x_{\ast,\ast}$-variables that are specialized to
$y_{\beta,s}-\frac{\ell\hbar}{2}, y_{\beta,s}-\frac{(\ell+1)\hbar}{2}$ with
$j(\beta)\leq \ell < i(\beta)$ in the definition of the specialization map~(\ref{specialization map yangian}).
Each of them contributes a multiple of $\hbar$ into $\phi_{\unl{d}}(F)$
and there are exactly $\sum_{\beta\in \Delta^+} d_\beta(i(\beta)-j(\beta))$
of such pairs. Hence, $F\in W^{(n)}_{\unl{k}}$.
\end{proof}

The following is the key result of this section:

\begin{Thm}\label{hard shuffle yangian}
The $\BC[\hbar]$-algebra embedding
  $\Psi\colon Y^>_\hbar(\ssl_n)\hookrightarrow \bar{W}^{(n)}$
of Proposition~\ref{simple 2 shuffle yangian} gives rise to a
$\BC[\hbar]$-algebra isomorphism
$\Psi\colon Y^>_\hbar(\ssl_n)\iso W^{(n)}$.
\end{Thm}

In view of Example~\ref{explaining integrality for n=2 yangian},
this theorem for $n=2$ is equivalent to the following result:

\begin{Lem}\label{integral n=2 case yangian}
Any symmetric polynomial $F\in \BC[\hbar][\{x_p\}_{p=1}^k]^{\Sigma_k}$
may be written as a $\BC[\hbar]$-linear combination of $\{\Psi(e_h)\}_{h\in H^+}$.
\end{Lem}

The proof of Lemma~\ref{integral n=2 case yangian} is completely analogous
to those of Lemmas~\ref{n=1 case},~\ref{integral n=1 case}, and crucially relies
on the following simple computation (cf.~Lemma~\ref{k-th fold product}):

\begin{Lem}\label{k-th fold product yangian}
For any $k\geq 1$ and $r\in \BN$, the $k$-th power
of $x^r\in \bar{W}^{(2)}_1$ equals
\begin{equation}\label{factorial formula yangian}
  \underbrace{x^r\star\cdots \star x^r}_{k\ \mathrm{times}}=k!\cdot (x_1\cdots x_k)^r.
\end{equation}
\end{Lem}

\begin{proof}
The proof is by induction in $k$ and boils down to the verification of the equality
\begin{equation}\label{combinatorial identity yangian}
   \sum_{p=1}^k \prod_{1\leq s\leq k}^{s\ne p} \frac{x_s-x_p+\hbar}{x_s-x_p}=k,
\end{equation}
which is proved similarly to~(\ref{combinatorial identity}).
\end{proof}

The proof of Theorem~\ref{hard shuffle yangian} for $n>2$ is completely
analogous to those of Theorems~\ref{Main Theorem 1},~\ref{Main Theorem 2}
and crucially utilizes its $n=2$ case, established in Lemma~\ref{integral n=2 case yangian}.
We refer the interested reader to~\cite[Proof of Theorem 3.30]{ts}
for more details.

\begin{Def}\label{integral element yangian}
$F\in \bar{W}^{(n)}_{\unl{k}}$ is \textbf{integral} if $F$
is divisible by $\hbar^{|\unl{k}|}$.
\end{Def}

\begin{Rem}
We note that any integral shuffle element is obviously good,
cf.~Definition~\ref{good element yangian}.
\end{Rem}

Let $\fW^{(n)}_{\unl{k}}\subset W^{(n)}_{\unl{k}}$ denote the $\BC[\hbar]$-submodule of 
all integral elements and set 
  $$\fW^{(n)}:=\underset{\unl{k}\in \BN^I}\bigoplus \fW^{(n)}_{\underline{k}}$$
The following is our second key result of this section:

\begin{Thm}\label{shuffle integral form yangian}
The $\BC[\hbar]$-algebra isomorphism $\Psi\colon Y^>_\hbar(\ssl_n)\iso W^{(n)}$
of Theorem~\ref{hard shuffle yangian} gives rise to a $\BC[\hbar]$-algebra
isomorphism $\Psi\colon \bY^>_\hbar(\ssl_n)\iso \fW^{(n)}$.
\end{Thm}

The proof of Theorem~\ref{shuffle integral form yangian} is completely analogous
to that of Theorem~\ref{shuffle integral form}, but is much simpler.
In particular, adapting Lemma~\ref{necessity for n>2} to the current setting,
the key combinatorial computation from its proof is not needed, while
Lemma~\ref{spanning integral} is adapted without any changes.
We refer the interested reader to~\cite[Proof of Theorem 3.9]{ts}
for more details.

\begin{Rem}
Let us note right away that the key simplification in the proof of
Theorem~\ref{shuffle integral form yangian} (comparing to that of
Theorem~\ref{shuffle integral form}) and in the definition of
the integral elements of Definition~\ref{integral element yangian}
(comparing to those of Definition~\ref{integral element})
is due to the following rank $1$ computations:
\begin{enumerate}
 \item[(1)] $\hbar^k(x_1\cdots x_k)^r\in \Psi(\bY^>_\hbar(\ssl_2))$
  for any $k,r\in \BN$, due to Lemma~\ref{k-th fold product yangian};

 \item[(2)] $(\vv-\vv^{-1})^k[k]_\vv!(x_1\cdots x_k)^r\in \Psi(\fU^>_\vv(L\ssl_2))$
  for any $k\in \BN,r\in \BZ$, due to Lemma~\ref{k-th fold product};

 \item[(3)] $(\vv-\vv^{-1})^k(x_1\cdots x_k)^r\notin \Psi(\fU^>_\vv(L\ssl_2))$
  for any $k>1,r\in \BZ$, due to Lemma~\ref{necessity for n=2}.
\end{enumerate}
\end{Rem}


\section{Generalizations to the super Yangian $Y_\hbar(\ssl(m|n))$}
\label{sec superyangian counterparts}

The super Yangian $Y_\hbar(\gl(m|n))$ of the Lie superalgebra $\gl(m|n)$
was first introduced in~\cite{na}, following the RTT formalism of~\cite{frt}.
Its finite-dimensional representations were classified in~\cite{z4}.
Around the same time, the super Yangian $Y_\hbar(A(m,n))$ of the classical
Lie superalgebra $A(m,n)$ in the new Drinfeld presentation was introduced in~\cite{s}.
As shown in \emph{loc.~cit.}, these super Yangians posses most of the properties
the usual Yangians have (including the PBWD bases). The explicit relation
between the super Yangians of~\cite{na,s} was established in~\cite{go}.

The primary goal of this section is to generalize
Theorem~\ref{hard shuffle yangian} to the case of $Y^>_\hbar(\ssl(m|n))$
(we note that
  $Y^>_\hbar(\ssl(m|n))\simeq Y^>_\hbar(\gl(m|n))\simeq Y^>_\hbar(A(m-1,n-1))$).
The resulting shuffle algebra $W^{(m|n)}$ is a mixture of the shuffle algebra
$S^{(m|n)}$ from Section~\ref{ssec trigonometric super shuffle} and
the rational shuffle algebra $W^{(n)}$ from Section~\ref{ssec rational shuffle algebra}.
We also generalize Theorem~\ref{shuffle integral form yangian} to the case of
$\bY^>_\hbar(\ssl(m|n))$.


\subsection{Algebras $Y^>_\hbar(\ssl(m|n))$ and $\bY^>_\hbar(\ssl(m|n))$}
\label{ssec super half-yangian sl_n}
\

Let $I=\{1,\ldots,m+n-1\}$, $(\bar{c}_{ij})_{i,j\in I}$ be defined as in
Section~\ref{ssec quantum super affine}, and $\hbar$ be a formal variable.
Following~\cite{s,go} (see Remark~\ref{correcting Stukopin} for a
correction of the defining relations in~\cite[Definition~2]{s}),
define $Y^>_\hbar(\ssl(m|n))$ to be the associative $\BC[\hbar]$-superalgebra
generated by $\{e_{i,r}\}_{i\in I}^{r\in \BN}$ with the $\BZ_2$-grading
  $|e_{m,r}|=\bar{1}, |e_{i,r}|=\bar{0}\ (i\ne m, r\in \BN)$,
and subject to the following defining relations:
\begin{equation}\label{superYan 1}
  [e_{i,r+1}, e_{j,s}]-[e_{i,r},e_{j,s+1}]=
  \frac{\bar{c}_{ij}\hbar}{2}\left(e_{i,r}e_{j,s}+e_{j,s}e_{i,r}\right)
  \ \mathrm{if}\ \bar{c}_{ij}\ne 0,
\end{equation}
\begin{equation}\label{superYan 2}
\begin{split}
  & [e_{i,r},e_{j,s}]=0\  \mathrm{if}\ \bar{c}_{ij}=0,\\
  & [e_{i,r_1},[e_{i,r_2},e_{j,s}]]+
    [e_{i,r_2},[e_{i,r_1},e_{j,s}]]=0\ \mathrm{if}\ \bar{c}_{ij}=\pm 1\ \mathrm{and}\ i\ne m,
\end{split}
\end{equation}
as well as quartic Serre relations:
\begin{equation}\label{superYan 3}
  [[e_{m-1,s},e_{m,r_1}],[e_{m+1,s'},e_{m,r_2}]]+
  [[e_{m-1,s},e_{m,r_2}],[e_{m+1,s'},e_{m,r_1}]]=0,
\end{equation}
where, as before, we use the super-bracket  $[a,b]=ab-(-1)^{|a||b|}\cdot ba$ for $\BZ_2$-homogeneous $a,b$.

\begin{Rem}\label{correcting Stukopin}
(a) The first relation of~(\ref{superYan 2}) implies the validity
of the second one for $i=m=j\pm 1$, which is also listed among
the defining relations of~\cite[Definition 2]{s}.

\noindent
(b) The wrong defining relation $[[e_{m-1,s},e_{m,r_1}],[e_{m+1,s'},e_{m,r_2}]]=0$
of~\cite[Definition~2]{s} should be replaced with the above relation~(\ref{superYan 3}).
\end{Rem}

Let $\{\alpha_i\}_{i=1}^{m+n-1}$ be as in Section~\ref{ssec quantum super affine}
and $\Delta^+$ be as in Section~\ref{ssec formulation main thm 4}.
For $\beta\in \Delta^+$, we define its parity $|\beta|\in \BZ_{2}$
via~(\ref{root parity}). We define the \emph{PBWD basis elements}
$e_\beta(r)\in Y^>_\hbar(\ssl(m|n))$ via~(\ref{higher roots yangian}), 
but with $[\cdot,\cdot]$ denoting the super-bracket.

Let $\bar{H}^+$ denote the set of all functions
$h\colon \Delta^+\times \BN\to \BN$ with finite support
and such that $h(\beta,r)\leq 1$ if $|\beta|=\bar{1}$.
The monomials
\begin{equation}\label{ordered superyangian}
  e_h\ :=\prod\limits_{(\beta,r)\in \Delta^+\times \BN}^{\rightarrow} e_\beta(r)^{h(\beta,r)}\ ,
  \qquad \forall\, h\in \bar{H}^+
\end{equation}
will be called the \emph{ordered PBWD monomials} of $Y^>_\hbar(\ssl(m|n))$.
Analogously to~\cite{l}, we have:

\begin{Thm}[\cite{s}]\label{pbwd for superyangian}
The elements $\{e_h\}_{h\in \bar{H}^+}$ form a basis of the free
$\BC[\hbar]$-module $Y^>_\hbar(\ssl(m|n))$.
\end{Thm}

For any $(\beta,r)\in \Delta^+\times \BZ$, define $\wt{e}_\beta(r)\in Y^>_\hbar(\ssl(m|n))$ via
  $$\wt{e}_\beta(r):=\hbar\cdot e_\beta(r).$$
For $h\in \bar{H}^+$, we also define $\wt{e}_h$ via~(\ref{ordered superyangian})
but using $\wt{e}_\beta(r)$ instead of $e_\beta(r)$.
Finally, let $\bY^>_\hbar(\ssl(m|n))$ denote the $\BC[\hbar]$-subalgebra
of $Y^>_\hbar(\ssl(m|n))$ generated by
$\{\wt{e}_\beta(r)\}_{\beta\in \Delta^+}^{r\in \BN}$.

The following result is analogous to~\cite[Theorem A.7]{ft2}:

\begin{Thm}\label{pbwd for gavarini superyangian}
(a) The subalgebra $\bY^>_\hbar(\ssl(m|n))$ is independent
of all our choices.

\noindent
(b) The ordered PBWD monomials $\{\wt{e}_h\}_{h\in \bar{H}^+}$
form a basis of the free $\BC[\hbar]$-module $\bY^>_\hbar(\ssl(m|n))$.
\end{Thm}


\subsection{Rational shuffle algebra $W^{(m|n)}$ and its integral form $\fW^{(m|n)}$}
\label{ssec rational shuffle super algebra}
\

Define the shuffle algebra $(\bar{W}^{(m|n)},\star)$ analogously
to the shuffle algebra $(S^{(m|n)},\star)$ of
Section~\ref{ssec trigonometric super shuffle}
with the following modifications:

\begin{itemize}

\item[(1)] 
all rational functions are defined over $\BC[\hbar]$;

\item[(2)] 
the rational functions
  $(\zeta_{i,j}(z))_{i,j\in I} \in \mathrm{Mat}_{I\times I}(\BC[\hbar](z))$
are chosen via
\begin{equation*}
  \zeta_{i,j}(z)=1+\frac{\bar{c}_{ij}\hbar}{2z};
\end{equation*}

\item[(3)] 
the shuffle product $\star$ is defined
via~(\ref{shuffle product}), but with the supersymmetrization $\SSym$
in place of the symmetrization $\Sym$ and $\zeta_{i,i'}(x_{i,r}-x_{i',r'})$ in place
of $\zeta_{i,i'}(x_{i,r}/x_{i',r'})$;

\item[(4)] 
the \emph{pole conditions}~(\ref{pole conditions supercase})
for $F\in \bar{W}^{(m|n)}_{\unl{k}}$ are replaced with
\begin{equation}\label{pole conditions super yangian}
  F=
  \frac{f(x_{1,1},\ldots,x_{n-1,k_{n-1}})}
       {\prod_{i=1}^{n-2}\prod_{r\leq k_i}^{r'\leq k_{i+1}}(x_{i,r}-x_{i+1,r'})},
\end{equation}
where $f\in \BC[\hbar][\{x_{i,r}\}_{i\in I}^{1\leq r\leq k_i}]$ is a
supersymmetric polynomial, that is, symmetric in $\{x_{i,r}\}_{r=1}^{k_i}$
for every $i\ne m$ and skew-symmetric in $\{x_{m,r}\}_{r=1}^{k_m}$;

\item[(5)] 
the \emph{first kind wheel conditions}~(\ref{wheel conditions 1 supercase})
for $F\in \bar{W}^{(m|n)}_{\unl{k}}$ are replaced with
\begin{equation}\label{wheel condition 1 super yangian}
  F(\{x_{i,r}\})=0\quad \mathrm{once}\quad
  x_{i,r_1}=x_{i+\epsilon,s}+\hbar/2=x_{i,r_2}+\hbar\
\end{equation}
for some $\epsilon\in \{\pm 1\},\, i\ne m,\, r_1\ne r_2,\, s$;

\item[(6)] 
the \emph{second kind wheel conditions}~(\ref{wheel conditions 2 supercase})
for $F\in \bar{W}^{(m|n)}_{\unl{k}}$ are replaced with
\begin{equation}\label{wheel condition 2 super yangian}
  F(\{x_{i,r}\})=0\quad \mathrm{once}\quad
  x_{m-1,s}=x_{m,r_1}+\hbar/2=x_{m+1,s'}=x_{m,r_2}-\hbar/2
\end{equation}
for some $r_1\ne r_2,\, s,\, s'$.

\end{itemize}

\noindent
In view of Theorem~\ref{pbwd for superyangian}, the 
\textbf{rational shuffle algebra} $\left(\bar{W}^{(m|n)},\star\right)$ is related
to $Y^>_\hbar(\ssl(m|n))$ via the following construction
(cf.~Propositions~\ref{simple 1 shuffle yangian},~\ref{simple 2 shuffle yangian}):

\begin{Prop}\label{simple shuffle super yangian}
The assignment $e_{i,r}\mapsto x_{i,1}^r\ (i\in I,r\in \BN)$
gives rise to a $\BC[\hbar]$-algebra embedding
$\Psi\colon Y^>_\hbar(\ssl(m|n))\hookrightarrow \bar{W}^{(m|n)}$.
\end{Prop}

To describe the images of $Y^>_\hbar(\ssl(m|n))$ and its subalgebra
$\bY^>_\hbar(\ssl(m|n))$ under the embedding $\Psi$, let us introduce
the specialization maps in the current setting.
Given a degree vector $\unl{d}=\{d_\beta\}_{\beta\in \Delta^+}\in \BN^{\Delta^+}$,
define $\unl{\ell}\in \BN^I$ via
  $\sum_{\beta\in \Delta^+} d_\beta \beta = \sum_{i\in I} \ell_i\alpha_i$.
The \textbf{specialization map}
\begin{equation*}
  \phi_{\unl{d}}\colon \bar{W}^{(m|n)}_{\unl{\ell}}\longrightarrow
  \BC[\hbar][\{y_{\beta,s}\}_{\beta\in \Delta^+}^{1\leq s\leq d_\beta}]
\end{equation*}
   is defined similar to~(\ref{specialization map yangian}),
but with the only change that the variable $x_{i,r}$ in the $s$-th copy
of the interval $[\beta]$ is specialized to $y_{\beta,s}-\frac{i\hbar}{2}$
if $i\leq m$ and to $y_{\beta,s}+\frac{(i-2m)\hbar}{2}$ if $i>m$.

\begin{Def}\label{good element super yangian}
(a) $F\in \bar{W}^{(m|n)}_{\unl{k}}$ is \textbf{good} if $\phi_{\unl{d}}(F)$
is divisible by $\hbar^{\sum_{\beta\in \Delta^+} d_\beta(i(\beta)-j(\beta))}$
for any degree vector $\unl{d}=\{d_\beta\}_{\beta\in \Delta^+}\in \BN^{\Delta^+}$
such that $\sum_{\beta\in \Delta^+} d_\beta \beta = \sum_{i\in I} k_i \alpha_i$.

\noindent
(b) $F\in \bar{W}^{(m|n)}_{\unl{k}}$ is \textbf{integral}
if $F$ is divisible by $\hbar^{|\unl{k}|}$.
\end{Def}

\begin{Rem}
We note that any integral shuffle element $F\in \bar{W}^{(m|n)}_{\unl{k}}$ is obviously good.
\end{Rem}

Let $W^{(m|n)}_{\unl{k}}\subset \bar{W}^{(m|n)}_{\unl{k}}$ and $\fW^{(m|n)}_{\unl{k}}\subset \bar{W}^{(m|n)}_{\unl{k}}$
denote the $\BC[\hbar]$-submodules of all good and integral elements, respectively, and set
  $$W^{(m|n)}:=\underset{\unl{k}\in \BN^I}\bigoplus W^{(m|n)}_{\underline{k}}\ , 
    \qquad 
    \fW^{(m|n)}:=\underset{\unl{k}\in \BN^I}\bigoplus \fW^{(m|n)}_{\underline{k}}$$
The following are the key results of this section:

\begin{Thm}\label{hard shuffle super yangian}
The $\BC[\hbar]$-algebra embedding
  $\Psi\colon Y^>_\hbar(\ssl(m|n))\hookrightarrow \bar{W}^{(m|n)}$
of Proposition~\ref{simple shuffle super yangian} gives rise to
a $\BC[\hbar]$-algebra isomorphism
$\Psi\colon Y^>_\hbar(\ssl(m|n))\iso W^{(m|n)}$.
\end{Thm}

\begin{Thm}\label{shuffle integral form super yangian}
The $\BC[\hbar]$-algebra isomorphism
$\Psi\colon Y^>_\hbar(\ssl(m|n))\iso W^{(m|n)}$
of Theorem~\ref{hard shuffle super yangian} gives rise to
a $\BC[\hbar]$-algebra isomorphism
$\Psi\colon \bY^>_\hbar(\ssl(m|n))\iso \fW^{(m|n)}$.
\end{Thm}

Both Theorems~\ref{hard shuffle super yangian},~\ref{shuffle integral form super yangian}
are proved completely analogously to
Theorems~\ref{hard shuffle superLie},~\ref{hard shuffle yangian},~\ref{shuffle integral form yangian}.
We refer the interested reader to~\cite[Proofs of Theorems 3.9, 3.30]{ts} for more details.


\section{Further directions}\label{sec further directions}

In this section, we briefly outline some of the related results
that will be addressed elsewhere.


\subsection{Integral forms of Grojnowski and Chari-Pressley and their PBWD bases}
\label{ssec grojnowski integral form}
\

We shall follow the notations of Section~\ref{sec Classical quantuma affine}.
For $i\in I, r\in \BZ, k\in \BN$, define the divided power
\begin{equation}\label{devided power}
  \se_{i,r}^{(k)}:=\frac{e_{i,r}^k}{[k]_\vv!}
\end{equation}
Following~\cite[\S7.8]{gr}, define the integral form $\sU^>_\vv(L\ssl_n)$
as the $\BC[\vv,\vv^{-1}]$-subalgebra of $U^>_\vv(L\ssl_n)$ generated by all
the divided powers $\{\se_{i,r}^{(k)}\}_{i\in I}^{r\in \BZ,k\in \BN}$.
The main objective of this section is to construct a family of PBWD bases
for $\sU^>_\vv(L\ssl_n)$ as well as to provide its shuffle realization.

Recall the PBWD basis elements $\{e_{\beta}(r)\}_{\beta\in \Delta^+}^{r\in \BZ}$
of~(\ref{higher roots}), which do depend on all the choices 1)--3) made prior
to their definition.
For $\beta\in \Delta^+, r\in \BZ, k\in \BN$, we define the divided power
\begin{equation}\label{higher root divided power}
  \se_{\beta}(r)^{(k)}:=\frac{e_{\beta}(r)^k}{[k]_\vv!}
\end{equation}
As a simple corollary of~\cite[Proof of Theorem 6.6]{lus}, we obtain:

\begin{Lem}
$\se_{\beta}(r)^{(k)}\in \sU^>_\vv(L\ssl_n)$ for any
$\beta\in \Delta^+, r\in \BZ, k\in \BN$.
\end{Lem}

\noindent
The monomials
  $$\se_h\ :=\prod\limits_{(\beta,r)\in \Delta^+\times \BZ}^{\rightarrow} \se_\beta(r)^{(h(\beta,r))}\ ,
    \quad \forall\, h\in H$$
will be called the \emph{ordered PBWD monomials} of $\sU^{>}_\vv(L\ssl_n)$.
Here, the arrow over the product sign refers to the total order~(\ref{extended order}).

Our first main result of this Section establishes the PBWD property for $\sU^{>}_\vv(L\ssl_n)$:

\begin{Thm}\label{Main Theorem Grojnowski}
The elements $\{\se_h\}_{h\in H}$ form a basis of
the free $\BC[\vv,\vv^{-1}]$-module $\sU^{>}_\vv(L\ssl_n)$.
\end{Thm}

The proof of Theorem~\ref{Main Theorem Grojnowski} is completely analogous
to our proofs of Theorems~\ref{Main Theorem 1},~\ref{Main Theorem 2} and
is based on (as well as used in) the description of $\Psi(\sU^>_\vv(L\ssl_n))$,
viewed as a subspace of $S^{(n)}$. For the latter purpose,
let us adapt Definition~\ref{good element yangian} to the current setting:

\begin{Def}\label{good element quantum}
$F\in S^{(n)}_{\unl{k}}$ is \textbf{good} if it satisfies the following two properties:
\begin{enumerate}
  \item[(i)]
    $F=
    \frac{f(x_{1,1},\ldots,x_{n-1,k_{n-1}})}
         {\prod_{i=1}^{n-2}\prod_{r\leq k_i}^{r'\leq k_{i+1}}(x_{i,r}-x_{i+1,r'})}$
    with $f\in \BC[\vv,\vv^{-1}][\{x_{i,r}^{\pm 1}\}_{i\in I}^{1\leq r\leq k_i}]^{\Sigma_{\unl{k}}}$;

  \item[(ii)]
    the specialization $\phi_{\unl{d}}(F)$ of~(\ref{specialization map})
    is divisible by $(\vv-\vv^{-1})^{\sum_{\beta\in \Delta^+} d_\beta(i(\beta)-j(\beta))}$
    for any degree vector $\unl{d}=\{d_\beta\}_{\beta\in \Delta^+}\in \BN^{\Delta^+}$
    such that $\sum_{\beta\in \Delta^+} d_\beta \beta = \sum_{i\in I} k_i \alpha_i$.
\end{enumerate}
\end{Def}

Let $\sS^{(n)}_{\unl{k}}\subset S^{(n)}_{\unl{k}}$ denote the $\BC[\vv,\vv^{-1}]$-submodule of all good elements and set 
  $$\sS^{(n)}:=\underset{\unl{k}\in \BN^I}\bigoplus \sS^{(n)}_{\underline{k}}$$

\begin{Lem}\label{necessity for Grojnowski}
$\Psi(\sU^>_\vv(L\ssl_n))\subseteq \sS^{(n)}$.
\end{Lem}

\begin{proof}
As $\sU^>_\vv(L\ssl_n)$ is generated by $\se_{i,r}^{(k)}$ over $\BC[\vv,\vv^{-1}]$,
it suffices to verify both properties~(i,~ii) of Definition~\ref{good element quantum}
for any shuffle element $F=\Psi(\se_{i_1,r_1}^{(k_1)}\cdots \se_{i_N,r_N}^{(k_N)})$.
The validity of~(i) for $F$ follows from the equality 
  $\Psi(\se_{i,r}^{(k)})=\vv^{-\frac{k(k-1)}{2}}(x_1\cdots x_k)^r$,
due to Lemma~\ref{k-th fold product}.
On the other hand, the validity of~(ii) for $F$ is established using the arguments from
our proof of Lemma~\ref{necessity for Yangian}.
\end{proof}

The second key result of this section provides
a shuffle realization of $\sU_\vv^{>}(L\ssl_n)$:

\begin{Thm}\label{shuffle Grojnowski}
The $\BC(\vv)$-algebra isomorphism $\Psi\colon U_\vv^{>}(L\ssl_n)\iso S^{(n)}$
of Theorem~\ref{hard shuffle} gives rise to a $\BC[\vv,\vv^{-1}]$-algebra
isomorphism $\Psi\colon \sU_\vv^{>}(L\ssl_n)\iso \sS^{(n)}$.
\end{Thm}

Define $\sU^<_\vv(L\ssl_n)$ as the $\BC[\vv,\vv^{-1}]$-subalgebra
of $U^<_\vv(L\ssl_n)$ generated by the divided powers
$\mathsf{f}_{i,r}^{(k)}:=f_{i,r}^k/[k]_\vv!\ (i\in I,r\in \BZ, k\in \BN)$.
Finally, consider the $\BC[\vv,\vv^{-1}]$-subalgebra $\sU^0_\vv(L\ssl_n)$
of $U^0_\vv(L\ssl_n)$ introduced in~\cite[\S3]{cp}, cf.~\cite[\S2.3]{t2}.
Following~\cite{cp}, define the integral form $\sU_\vv(L\ssl_n)$ as
the $\BC[\vv,\vv^{-1}]$-subalgebra of $U_\vv(L\ssl_n)$ generated by
$\sU^<_\vv(L\ssl_n),\sU^0_\vv(L\ssl_n),\sU^>_\vv(L\ssl_n)$.

\begin{Rem}\label{Lusztig vs Grojnowski}
Identifying $U_\vv(L\ssl_n)$ with the Drinfeld-Jimbo quantum loop
algebra $U^{\mathrm{DJ}}_\vv(L\ssl_n)$, see~\cite{d}, the form
$\sU_\vv(L\ssl_n)$ is identified with the Lusztig form
of $U^{\mathrm{DJ}}_\vv(L\ssl_n)$, due to~\cite{cp}.
\end{Rem}

The following triangular decomposition of $\sU_\vv(L\ssl_n)$ is
due to~\cite[Proposition 6.1]{cp}:

\begin{Thm}[\cite{cp}]\label{triangular CP}
The multiplication map
\begin{equation*}
  m\colon
   \sU^{<}_\vv(L\ssl_n)\otimes_{\BC[\vv,\vv^{-1}]}
   \sU^{0}_\vv(L\ssl_n)\otimes_{\BC[\vv,\vv^{-1}]}
   \sU^{>}_\vv(L\ssl_n)\longrightarrow \sU_\vv(L\ssl_n)
\end{equation*}
is an isomorphism of the free $\BC[\vv,\vv^{-1}]$-modules.
\end{Thm}

Combining Theorems~\ref{Main Theorem Grojnowski} and~\ref{triangular CP},
we obtain a family of PBWD bases for the form $\sU_\vv(L\ssl_n)$.

\begin{Rem}
The results of
Theorems~\ref{Main Theorem 2},~\ref{shuffle integral form},~\ref{Main Theorem Grojnowski},~\ref{shuffle Grojnowski}
were recently used in~\cite{t2} to establish the duality between
the forms $\fU^>_\vv(L\ssl_n)$ and $\sU^<_\vv(L\ssl_n)$
(resp.\ $\fU^<_\vv(L\ssl_n)$ and $\sU^>_\vv(L\ssl_n)$)
with respect to the new Drinfeld pairing.
We refer the interested reader to~\cite{t2} for more details.
\end{Rem}


\subsection{Generalizations to all Dynkin diagrams associated with $\ssl(m|n)$}
\label{ssec all Dynkin diagrams}
\

Recall that a novel feature of Lie superalgebras (in contrast to Lie algebras)
is that they admit several non-isomorphic Dynkin diagrams. Likewise, one may
consider various quantizations of universal enveloping superalgebras starting
from different Dynkin diagrams. The explicit isomorphism of such algebras
associated to various Dynkin diagrams of the same type is highly non-trivial:
it has been established for quantum finite/affine superalgebras in~\cite{y},
but seems to be an open question for general super Yangians.
Furthermore, the \emph{positive subalgebras} (those generated by $\{e_{i,r}\}$)
do essentially depend on the choice of a Dynkin diagram.

In the recent paper~\cite{ts}, we address the above question for the
Lie superalgebra $A(m,n)$ as well as generalize the results of
Sections~\ref{sec super-Lie quantuma affine},~\ref{sec superyangian counterparts}
to all of its Dynkin diagrams. Explicitly, given a superspace
$V=V_{\bar{0}}\oplus V_{\bar{1}}$ with a basis $\sfv_1,\ldots,\sfv_n$ such that each $\sfv_i$ is
either \emph{even} ($\sfv_i\in V_{\bar{0}}$) or \emph{odd} ($\sfv_i\in V_{\bar{1}}$),
one may define the quantum loop superalgebras $U_\vv(L\gl(V)), U_\vv(L\ssl(V))$
as well as the super Yangians $Y_\hbar(\gl(V)), Y_\hbar(\ssl(V))$, both
in the RTT presentation of~\cite{frt} and the new Drinfeld presentation of~\cite{d}
(their equivalence is established following the ideas of~\cite{df}). The corresponding
\emph{positive subalgebras} $U^>_\vv(L\ssl(V))\simeq U^>_\vv(L\gl(V))$
(resp.\ $Y^>_\hbar(\ssl(V))\simeq Y^>_\hbar(\gl(V))$) are generated by
$\{e_{i,r}\}_{i\in I}^{r\in \BZ}$ (resp.\ $\{e_{i,r}\}_{i\in I}^{r\in \BN}$)
subject to the defining relations~\cite[(4.2--4.5)]{ts}
(resp.~\cite[(2.58--2.60, 2.78)]{ts}) and with the $\BZ_2$-grading
$|e_{i,r}|=|\alpha_i|$, where
\begin{equation*}
  |\alpha_i|=
  \begin{cases}
    \bar{0}, & \text{if } \sfv_i\ \text{and } \sfv_{i+1}\ \text{have\ the same parity} \\
    \bar{1}, & \text{otherwise }
  \end{cases}.
\end{equation*}

The construction of the PBWD bases for $U^>_\vv(L\ssl(V)),Y^>_\hbar(\ssl(V))$
and their integral forms $\fU^>_\vv(L\ssl(V)),\bY^>_\hbar(\ssl(V))$ is similar
to Theorems~\ref{Main Theorem 4},~\ref{Main Theorem 4.1},~\ref{pbwd for superyangian},~\ref{pbwd for gavarini superyangian}.
The corresponding \emph{ordered PBWD monomials} are defined analogously
to~(\ref{ordered supercase},~\ref{ordered superyangian}) with the indexing
sets $\bar{H},\bar{H}^+$ defined as before, but using a different $\BZ_2$-grading
on $\Delta^+$:
  $$|\alpha_j+\alpha_{j+1}+\ldots+\alpha_i|=|\alpha_j|+|\alpha_{j+1}|+\ldots+|\alpha_i|.$$
The associated shuffle algebras $S^{(V)},W^{(V)}$ and their integral forms
$\fS^{(V)},\fW^{(V)}$ are defined similar to
$S^{(m|n)},W^{(m|n)},\fS^{(m|n)},\fW^{(m|n)}$. Their elements
are supersymmetric rational functions in $\{x_{\ast,\ast}\}$, that is,
symmetric in $\{x_{i,\ast}\}$ if $|\alpha_i|=\bar{0}$ and skew-symmetric
in $\{x_{i,\ast}\}$ if $|\alpha_i|=\bar{1}$.


\end{document}